\pdfoutput=1
\documentclass[3p]{elsarticle}
\usepackage{epsfig,balance}
\usepackage{amssymb}
\usepackage{amsmath}
\usepackage{amsfonts}
\usepackage{hhline}
\usepackage{array}
\usepackage{verbatim}
\usepackage{color}
\usepackage{balance}
\usepackage{paralist}
\usepackage{amsthm}
\newtheoremstyle{note}%
  {3pt}
  {0pt}
  {\itshape}
  {}
  {\bfseries}
  {:}
  {.5em}
  {}
\theoremstyle{note}
\newtheorem{theorem}{\textbf{Theorem}}

\newtheorem{example}{\textbf{Example}}
\newtheorem{lemma}{\textbf{Lemma}}
\newtheorem{corollary}{\textbf{Corollary}}
\newtheorem{proposition}{\textbf{Proposition}}
\newtheorem{remark}{Remark}
\newcommand{\transpose}{{\!\scriptscriptstyle\mathrm T}}
\newcommand{\norm}[1]{\lVert#1\rVert}
\newcommand{\Rbb}{\mathbb{R}}

\renewcommand{\l}{\ell}
\newcommand\ip[2]{\langle #1, #2\rangle}
\renewcommand{\L}{{\mathcal{L}}}
\newcommand{\G}{{\mathcal{G}}}
\newcommand{\E}{{\mathcal{E}}}

\newcommand{\V}{{\mathcal{V}}}

\renewcommand{\S}{{\mathcal{S}}}

\begin{document}
\title{Vertex-Frequency Analysis on Graphs} 
\author
{David I Shuman} 
\ead{david.shuman@epfl.ch}
\author
{Benjamin Ricaud} 
\ead{benjamin.ricaud@epfl.ch} 
\author
{Pierre Vandergheynst\fnref{fn1,fn2}} 
\ead{pierre.vandergheynst@epfl.ch}
\fntext[fn1]{This work was supported by FET-Open grant number 255931 UNLocX.} 
\fntext[fn2]{Part of the work reported here was presented at the \emph{IEEE Statistical Signal Processing Workshop, August 2011, Ann Arbor, MI}.} 
\address
{Signal Processing Laboratory (LTS2), Ecole Polytechnique F{\'e}d{\'e}rale de Lausanne (EPFL), Lausanne, Switzerland} 
\begin{abstract}
One of the key challenges in the area of signal processing on graphs is to design dictionaries and transform methods to 
identify and exploit structure in signals on weighted graphs. To do so, we need to account for the intrinsic geometric structure of the underlying graph data domain. In this paper, we generalize one of the most important signal processing tools - windowed Fourier analysis - to the graph setting. Our approach is to first
define generalized convolution, translation, and modulation operators for signals on graphs, and explore related properties such as the localization of translated and modulated graph kernels. We then use these operators to
define a 
 windowed graph Fourier transform, 
 enabling vertex-frequency analysis. When we apply this transform to a signal with 
frequency components that vary along a path graph, the resulting spectrogram matches our intuition from classical discrete-time signal processing.
Yet, our construction is fully generalized and can be applied to analyze signals on any undirected, connected, weighted graph.
\end{abstract}
\begin{keyword}
Signal processing on graphs; time-frequency analysis; generalized translation and modulation; spectral graph theory; localization; clustering
\end{keyword}
\maketitle

\section{Introduction}
\label{sec:intro}
In 
applications such as social networks, electricity networks, transportation networks, and sensor networks, data naturally reside on the vertices of weighted graphs.  
Moreover, weighted graphs are a flexible tool that can be used to describe similarities between data points in statistical learning problems, functional connectivities between different regions of the brain, and the geometric structures of countless other topologically-complex data domains. 

In order to  
reveal relevant structural properties of such data on graphs
and/or  
sparsely represent different classes of signals on graphs, we can construct 
dictionaries of atoms, and represent graph signals as linear combinations of the dictionary atoms. The design of such dictionaries is one of the fundamental problems of signal processing, and the literature is filled with a wide range of dictionaries, including, e.g.,  Fourier, time-frequency, curvelet, shearlet, and bandlet dictionaries 
(see, e.g., \cite{rubinstein_dict_learning} for an excellent historical overview of dictionary design methods and signal transforms).

Of course, the dictionary needs to be tailored to a given class of signals under consideration. Specifically,
as exemplified in \cite[Example 1]{shuman_SPM}, in order to identify and exploit structure in signals on weighted graphs, we need to account for the intrinsic geometric structure of the underlying data domain when designing dictionaries and signal transforms. When we construct dictionaries of features on weighted graphs, it is also desirable to (i) explicitly control how these features change from vertex to vertex, and (ii) ensure that we treat vertices in a homogeneous way (i.e., the resulting dictionaries are invariant to permutations in the vertex labeling). Unfortunately, weighted graphs are irregular structures that lack a shift-invariant notion of translation, a key component in many 
signal processing techniques 
for data on regular Euclidean spaces. Thus, many of the existing 
dictionary design techniques cannot be directly applied to signals on graphs in a meaningful manner, and an important challenge is to design  
new localized transform methods that account for the structure of the data domain. 

Accordingly, a number of new multiscale wavelet transforms for signals on graphs have been introduced recently (see
\cite{shuman_SPM} and references therein for a review of wavelet transforms for signals on graphs). 
Although the field of signal processing on graphs is still young, the hope is that such transforms can be used 
to efficiently extract information from high-dimensional data on graphs (either statistically or visually), as well as to regularize ill-posed inverse problems. 

Windowed Fourier transforms, 
also called short-time Fourier transforms, are another important class of time-frequency analysis tools in classical signal processing. They 
are particularly useful in extracting information from signals with oscillations that are localized in time or space. Such signals appear frequently in applications such as audio and speech processing, vibration analysis, and radar detection. Our aim here is to generalize windowed Fourier analysis to the graph setting. 

Underlying the classical windowed Fourier transform are the translation and modulation operators. While these fundamental operations seem simple in the classical setting, they become 
significantly 
more challenging when we deal with signals on graphs. For example, when we want to translate the blue Mexican hat wavelet on the real line in Figure \ref{Fig:essence}(a) to the right by 5, the result is the dashed red signal. However, it is not immediately clear what it means to translate the blue signal in Figure \ref{Fig:essence}(c) on the weighted graph in Figure \ref{Fig:essence}(b) ``to vertex 1000.'' Modulating a signal on the real line by a complex exponential corresponds to translation in the Fourier domain. 
However, the analogous spectrum in the graph setting is discrete and bounded, and therefore it is difficult to define a modulation in the vertex domain that corresponds to translation in the graph spectral domain. 

In this paper, an extended version of the short workshop proceeding \cite{shuman_SSP_2012}, we define generalized convolution, translation, and modulation operators for signals on graphs, analyze properties of these operators, and then use them to adapt the classical windowed Fourier transform to the graph setting. The result is a method to construct windowed Fourier frames, dictionaries of atoms adapted to the underlying graph structure that 
 enable vertex-frequency analysis, a generalization of time-frequency analysis to the graph setting. After a brief review of the classical windowed Fourier transform in the next section and some spectral graph theory background in Section \ref{Se:notation}, we introduce and study generalized convolution and translation operators in Section \ref{Se:operators} and generalized modulation operators in Section \ref{Se:modulation}. We then define and explore the properties of windowed graph Fourier frames in Section \ref{Se:frame}, where we also present illustrative examples of a graph spectrogram analysis tool and signal-adapted graph clustering. We conclude in Section \ref{Se:conclusion} with some comments on open issues.

\begin{figure}[h]
\hfill
\begin{minipage}[b]{.24\linewidth}
   \centering
   \centerline{\includegraphics[width=\linewidth]{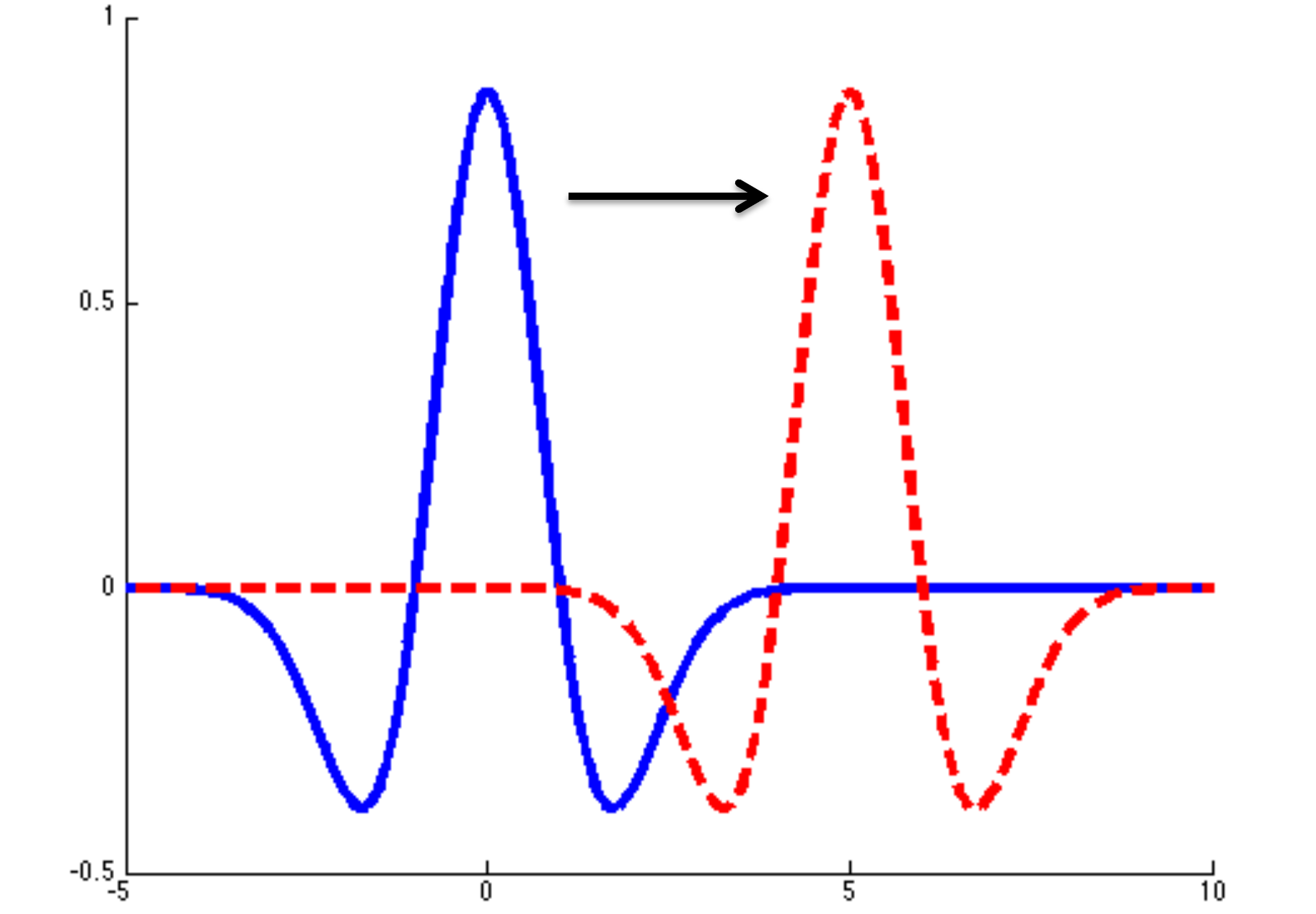}}
\centerline{\small{(a)}}
\end{minipage}
\hfill
\begin{minipage}[b]{.24\linewidth}
   \centering
   \centerline{\includegraphics[width=\linewidth]{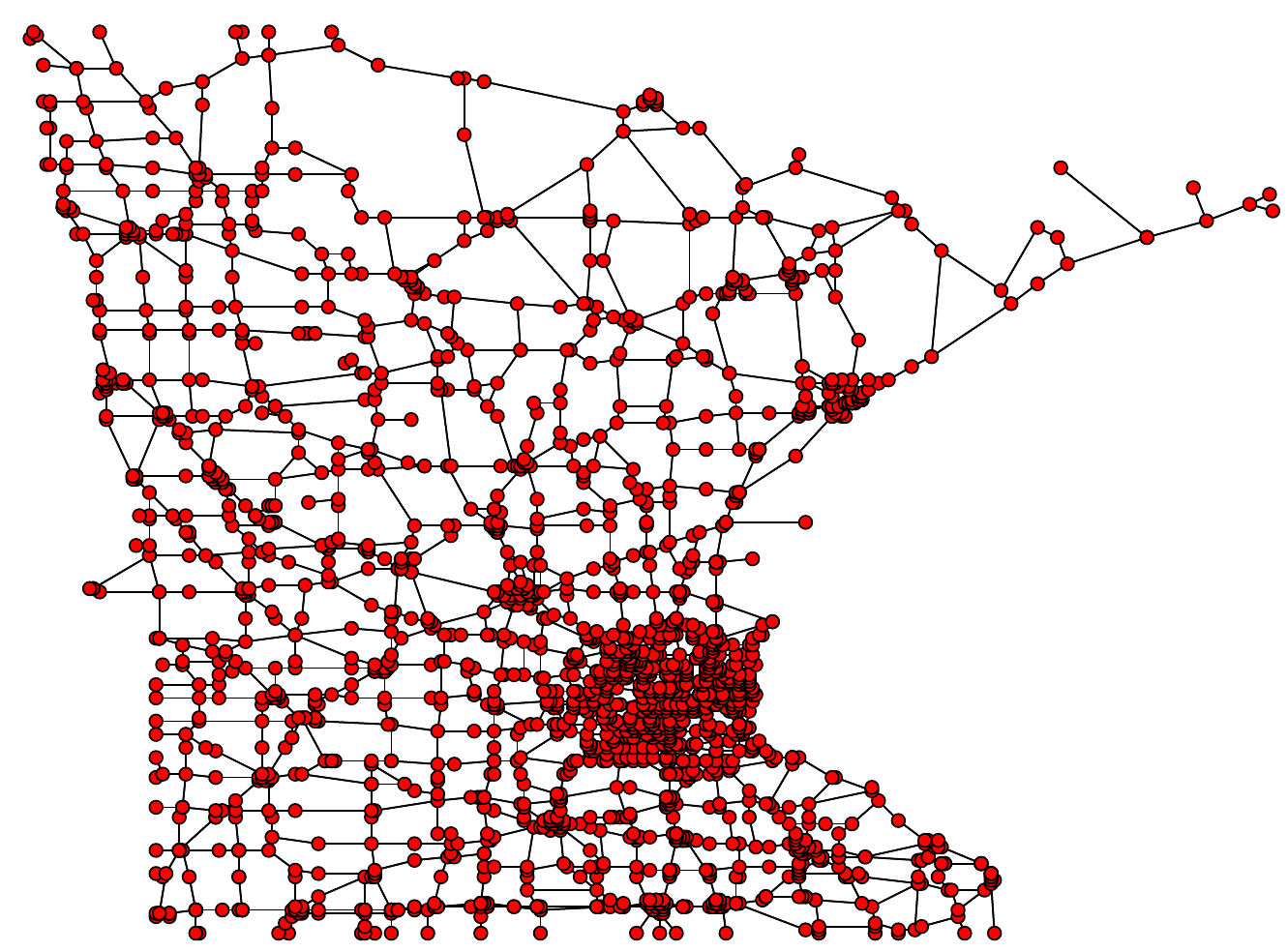}}
\centerline{\small{(b)~~}}
\end{minipage}
\hfill
\begin{minipage}[b]{.3\linewidth}
   \centering
   \centerline{\includegraphics[width=\linewidth]{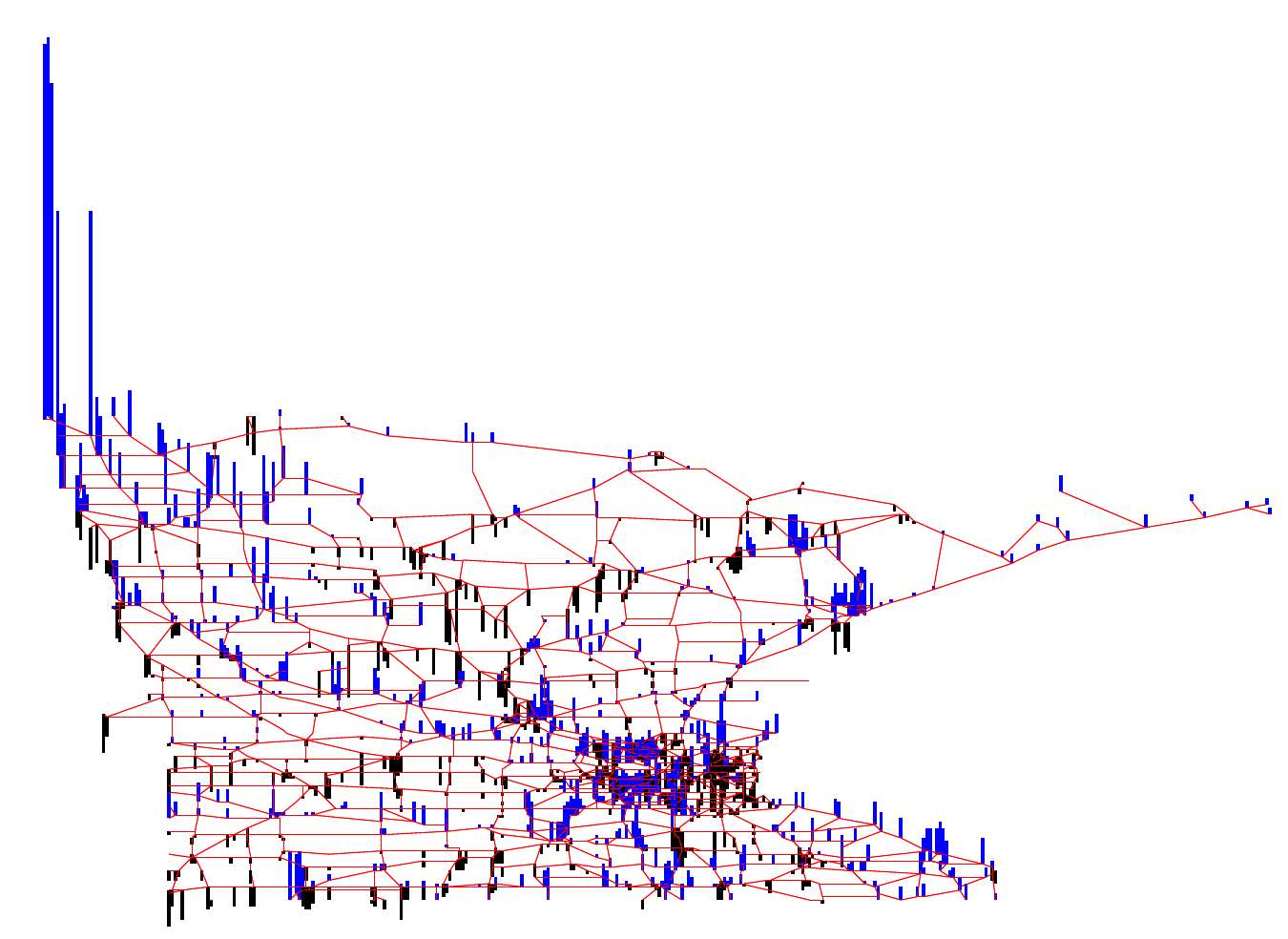}}
\centerline{\small{(c)~~}}
\end{minipage}
\hfill
\caption {(a) Classical translation. (b) The Minnesota road graph \cite{gleich}, whose edge weights are all equal to 1. (c) What does it mean to ``translate'' this signal on the vertices of the Minnesota road graph? The blue and black lines represent the magnitudes of the positive and negative components of the signal, respectively.}
  \label{Fig:essence}
\end{figure}

\section{The Classical Windowed Fourier Transform}
\label{Se:classical}
For any $f \in L^2(\Rbb)$ and $u \in \Rbb$, the translation operator $T_u: L^2(\Rbb) \rightarrow L^2(\Rbb)$ is defined by 
\begin{align}\label{Eq:classical_translation}
\left(T_u f\right)(t) := f(t-u),
\end{align}
and for any $\xi \in \Rbb$, the modulation operator $M_{\xi}: L^2(\Rbb) \rightarrow L^2(\Rbb)$ is defined by 
\begin{align} \label{Eq:classical_modulation}
\left(M_{\xi} f\right)(t) :=e^{2\pi i \xi t}f(t).
\end{align}
Now let $g \in L^2(\Rbb)$ be a window (i.e., a smooth, localized function) with $\norm{g}_2=1$. Then a windowed Fourier atom (see, e.g., \cite{flandrin}, \cite{groechenig}, \cite[Chapter 4.2]{mallat}) is given by
\begin{align}\label{Eq:classical_atom}
g_{u,\xi}(t):=\left(M_{\xi} T_u g\right)(t) = g(t-u)e^{2 \pi i \xi t},
\end{align}
and the windowed Fourier transform (WFT) 
of a function $f \in L^2(\Rbb)$ is
\begin{align}\label{Eq:classical_STFT}
Sf(u,\xi):=\ip{f}{g_{u,\xi}}=\int_{-\infty}^{\infty}f(t){[g(t-u)]^*}e^{-2 \pi i \xi t} dt.
\end{align}
An example of a windowed Fourier atom is shown in Figure \ref{Fig:classical_atom}.

\begin{figure}[h]
\centering
\hfill
\begin{minipage}[b]{.23\linewidth}
\hspace{.3in}\centerline{{$~~g(t)$}}
\centerline{\includegraphics[width=.85\linewidth]{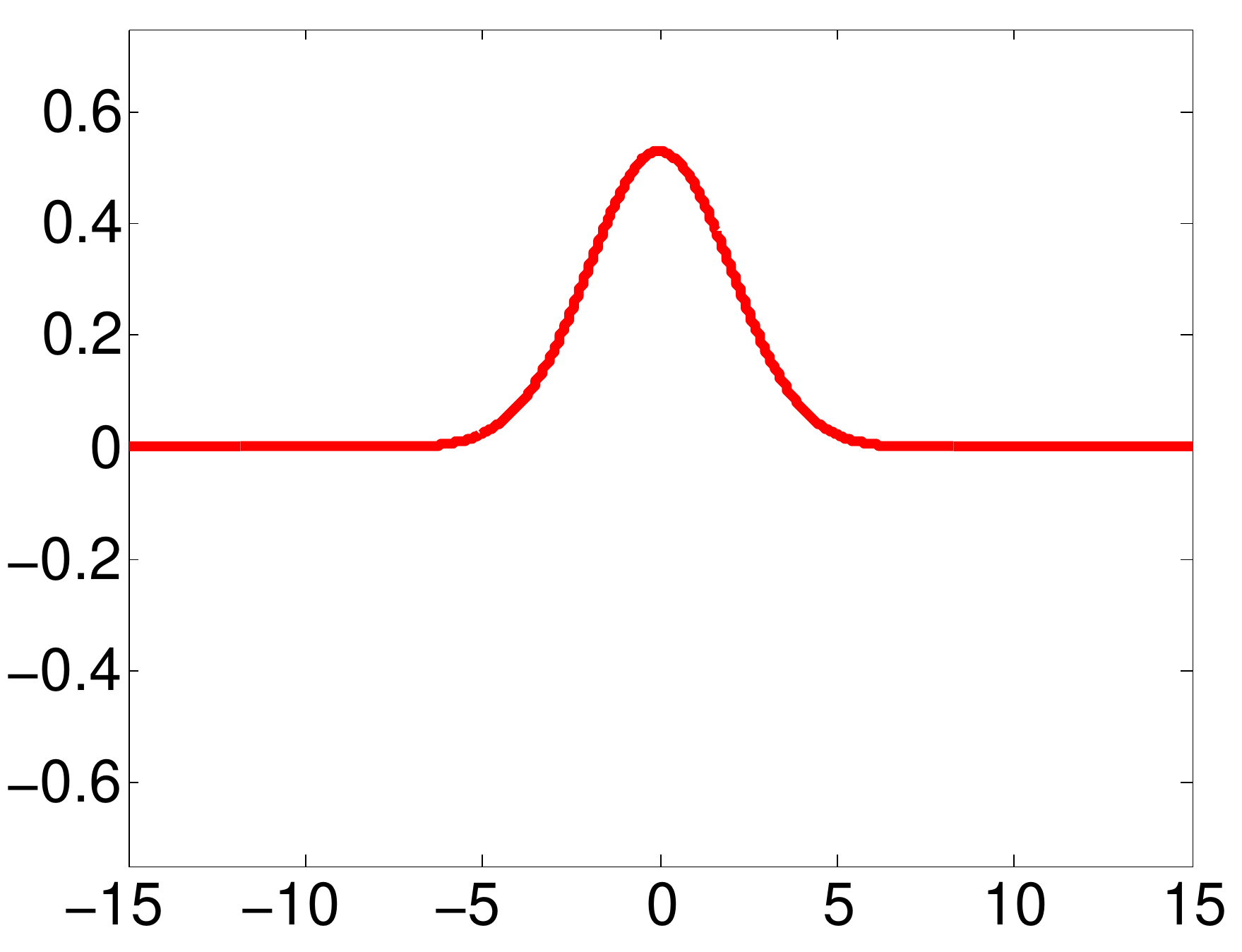}} 
\centerline{\small{~~(a)}}
\end{minipage} 
\hfill
\begin{minipage}[b]{.12\linewidth}
\centerline{{\small{Translation $T_5$}}} 
\centerline{$\Longrightarrow$}
\vspace{.55in}
\end{minipage}
\hfill
\begin{minipage}[b]{.23\linewidth}
\hspace{.3in}\centerline{{$~~(T_5 g)(t)=g(t-5)$}}
\centerline{\includegraphics[width=.85\linewidth]{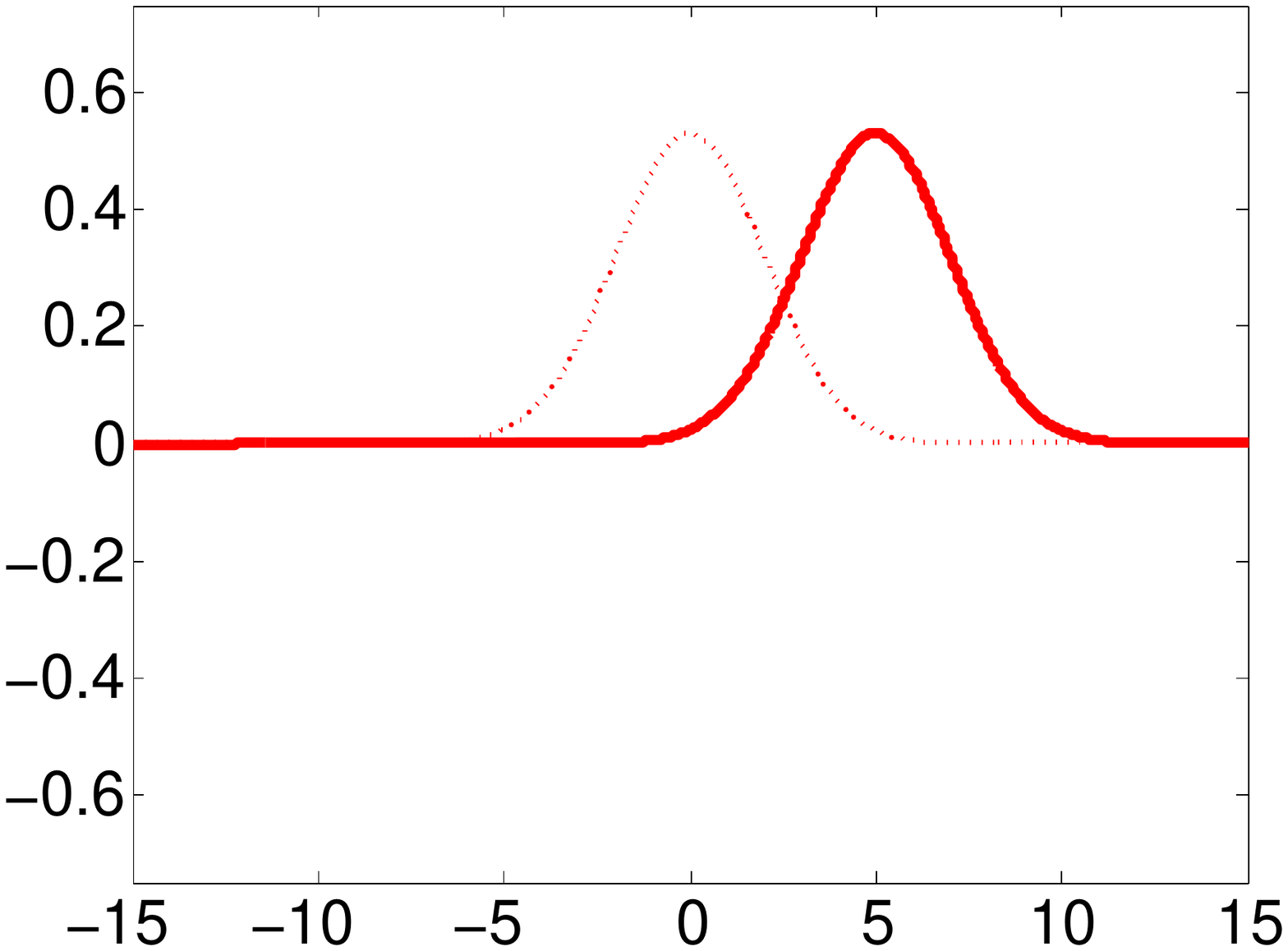}}   
\centerline{\small{~~(b)}}  
\end{minipage}
\hfill
\begin{minipage}[b]{.12\linewidth}
\centerline{{\small{Modulation $M_{\frac{1}{2}}$}}} 
\centerline{$\Longrightarrow$}
\vspace{.52in}
\end{minipage}
\hfill
\begin{minipage}[b]{.23\linewidth}
\hspace{.3in} \centerline{{$~~g_{5,\frac{1}{2}}(t)=(M_{\frac{1}{2}}T_5 g)(t)=g(t-5)e^{2\pi i \left(\frac{1}{2}\right)t}$}}
\centerline{\includegraphics[width=.85\linewidth]{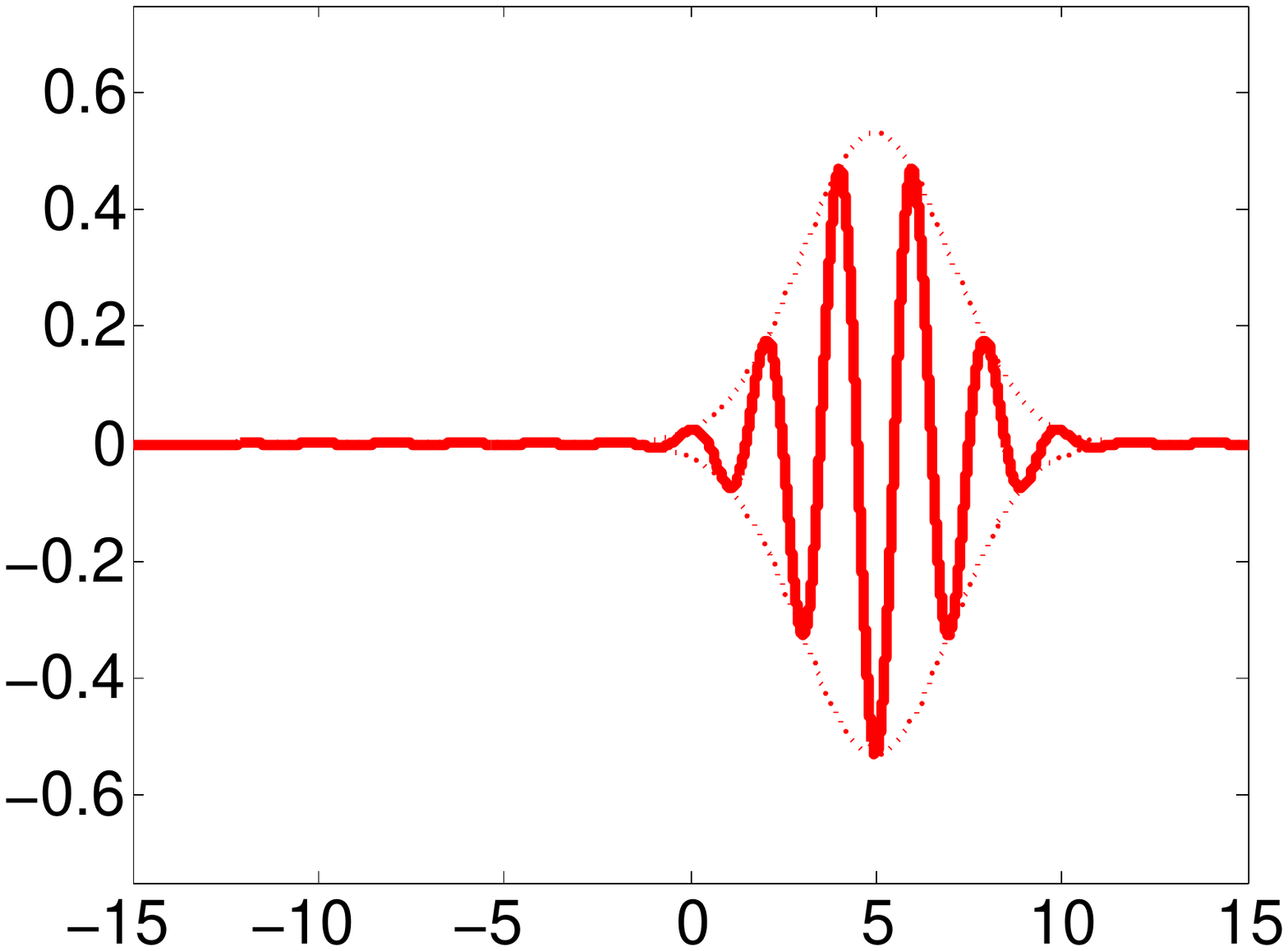}} 
\centerline{\small{~~(c)}}
\end{minipage} 
\hspace{.2in}\hfill
\caption {A classical windowed Fourier atom. In this example, the window is a Gaussian with standard deviation equal to 2 and scaled so that $\norm{g}_2=1$. The real part of the atom $g_{5,\frac{1}{2}}$ is shown in (c).} 
 \label{Fig:classical_atom}
\end{figure}

A second, perhaps more intuitive, way to interpret $Sf(u,\xi)$ is as the Fourier transform of $f (T_u g)^*$, evaluated at frequency $\xi$. That is, we multiply the signal $f$ by (the complex conjugate of) a translated window $T_u g$ in order to localize the signal to a specific area of interest in time, and then perform Fourier analysis on this localized, windowed signal. This interpretation is illustrated in Figure \ref{Fig:classical_sliding}. 

\begin{figure}[h]
\centering
\hfill
\begin{minipage}[b]{.23\linewidth}
\hspace{.3in}\centerline{{$~~f(t)g(t+6)$}}
\centerline{\includegraphics[width=.9\linewidth]{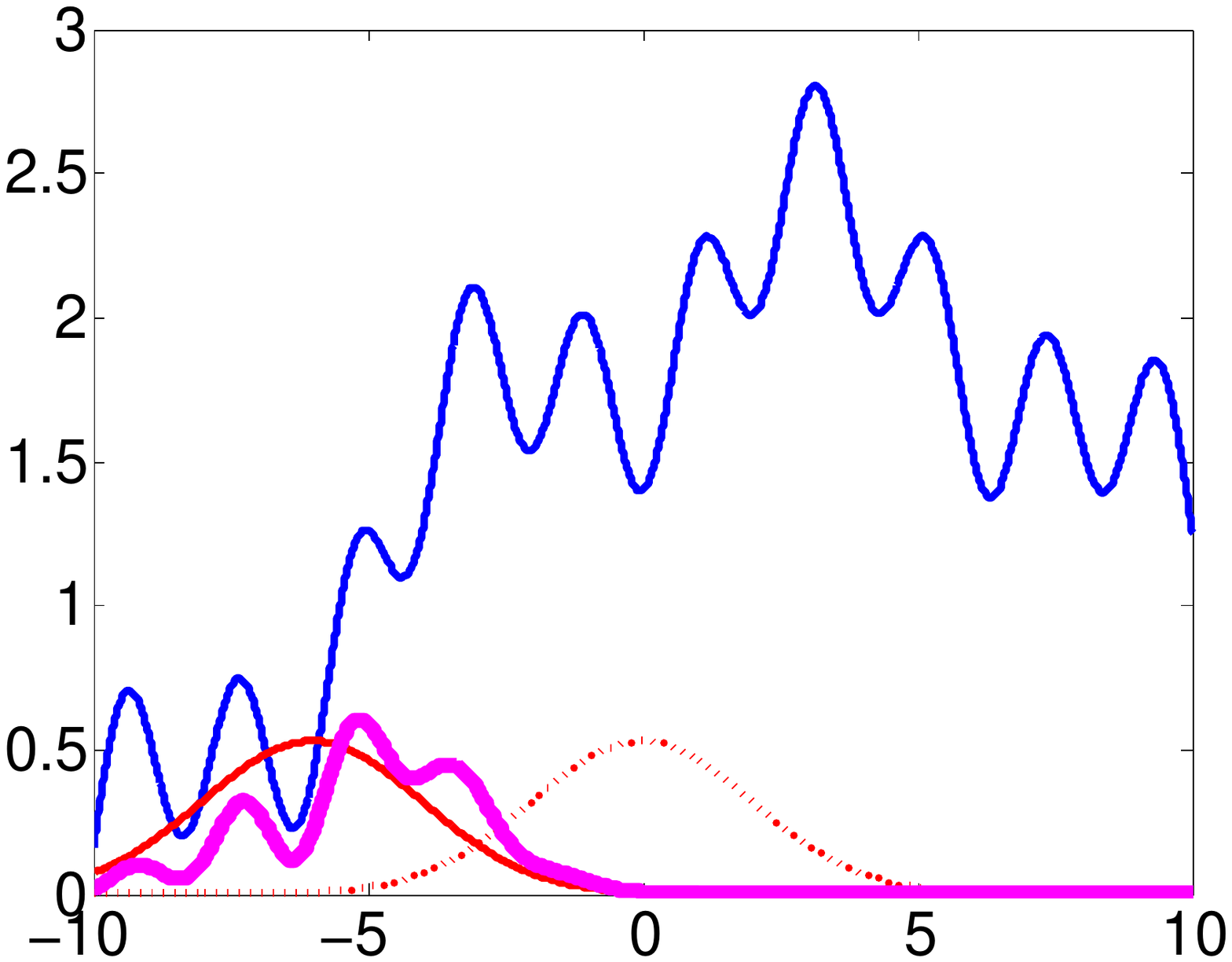}} 
\centerline{\small{~~(a)}}
\end{minipage} 
\hfill
\begin{minipage}[b]{.23\linewidth}
\hspace{.3in}\centerline{{$~~f(t)g(t+2)$}}
\centerline{\includegraphics[width=.9\linewidth]{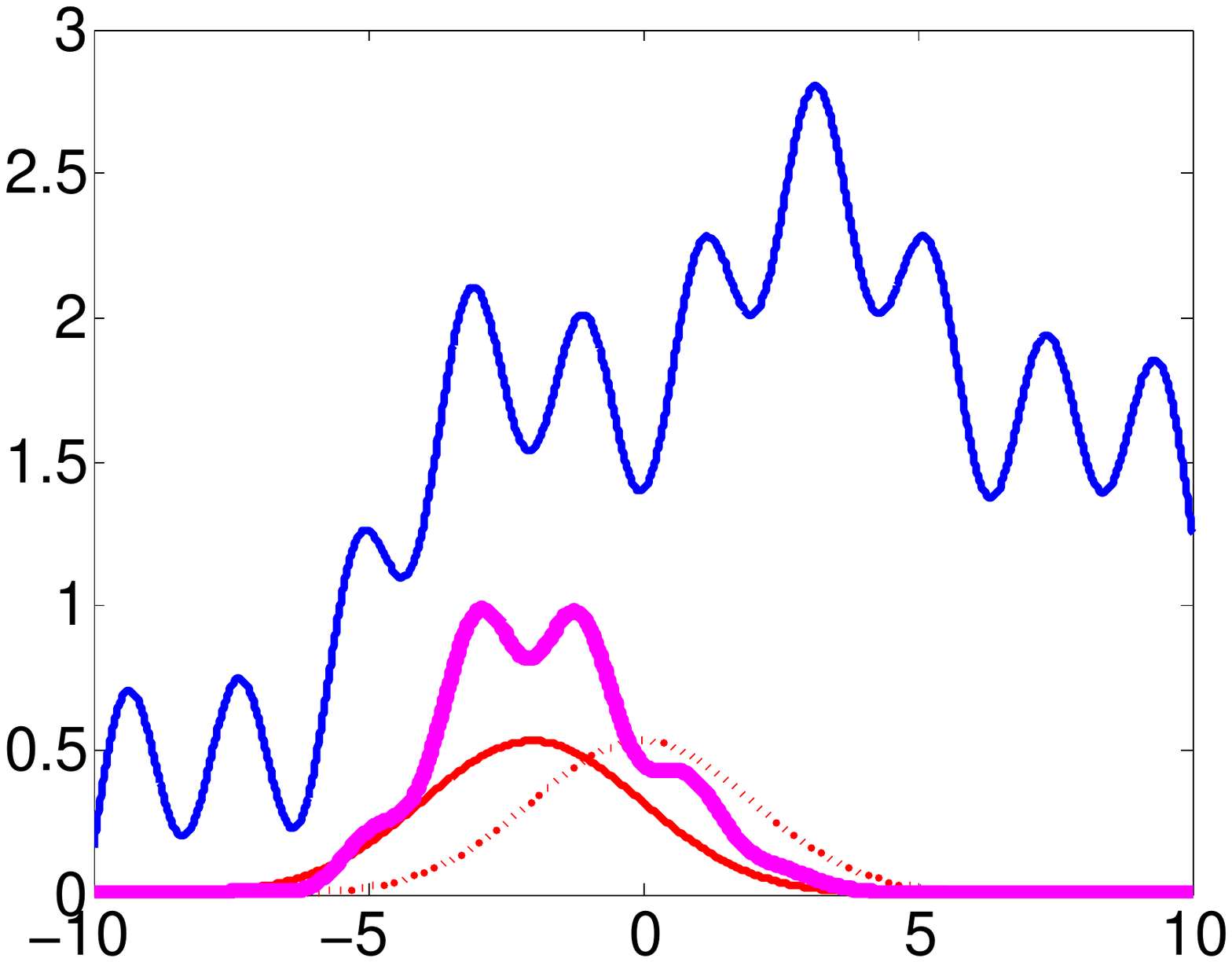}}   
\centerline{\small{~~(b)}}  
\end{minipage}
\hfill
\begin{minipage}[b]{.23\linewidth}
\hspace{.3in} \centerline{{$~~f(t)g(t-2)$}}
\centerline{\includegraphics[width=.9\linewidth]{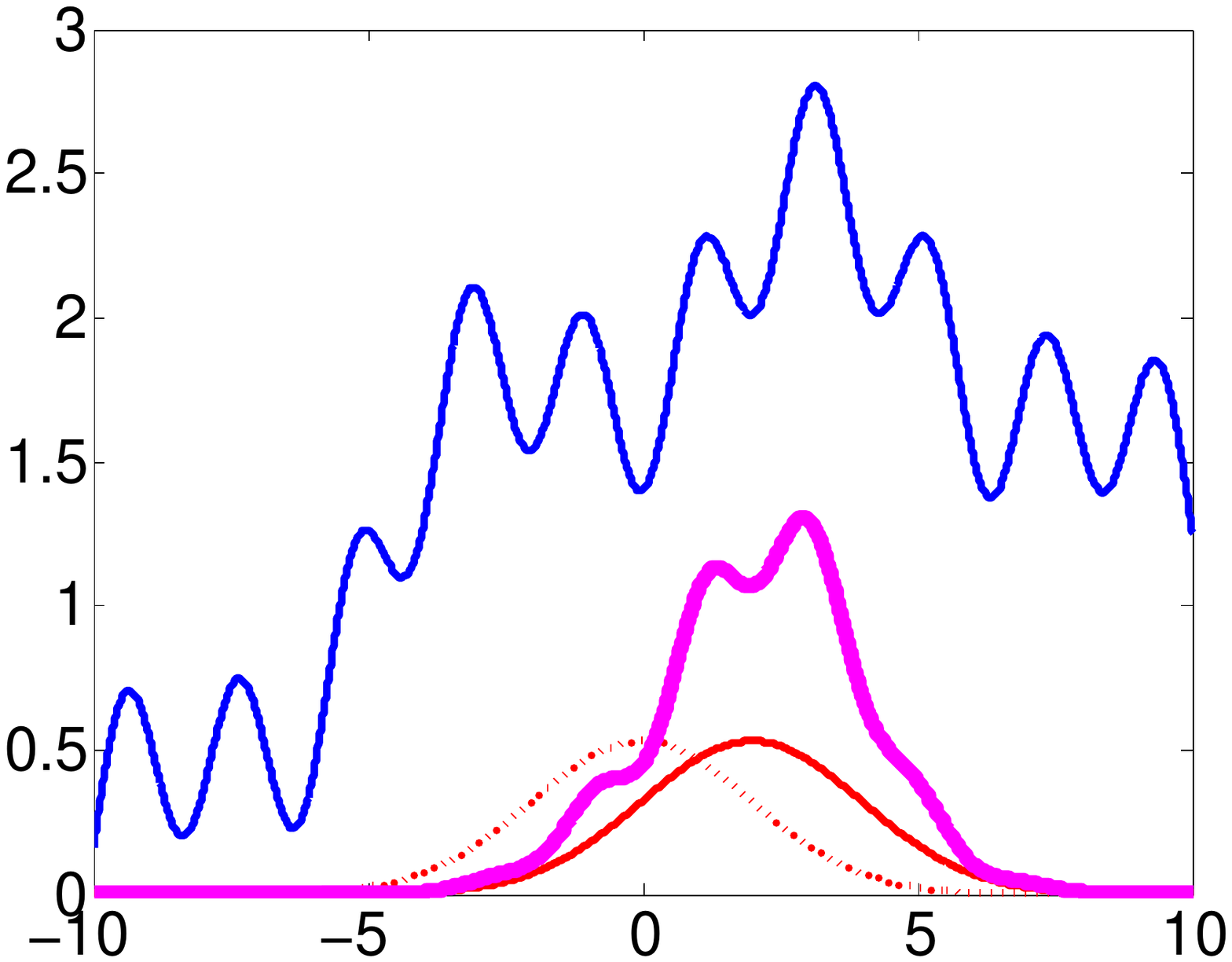}} 
\centerline{\small{~~(c)}}
\end{minipage} 
\hfill
\begin{minipage}[b]{.23\linewidth}
\hspace{.3in} \centerline{{$~~f(t)g(t-6)$}}
\centerline{\includegraphics[width=.9\linewidth]{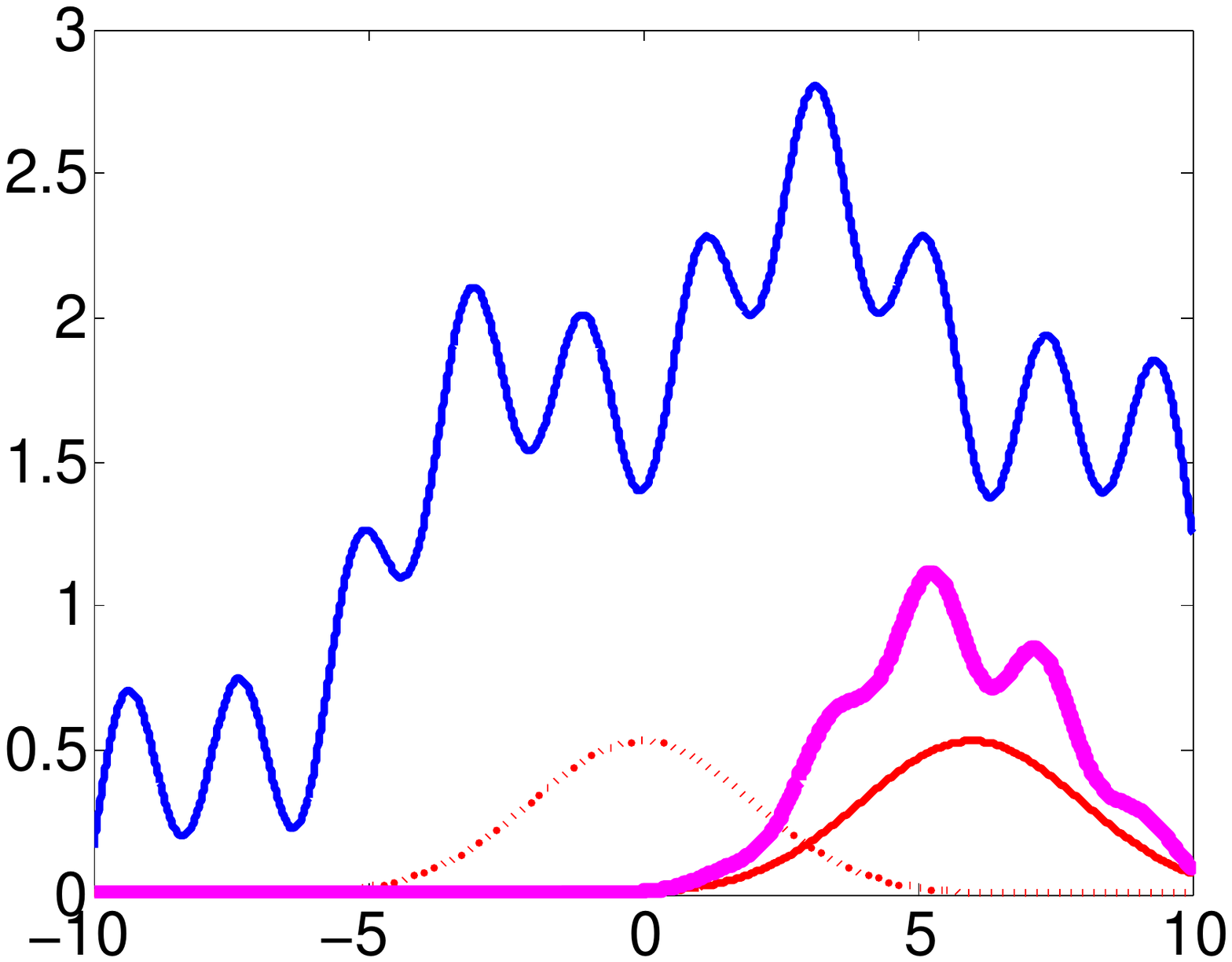}} 
\centerline{\small{~~(d)}}
\end{minipage} 
\hspace{.2in}\hfill
\caption {Second interpretation of the windowed Fourier transform: multiply the signal $f$ (shown in blue) by a sliding window (shown in solid red) to obtain a windowed signal (shown in magenta); then take Fourier transforms of the windowed signals.} 
 \label{Fig:classical_sliding}
\end{figure}

As mentioned in Section \ref{sec:intro}, our plan for the rest of the paper is to generalize the translation and modulation operators to the graph setting, and then mimic the classical windowed Fourier transform construction of \eqref{Eq:classical_atom} and \eqref{Eq:classical_STFT}.

\section{Spectral Graph Theory Notation and Background}
\label{Se:notation}
We consider undirected, connected, weighted graphs $\G = \{\V,\E,W\}$, where  $\V$ is a finite set of vertices $\V$ with $|\V|=N$, $\E$ is a set of edges, and $W$ is a weighted adjacency matrix 
(see, e.g., \cite{chung} for all definitions in this section). 
A signal $f: \V \rightarrow \Rbb$ defined on the vertices of the graph may be represented as a vector ${f} \in \Rbb^N$, where the $n^{th}$ component of the vector ${f}$ represents the signal value at the $n^{th}$ vertex in $\V$. The non-normalized graph Laplacian is defined as $\L:={D}-{W}$, where $D$ is the diagonal degree matrix. We denote by $d$ the vector of degrees (i.e., the diagonal elements of $D$), so that $d_n=\sum_{m \neq n} W_{mn}$ is the degree of vertex $n$. Then $d_{\min}:=\min_n \{d_n\}$ and $d_{\max}:=\max_n \{d_n\}$. 

As the graph Laplacian $\L$ is a real symmetric
matrix, it has a complete set of 
orthonormal eigenvectors, which we denote 
by $\left\{{\chi}_{\l}\right\}_{\l=0,1,\ldots,N-1}$. 
Without loss of generality, we assume that the associated real, non-negative Laplacian eigenvalues 
are ordered as 
$0=\lambda_0 < \lambda_1 \leq \lambda_2 ... \leq \lambda_{N-1}:=\lambda_{\max}$, and we denote the graph Laplacian spectrum by $\sigma(\L):=\{\lambda_0, \lambda_1,\ldots,\lambda_{N-1}\}$.

\subsection{The Graph Fourier Transform and the Graph Spectral Domain}
The classical Fourier transform is the expansion of a function $f$ in terms of the eigenfunctions of the Laplace operator, i.e., 
$\hat{f}(\xi)= \ip{f}{e^{2 \pi i \xi t}}$. 
Analogously, the \emph{graph Fourier transform} $\hat{{f}}$ of a function ${f}\in\mathbb{R}^{N}$ on the vertices of $\G$
is the expansion of ${f}$ in terms of the eigenfunctions of the graph Laplacian. It is defined by
\begin{align}\label{Eq:graph_FT}
\hat{f}(\lambda_{\l}) := \ip{f}{\chi_{\l}} = \sum_{n=1}^N f(n) \chi^*_{\l}(n), 
\end{align}
where we adopt the convention that the inner product be
conjugate-linear in the second argument. 
The \emph{inverse graph Fourier transform} is then given by
\begin{align}\label{Eq:graph_IFT}
f(n) = \sum_{\l=0}^{N-1} \hat{f}(\lambda_{\l}) \chi_{\l}(n).
\end{align}
With this definition of the graph Fourier transform, the Parseval relation holds; i.e., for any $f,g, \in \Rbb^N$,
\begin{align*}
\ip{f}{g}=\ip{\hat{f}}{\hat{g}},
\end{align*}
and thus
\begin{align*}
\sum_{n=1}^N |f(n)|^2=\norm{f}_2^2=\ip{f}{f}=\ip{\hat{f}}{\hat{f}}=\norm{\hat{f}}_2^2=\sum_{\l=0}^{N-1}|\hat{f}(\lambda_{\l})|^2.
\end{align*}

Note that the definitions of the graph Fourier transform and its inverse in \eqref{Eq:graph_FT} and \eqref{Eq:graph_IFT}  depend on the choice of graph Laplacian eigenvectors, which is not necessarily unique. Throughout this paper, we do not specify how to choose these eigenvectors, but assume they are fixed. The ideal choice of the eigenvectors in order to optimize the theoretical analysis conducted here and elsewhere remains an interesting open question; however, in most applications with extremely large graphs, the explicit computation of a full eigendecomposition is not practical anyhow, and methods that only utilize the graph Laplacian through sparse matrix-vector multiplication are preferred. We discuss these computational issues further in Section \ref{Se:comp}.

It is also possible to use other bases to define the forward and inverse graph Fourier transforms. The eigenvectors of the normalized graph Laplacian $\tilde{\L}:=D^{-\frac{1}{2}}\L D^{-\frac{1}{2}}$ comprise one such basis that is prominent in the graph signal processing literature. While we use the non-normalized graph Laplacian eigenvectors as the Fourier basis throughout this paper, the normalized graph Laplacian eigenvectors can also be used to define generalized translation and modulation operators, and we comment briefly on the resulting differences in the Appendix.

We consider signals' representations in both the vertex domain (analogous to the time/space domains in classical Euclidean settings) and the graph spectral domain (analogous to the frequency domain in classical settings). As an example, in Figure \ref{Fig:two_domains}, we show these two different representations of the signal from Figure \ref{Fig:essence}(c). 
In this case, we actually generate the signal in the graph spectral domain, starting with a continuous \emph{kernel}, $\hat{f}: [0,\lambda_{\max}] \rightarrow \Rbb$, given by $\hat{f}(\lambda_{\l}):=C e^{-\tau \lambda_{\l}}$, where $\tau=5$. We then form the discrete signal $\hat{f}$ by evaluating the continuous kernel $\hat{f}(\cdot)$ at each of the graph Laplacian eigenvalues in $\sigma(\L)$. The constant $C$ is chosen so that $\norm{f}_2=1$, and the kernel $\hat{f}(\cdot)$ is referred to as a normalized heat kernel (see, e.g., \cite[Chapter 10]{chung}). 
Finally, we generate the signal shown in both Figure  
\ref{Fig:essence}(c) and \ref{Fig:two_domains}(a) by taking the inverse graph Fourier transform \eqref{Eq:graph_IFT} of $\hat{f}$.
\begin{figure}[h]
\hfill
\begin{minipage}[b]{.4\linewidth}
   \centering
   \centerline{\includegraphics[width=\linewidth]{fig_minn_signal_3}} 
   \vspace{.18in}
\centerline{\small{(a)~~~~~~~}}
\end{minipage}
\hfill
\hspace{.4in}
\begin{minipage}[b]{.3\linewidth}
   \centering
   \centerline{\includegraphics[width=\linewidth]{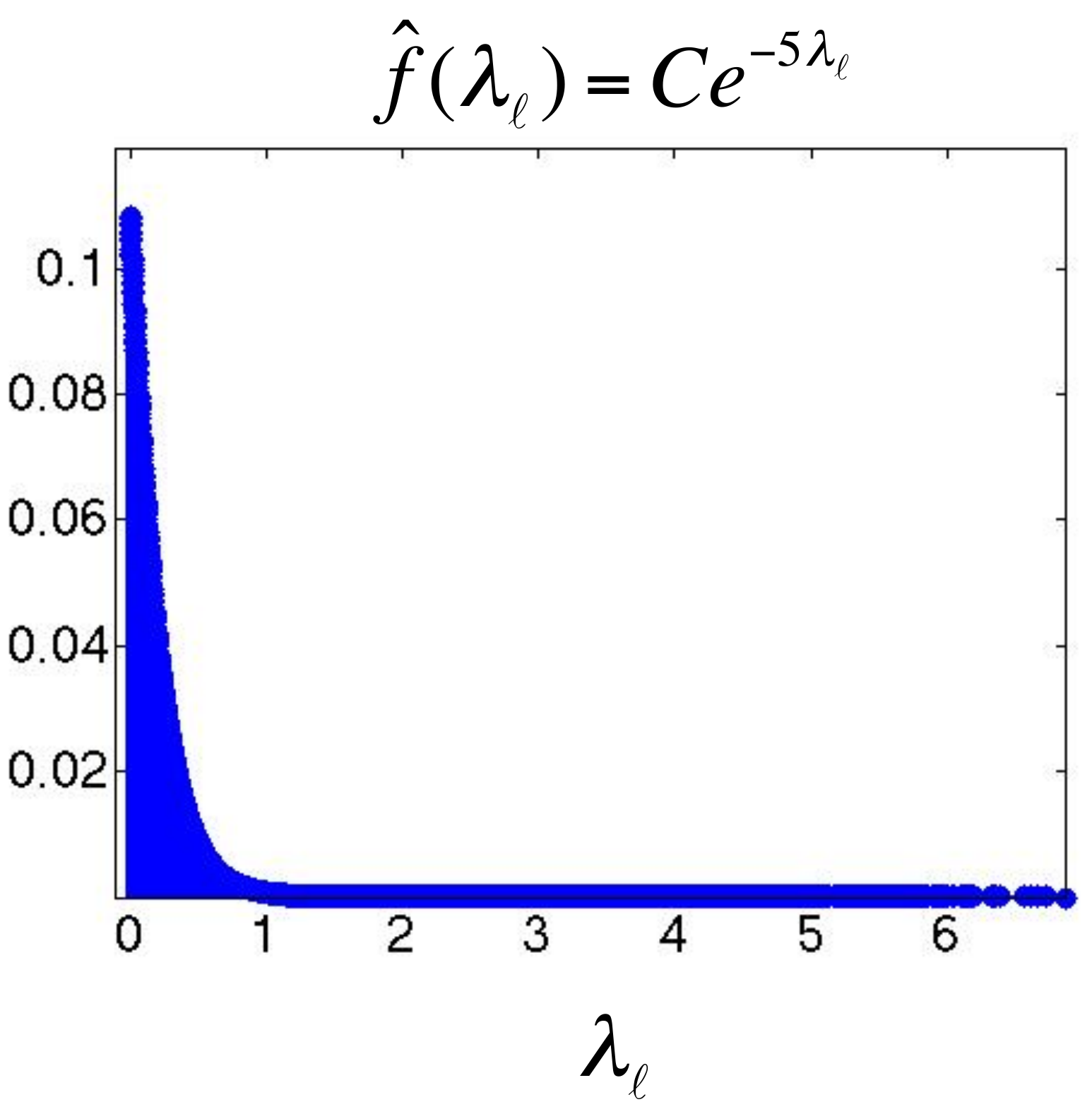}}
\centerline{\small{~~~~~~(b)}}
\end{minipage}
\hspace{.5in}
\hfill
\caption {A signal $f$ represented in two domains. (a) The vertex domain. (b) The graph spectral domain.}
  \label{Fig:two_domains}
\end{figure}

Some intuition about the graph spectrum can also be carried over from the classical setting to the graph setting. In the classical setting, the Laplacian eigenfunctions (complex exponentials) associated with lower eigenvalues (frequencies) are relatively smooth, whereas those associated with higher eigenvalues oscillate more rapidly. The graph Laplacian eigenvalues and associated eigenvectors satisfy
\begin{align*}
\lambda_{\l}=\chi_{\l}^{\transpose} \L \chi_{\l} = \sum_{(m,n)\in \E} W_{mn} [\chi_{\l}(m)-\chi_{\l}(n)]^2.
\end{align*}
Therefore, since each term in the summation of the right-hand side is non-negative, the eigenvectors associated with smaller eigenvalues are smoother; i.e., the component differences between neighboring vertices are small (see, e.g., \cite[Figure 2]{shuman_SPM}). As the eigenvalue or ``frequency'' increases, larger differences in neighboring components of the graph Laplacian eigenvectors may be present (see \cite{ shuman_SPM,zhu_rabbat1} for further discussions of notions of frequency for the graph Laplacian eigenvalues). This well-known property has been extensively utilized in a wide range of problems, including spectral clustering \cite{spectral_clustering}, machine learning \cite[Section III]{chapelle}, and ill-posed inverse problems in image processing \cite{elmoataz}.

\subsection{Localization of Graph Laplacian Eigenvectors and Coherence}
There have recently been a number of interesting research results concerning the localization properties of graph Laplacian eigenvectors. 
For different classes of random graphs, \cite{dekel, dumitriu, tran} show that with high probability for graphs of sufficiently large size, the eigenvectors of the graph Laplacian (or in some cases, the graph adjacency operator), are delocalized; i.e., the restriction of the eigenvector to a large set must have substantial energy, or in even stronger statements, the element of the matrix $\boldsymbol{\chi}:=\left[\chi_0,\chi_1,\ldots,\chi_{N-1}\right]$ with the largest absolute value is small. We refer to this latter value as the \emph{mutual coherence} (or simply \emph{coherence}) between the basis of Kronecker deltas on the graph and the basis of graph Laplacian eigenvectors:
\begin{align}\label{Eq:mu_def}
\mu:=\max_{\substack{\l \in \{0,1,\ldots,N-1\} \\ i \in \{1,2,\ldots,N\}}} |\ip{\chi_{\l}}{\delta_i}| \in \left[\frac{1}{\sqrt{N}},1\right],
\end{align}
where
\begin{align*}
\delta_i(n)=
\begin{cases}
1,&\hbox{ if }i=n \\
0,&\hbox{ otherwise }
\end{cases}.
\end{align*}
While the previously mentioned non-localization results rely on estimates from random matrix theory, Brooks and Lindenstrauss \cite{brooks} also show that for sufficiently large, unweighted, non-random, regular graphs that do not have too many short cycles through the same vertex, in order for $\sum_{i\in \S} \left|\chi_{\l}(i)\right|^2 > \epsilon$ for any $\l$, the subset $\S \subset \V$ must satisfy $|\S|\geq N^{\delta}$, where the constant $\delta$ depends on both $\epsilon$ and structural restrictions placed on the graph. 

These non-localization results are consistent with the intuition one might gain from considering the eigenvectors of the Laplacian for the unweighted path  
and ring graphs shown in Figure \ref{Fig:path_ring}. The eigenvalues of the graph Laplacian of the unweighted path graph with $N$ vertices are given by
\begin{align*}
\lambda_{\l}=2-2 \cos\left(\frac{\pi \l}{N}\right),~\forall \l \in \{0,1,\ldots,N-1\},
\end{align*} 
and one possible choice of associated orthonormal eigenvectors is
\begin{align}\label{Eq:DCT_eigs}
\chi_0(n)&=\frac{1}{\sqrt{N}},~\forall n \in \{1,2,\ldots,N\},\hbox{ and } \nonumber \\
\chi_{\l}(n)&=\sqrt{\frac{2}{N}}\cos\left(\frac{\pi \l (n-0.5)}{N}\right) \hbox{ for } \l=1,2,\ldots,N-1.
\end{align}
These graph Laplacian eigenvectors, which are shown in Figure \ref{Fig:path_ring}(b), 
are the basis vectors 
in the Discrete Cosine Transform (DCT-II) transform \cite{dct} used in JPEG image compression. Like the continuous complex exponentials, 
they are non-localized, globally oscillating functions. 
The \emph{unordered} eigenvalues of the graph Laplacian of the unweighted ring graph with $N$ vertices are given by (see, e.g., \cite[Chapter 3]{gray},
\cite[Proposition 1]{olfati_ramanujan})
\begin{align*}
\bar{\lambda}_{\l}=2-2 \cos\left(\frac{2\pi \l}{N}\right),~\forall \l \in \{0,1,\ldots,N-1\},
\end{align*} 
and one possible choice of associated orthonormal eigenvectors is
\begin{align}\label{Eq:ring_eigenvectors}
\chi_{\l}=\frac{1}{\sqrt{N}}\left[1,\omega^{\l},\omega^{2\l},\ldots,\omega^{(N-1)\l}\right]^{\transpose},\hbox{ where }\omega=e^{\frac{2\pi j}{N}}.
\end{align}
These eigenvectors correspond to the columns of the Discrete Fourier Transform (DFT) matrix.
With this choice of 
eigenvectors, the coherence of the unweighted 
ring graph with $N$ vertices 
is $\frac{1}{\sqrt{N}}$, 
the smallest possible coherence of any graph. In this case, the basis of DFT columns and the basis of Kronecker deltas on vertices of the graph are said to be \emph{mutually unbiased bases}.

\begin{figure}[h]
\hfill
\begin{minipage}[b]{.23\linewidth}
   \centering
   \centerline{\includegraphics[width=\linewidth]{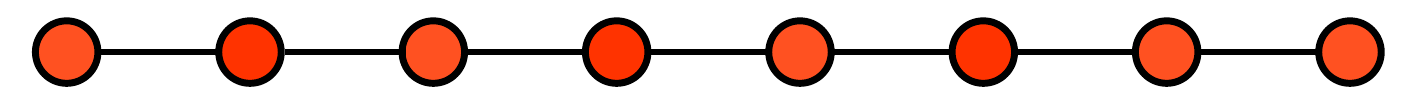}}
    \vspace{.6in}
\centerline{\small{(a)}}
\end{minipage}
\hfill
\begin{minipage}[b]{.4\linewidth}
   \centering
   \centerline{\includegraphics[width=\linewidth]{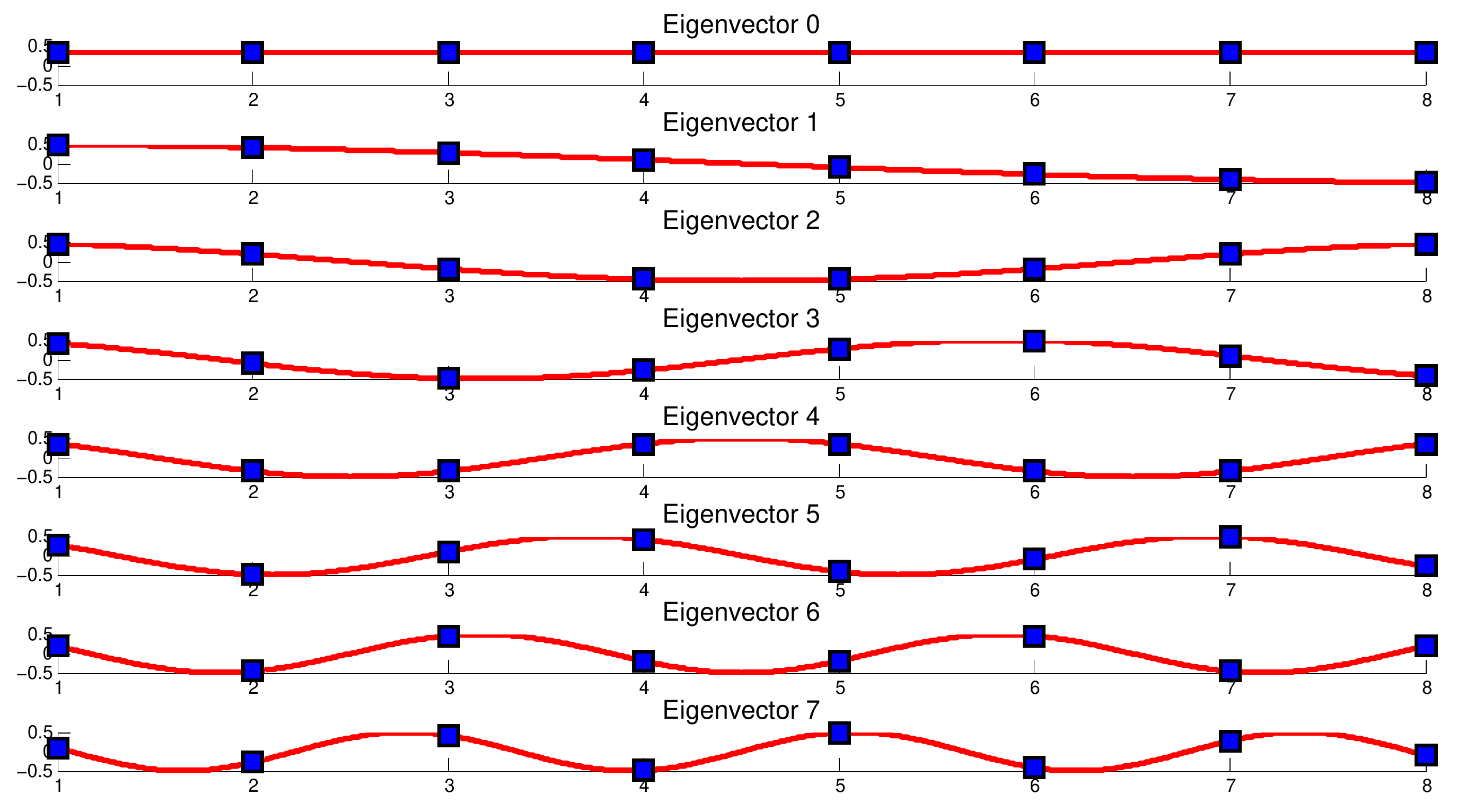}} 
\centerline{\small{(b)}}
\end{minipage}
\hfill
\begin{minipage}[b]{.23\linewidth}
   \centering
   \centerline{\includegraphics[width=\linewidth]{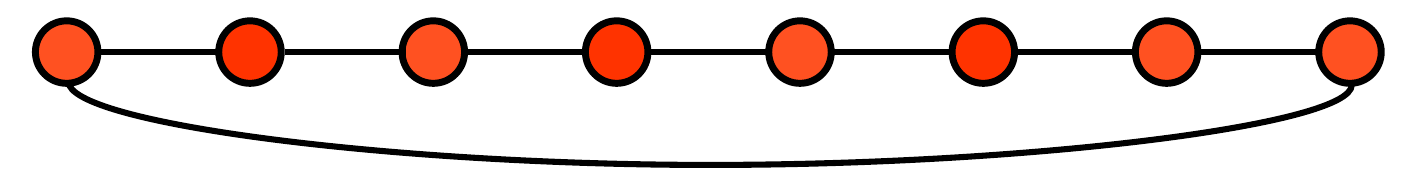}} 
   \vspace{.5in}
\centerline{\small{(c)}}
\end{minipage} \hfill 
\caption {(a) The path graph. (b) One possible choice of the graph Laplacian eigenvectors of the path graph is the discrete cosines. (c) The ring graph.}  
\label{Fig:path_ring}
\end{figure}

However, empirical studies such as \cite{mcgraw} show that certain graph Laplacian eigenvectors may in fact be highly localized, especially when the graph features one or more vertices whose degrees are significantly higher or lower than the average degree, or when the graph features a high degree of clustering (many triangles in the graph). Moreover, Saito and Woei \cite{saito} identify a class of starlike trees with highly-localized eigenvectors, some of which are even close to a delta on the vertices of the graph. The following example shows two less structured graphs that also have high coherence.
\begin{example} \label{Ex:high_coherence}
In Figure \ref{Fig:coherence}, we show two weighted, undirected graphs with coherences of 0.96 and 0.94. The sensor network is generated by randomly placing 500 nodes in the $[0,1] \times [0,1]$ plane. The Swiss roll graph is generated by randomly sampling 1000 points from the two-dimensional Swiss roll manifold embedded in $\Rbb^3$. In both cases, the weights are assigned with a thresholded Gaussian kernel weighting function based on the Euclidean distance between nodes:
\begin{align}\label{Eq:gkw}
W_{ij}=
\begin{cases}
\exp\left({-\frac{[d(i,j)]^2}{2\sigma_1^2}}\right) &\mbox{if } d(i,j) \leq \sigma_2 \\
0 &\mbox{otherwise}
\end{cases}.
\end{align}
For the sensor network, $\sigma_1=0.074$ and $\sigma_2 = 0.075$. For the Swiss roll graph, $\sigma_1=0.100$ and $\sigma_2= 0.215$. The coherences are based on the orthonormal Laplacian eigenvectors computed by \verb=MATLAB='s \verb=svd= function.  
\end{example}
\begin{figure}[h]
\hfill
\begin{minipage}[b]{.32\linewidth}
   \centering
   \centerline{\includegraphics[width=\linewidth]{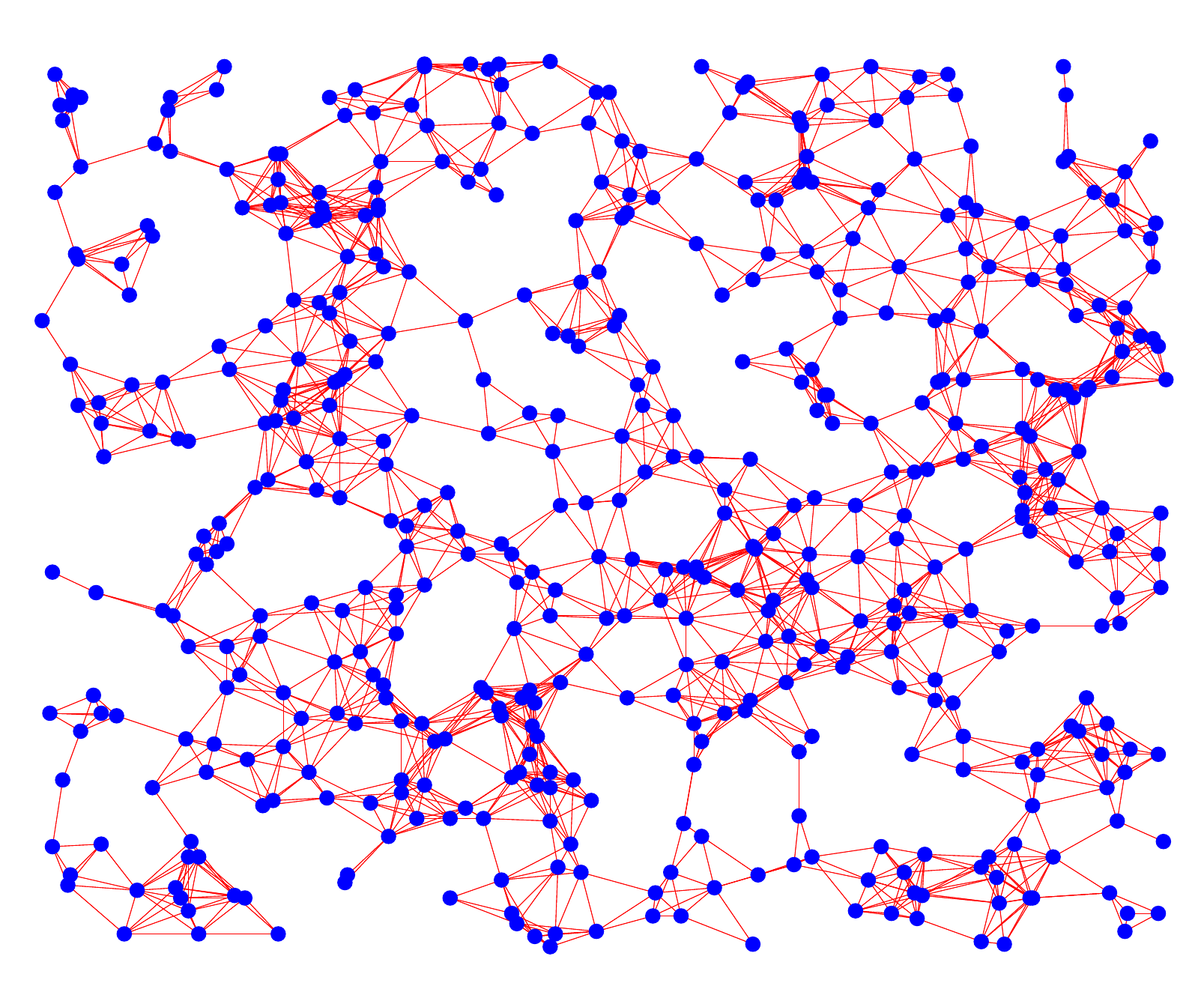}} 
\centerline{\small{(a)}}
\end{minipage}
\hfill
\begin{minipage}[b]{.36\linewidth}
   \centering
   \centerline{\includegraphics[width=\linewidth]{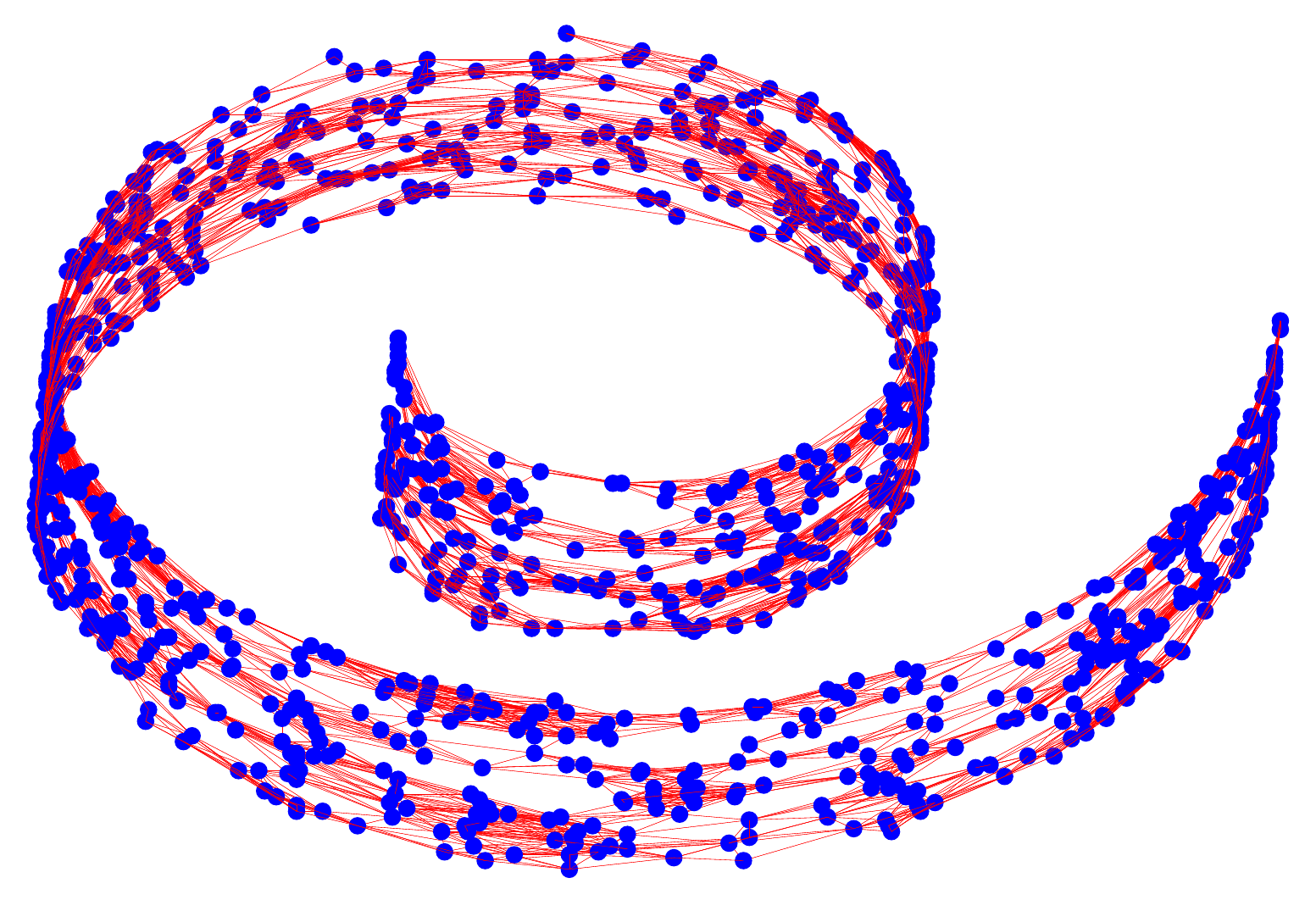}} 
\centerline{\small{(b)}}
\end{minipage} \hfill \hfill
\caption {Two graphs with high coherence. (a) The coherence of this random sensor network is 0.96. (b) The coherence of this graph on a cloud of 1000 points randomly sampled on the Swiss roll manifold is 0.94.} 
  \label{Fig:coherence}
\end{figure}
The existence of localized eigenvectors can limit the degree to which our intuition from classical time-frequency analysis extends to localized vertex-frequency analysis of signals on graphs. We discuss this point in more detail in Sections \ref{Se:translation}, \ref{Se:tiling}, and \ref{Se:limitations}. 

Finally, we define some 
quantities that are closely related to the coherence $\mu$  and 
will also be useful in our analysis. 
We denote the largest absolute value of the elements of a given graph Laplacian eigenvector by
\begin{align}\label{Eq:mu_l_def}
\mu_{\l} := \norm{\chi_{\l}}_{\infty}=\max_{i \in \{1,2,\ldots,N\}} \left|\chi_{\l}(i)\right|,
\end{align}
and the largest absolute value of a given row of $\boldsymbol{\chi}$ by
\begin{align}\label{Eq:tilde_mu_i_def}
\nu_i:=\max_{\l \in \{0,1,\ldots,N-1\}} \left|\chi_{\l}(i)\right|.
\end{align}
Note that 
\begin{align}\label{Eq:mu_equiv}
\mu=\max_{\l \in \{0,1,\ldots,N-1\}} \left\{ \mu_{\l}\right\}=\max_{i \in \{1,2,\ldots,N\}} \left\{ 
\nu_{i}\right\}.
\end{align}

\section{Generalized Convolution and Translation Operators} \label{Se:operators}

The main objective of this section is to define a generalized translation operator that allows us to shift a window around the vertex domain so that it is localized around any given vertex, just as we shift a window along the real line to any center point in the classical windowed Fourier transform for signals on the real line. We use a generalized notion of translation that is -- aside from a constant factor that depends on the number of vertices in the graph -- the same notion used as one component of the spectral graph wavelet transform (SGWT) in \cite[Section 4]{sgwt}.
However, in order to leverage intuition from classical time-frequency analysis, we motivate its definition differently here by first defining a generalized convolution operator for signals on graphs. We then discuss and analyze a number of properties of the generalized translation as a standalone operator, including the localization of translated kernels. 

\subsection{Generalized Convolution of Signals on Graphs}
\label{Se:convolution}
For signals $f,g \in L^2(\Rbb)$, the classical convolution product $h=f \ast g$ 
is defined as
\begin{align} \label{Eq:conv_class}
h(t)=(f \ast g)(t):=\int_{\Rbb} f(\tau)g(t-\tau)d\tau.
\end{align}
Since the simple translation $g(t-\tau)$ cannot be directly extended to the graph setting, we cannot directly generalize \eqref{Eq:conv_class}. However, the classical convolution product also 
satisfies
\begin{align}\label{Eq:conv}
h(t)=(f \ast g)(t)=\int_{\Rbb}\hat{h}(k)\psi_k(t) dk 
= \int_{\Rbb}\hat{f}(k)\hat{g}(k)\psi_k(t) dk,
\end{align}
where $\psi_k(t)=e^{2\pi i k t}$. This important property that convolution in the time domain is equivalent to multiplication in the Fourier domain is the notion we generalize instead. Specifically, by replacing the complex exponentials in \eqref{Eq:conv} with the graph Laplacian eigenvectors, we define a \emph{generalized convolution} of signals $f,g \in \Rbb^N$ on a graph by
\begin{align}\label{Eq:gen_convolution}
(f \ast g)(n):=\sum_{\l=0}^{N-1} \hat{f}(\lambda_{\l})\hat{g}(\lambda_{\l}) \chi_{\l}(n).
\end{align}
Using notation from the theory of matrix functions \cite{higham}, we can also write the generalized convolution as 
\begin{align*}
f \ast g =\hat{g}(\L)f=\boldsymbol{\chi}\left[
\begin{array}{ccc}
\hat{g}(\lambda_0) && 0 \\
&\ddots & \\
0&&\hat{g}(\lambda_{N-1})
\end{array}
\right]\boldsymbol{\chi}^* f. 
\end{align*}
\begin{proposition} \label{Prop:conv_prop}
The generalized convolution product defined in \eqref{Eq:gen_convolution}
satisfies the following properties:
\begin{enumerate}
\item Generalized convolution in the vertex domain is multiplication in the graph spectral domain:
\begin{align}\label{Eq:conv_thm}
\widehat{f \ast g} = \hat{f} \hat{g}.
\end{align}
\item Let $\alpha \in \Rbb$ be arbitrary. Then
\begin{align}\label{Eq:scalar_mult}
\alpha (f \ast g) = (\alpha f) \ast g = f \ast (\alpha g).
\end{align}
\item Commutativity:
\begin{align}\label{Eq:commutativity}
f \ast g = g \ast f.
\end{align}
\item Distributivity:
\begin{align}\label{Eq:distributivity}
f \ast (g+h) = f \ast g + f \ast h.
\end{align}
\item Associativity:
\begin{align}\label{Eq:associativity}
(f \ast g) \ast h = f \ast (g \ast h).
\end{align}
\item Define a function $g_0 \in \Rbb^N$ by $g_0(n):=\sum_{\l=0}^{N-1} \chi_{\l}(n)$. Then $g_0$ is an identity for the generalized convolution product:
\begin{align} \label{Eq:identity}
f \ast g_0 = f.
\end{align}
\item An invariance property with respect to the graph Laplacian (a difference operator):
\begin{align}\label{Eq:diff_op}
\L (f \ast g) = (\L f) \ast g = f \ast (\L g).
\end{align}
\item The sum of the generalized convolution product of two signals is a constant times the product of the sums of the two signals:
\begin{align}\label{Eq:integration}
\sum_{n=1}^N (f \ast g)(n) = \sqrt{N} \hat{f}(0)\hat{g}(0) = \frac{1}{\sqrt{N}}\left[\sum_{n=1}^N f(n) \right] \left[\sum_{n=1}^N g(n) \right].
\end{align}
\end{enumerate}
\end{proposition}


\subsection{Generalized Translation on Graphs}
\label{Se:translation}
The application of the classical translation operator $T_u$ defined in \eqref{Eq:classical_translation}
to a function $f \in L^2(\Rbb)$ can be seen as a convolution with $\delta_u$:
\begin{align*}
(T_u f)(t):=f(t-u)=(f \ast \delta_u)(t) 
\stackrel{\eqref{Eq:conv}}= \int_{\Rbb}\hat{f}(k)\widehat{\delta_u}(k)\psi_k(t) dk 
= \int_{\Rbb}\hat{f}(k){\psi}^*_k(u)\psi_k(t) dk,
\end{align*}
where the equalities are in the weak sense.
Thus, for any signal $f \in \Rbb^N$ defined on the 
the graph $\G$ and any $i \in \{1,2,\ldots,N\}$, we also define a \emph{generalized translation operator} $T_{i}: \Rbb^N \rightarrow \Rbb^N$ via generalized convolution with a delta centered at vertex $i$:
\begin{equation} \label{Eq:new_translation}
\left(T_i f\right)(n):= \sqrt{N}(f \ast \delta_i)(n) \stackrel{\eqref{Eq:gen_convolution}}= 
\sqrt{N}\sum_{\l=0}^{N-1}\hat{f}(\lambda_{\l})\chi_{\l}^*(i)\chi_{\l}(n).
\end{equation}
The translation \eqref{Eq:new_translation} is a kernelized operator. The window to be shifted around the graph is defined in the graph spectral domain via the kernel $\hat{f}(\cdot)$. To translate this window to vertex $i$, the $\l^{th}$ component of the kernel is multiplied by $\chi_{\l}^*(i)$, and then an inverse graph Fourier transform is applied. 
As an example, in Figure \ref{Fig:trans}, we apply generalized translation operators to 
the 
normalized heat kernel from Figure \ref{Fig:essence}(c). We can see that doing so has the desired effect of shifting a window around the graph, centering it at any given vertex $i$.
\begin{figure}[h]
\hfill
\begin{minipage}[b]{.31\linewidth}
   \centering
   \centerline{\includegraphics[width=\linewidth]{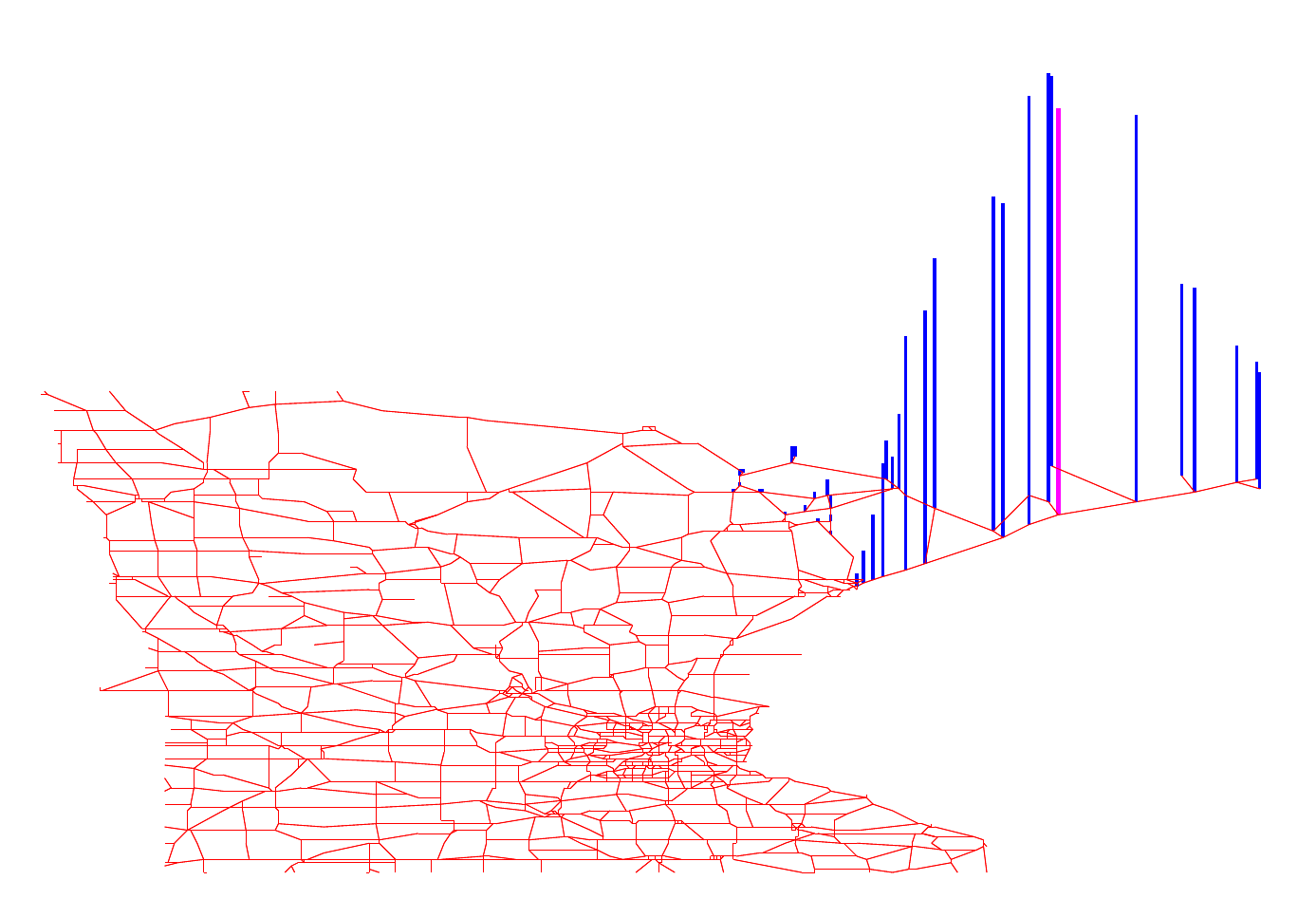}} 
\centerline{\small{(a)}~~}
\end{minipage}
\hfill
\begin{minipage}[b]{.31\linewidth}
   \centering
   \centerline{\includegraphics[width=\linewidth]{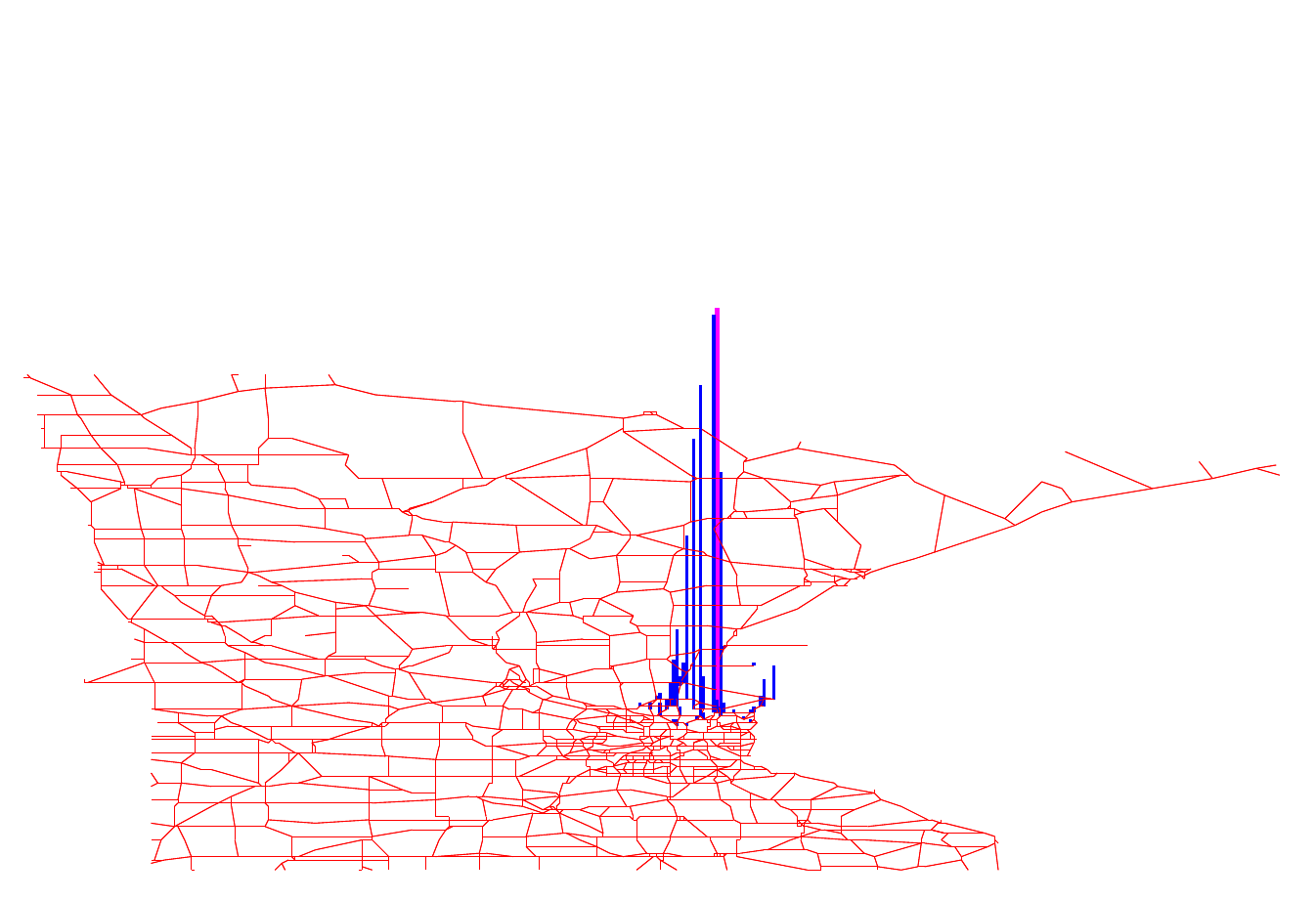}} 
\centerline{\small{(b)}~~~~}
\end{minipage}
\hfill
\begin{minipage}[b]{.31\linewidth}
   \centering
   \centerline{\includegraphics[width=\linewidth]{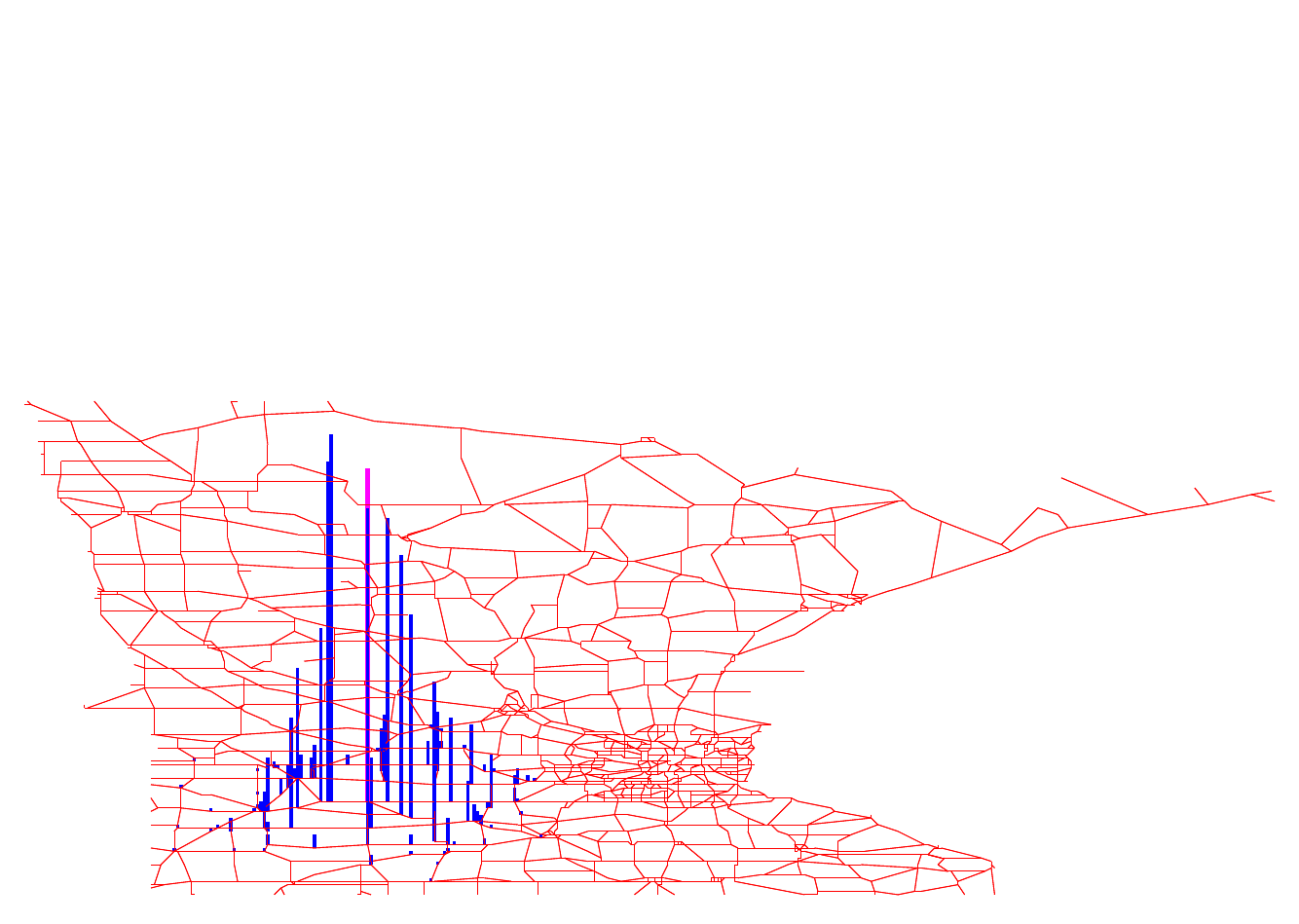}} 
\centerline{\small{(c)}~~~~}
\end{minipage} \hfill 
%
\caption {The translated signals (a) $T_{200} f$, (b) $T_{1000} f$, and (c) $T_{2000} f$, where $f$, the signal shown in Figure \ref{Fig:essence}(c), is a normalized heat kernel satisfying $\hat{f}(\lambda_{\l})=Ce^{-5\lambda_{\l}}$. The component of the translated signal at the center vertex is highlighted in magenta.} 
  \label{Fig:trans}
\end{figure}

\subsection{Properties of the Generalized Translation Operator}
Some expected properties of the generalized translation operator follow immediately from the generalized convolution properties of Proposition \ref{Prop:conv_prop}.
\begin{corollary}\label{Co:trans_conv}
For any $f, g \in \Rbb^N$ and $i, j \in \{1,2,\ldots,N\}$,
\begin{enumerate}
\item $T_i(f \ast g) = (T_i f) \ast g = f \ast (T_i g)$.
\item $T_i T_j f = T_j T_i f$.
\item $\sum_{n=1}^N (T_i f)(n)= \sqrt{N}\hat{f}(0)=\sum_{n=1}^N f(n)$.
\end{enumerate}
\end{corollary}
However, the niceties end there, and we should also point out some properties that are true for the classical translation operator, but not for the generalized translation operator for signals on graphs. First, unlike the classical case, the set of translation operators $\{T_i\}_{i\in \{1,2,\ldots,N\}}$ do not form a mathematical group; i.e., $T_i T_j \neq T_{i+j}$. In the very special case of \emph{shift-invariant} graphs \cite[p. 158]{grady}, which are graphs for which the DFT basis vectors \eqref{Eq:ring_eigenvectors} are graph Laplacian eigenvectors (the unweighted ring graph shown in Figure \ref{Fig:path_ring}(c) is one such graph), we have
\begin{align} \label{Eq:ring_per}
T_i T_j = T_{\left[\bigl((i-1)+(j-1)\bigr) \hbox{ mod } N \right] +1},~\forall i,j \in \{1,2,\ldots,N\}.
\end{align}
However, \eqref{Eq:ring_per} is not true in general for arbitrary graphs. Moreover, while the idea of successive translations $T_i T_j$ carries a clear meaning in the classical case, it is not a particularly meaningful concept in the graph setting due to our definition of generalized translation as a kernelized operator.

Second, unlike the classical translation operator, the generalized translation operator is not an isometric operator; i.e., $\norm{T_i f}_2 \neq \norm{f}_2$ for all indices $i$ and signals $f$. Rather, we have
\begin{lemma}\label{Le:trans_norm_bounds}
For any $f\in \Rbb^N$,
\begin{align}\label{Eq:trans_norm_bounds}
|\hat{f}(0)| \leq \norm{T_i f}_2 \leq \sqrt{N}
\nu_i \norm{f}_2 \leq \sqrt{N} \mu \norm{f}_2.
\end{align}
\end{lemma}
\begin{proof}
\begin{align} \label{Eq:trans_operator_1}
\norm{T_i f}_2^2 &= \sum_{n=1}^N \left(\sqrt{N}\sum_{\l=0}^{N-1} \hat{f}(\lambda_{\l}) \chi_{\l}^*(i) \chi_{\l}(n) \right)^2 \nonumber \\
&=N \sum_{\l=0}^{N-1} \sum_{\l^{\prime}=0}^{N-1} \hat{f}(\lambda_{\l}) \hat{f}(\lambda_{\l^{\prime}}) \chi_{\l}^*(i) \chi_{\l^{\prime}}^*(i) \sum_{n=1}^N  \chi_{\l}(n) \chi_{\l^{\prime}}(n) \nonumber \\
&= N \sum_{\l=0}^{N-1} |\hat{f}(\lambda_{\l})|^2 \left|\chi_{\l}^*(i) \right|^2 \\
& \leq N 
\nu_i^2 \norm{f}_2^2. \label{Eq:trans_operator_2}
\end{align}
Substituting $\chi_{0}(i)=\frac{1}{\sqrt{N}}$ into \eqref{Eq:trans_operator_1} yields the first inequality in \eqref{Eq:trans_norm_bounds}.
\end{proof}
\noindent If $\mu=\frac{1}{\sqrt{N}}$, as is the case for the ring graph with the DFT graph Laplacian eigenvectors, then $|\chi_{\l}^*(i)|^2=\frac{1}{N}$ for all $i$ and $\l$, and \eqref{Eq:trans_operator_1} becomes $\norm{T_i f}_2=\norm{f}_2$.  However, for general graphs, 
$|\chi_{\l}^*(i)|$ may be small or even zero for some $i$ and $\l$, and thus $\norm{T_i f}_2$ may be significantly smaller than $\norm{f}_2$, and in fact may even be zero. Meanwhile, if $\mu$ is close to 1, then we may also have the case that $\norm{T_i f}_2 \approx \sqrt{N} \norm{f}_2$. Figures \ref{Fig:trans_norms_low} and \ref{Fig:trans_norms_high} show examples of $\left\{\norm{T_i f}_2\right\}_{i=1,2,\ldots,N}$ for different graphs and different kernels.


\begin{figure}[h]
\hfill
\begin{minipage}[b]{.31\linewidth}
   \centering
   \centerline{\includegraphics[width=\linewidth]{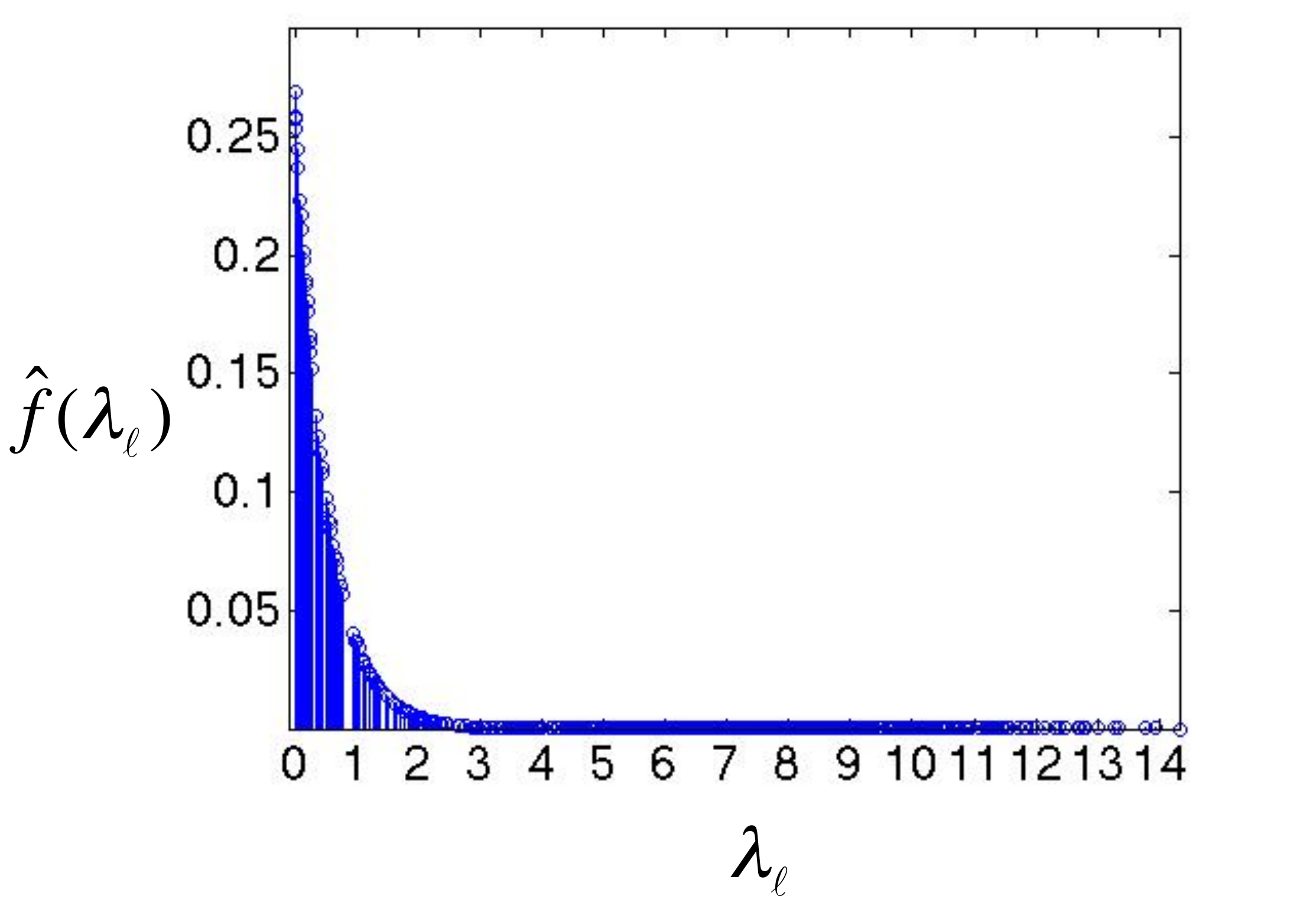}} 
\centerline{\small{(a)}~~}
\end{minipage}
\hfill
\begin{minipage}[b]{.31\linewidth}
   \centering
   \centerline{\includegraphics[width=\linewidth]{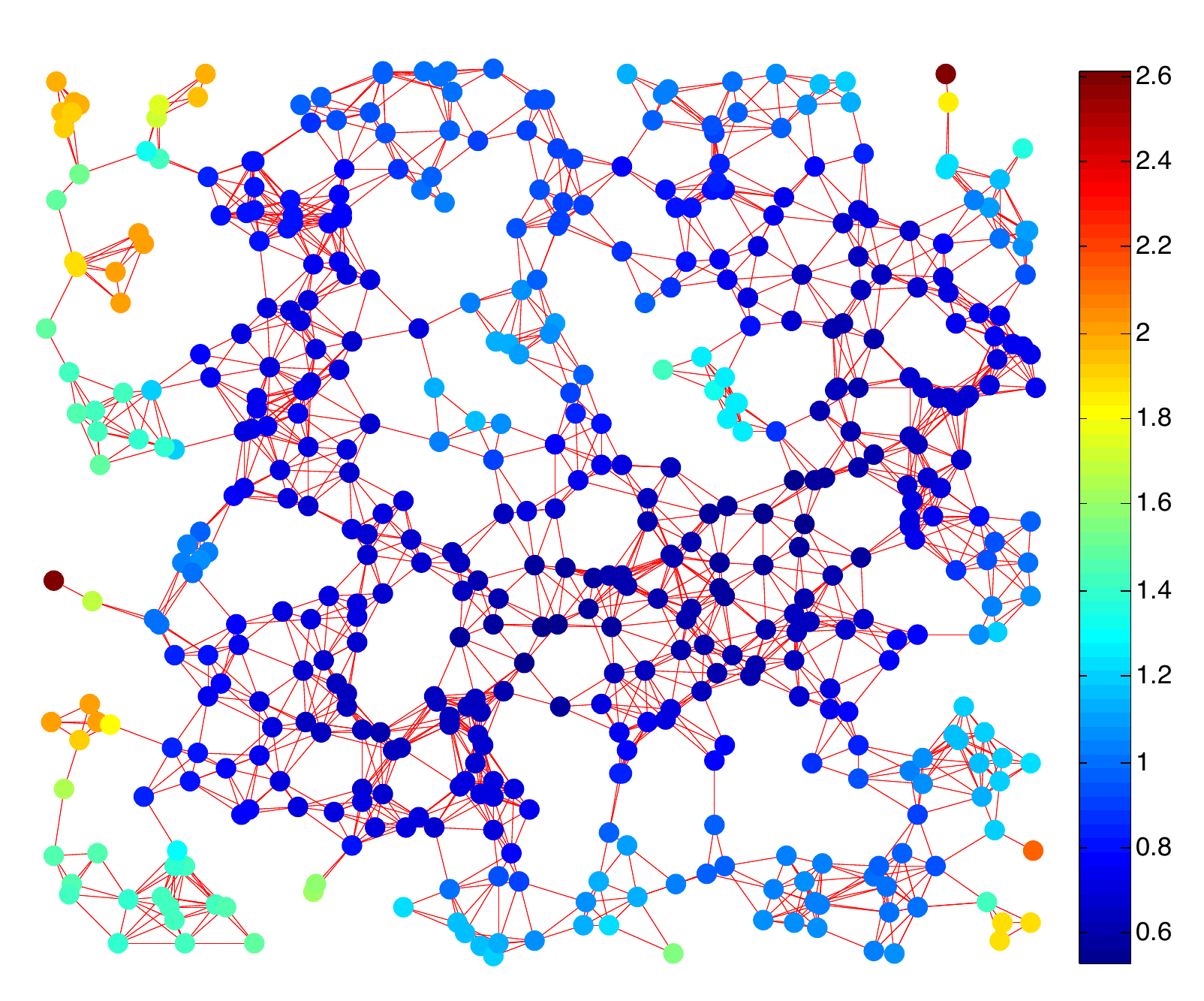}} 
\centerline{\small{(b)}~~~~}
\end{minipage}
\hfill
\begin{minipage}[b]{.31\linewidth}
   \centering
   \centerline{\includegraphics[width=\linewidth]{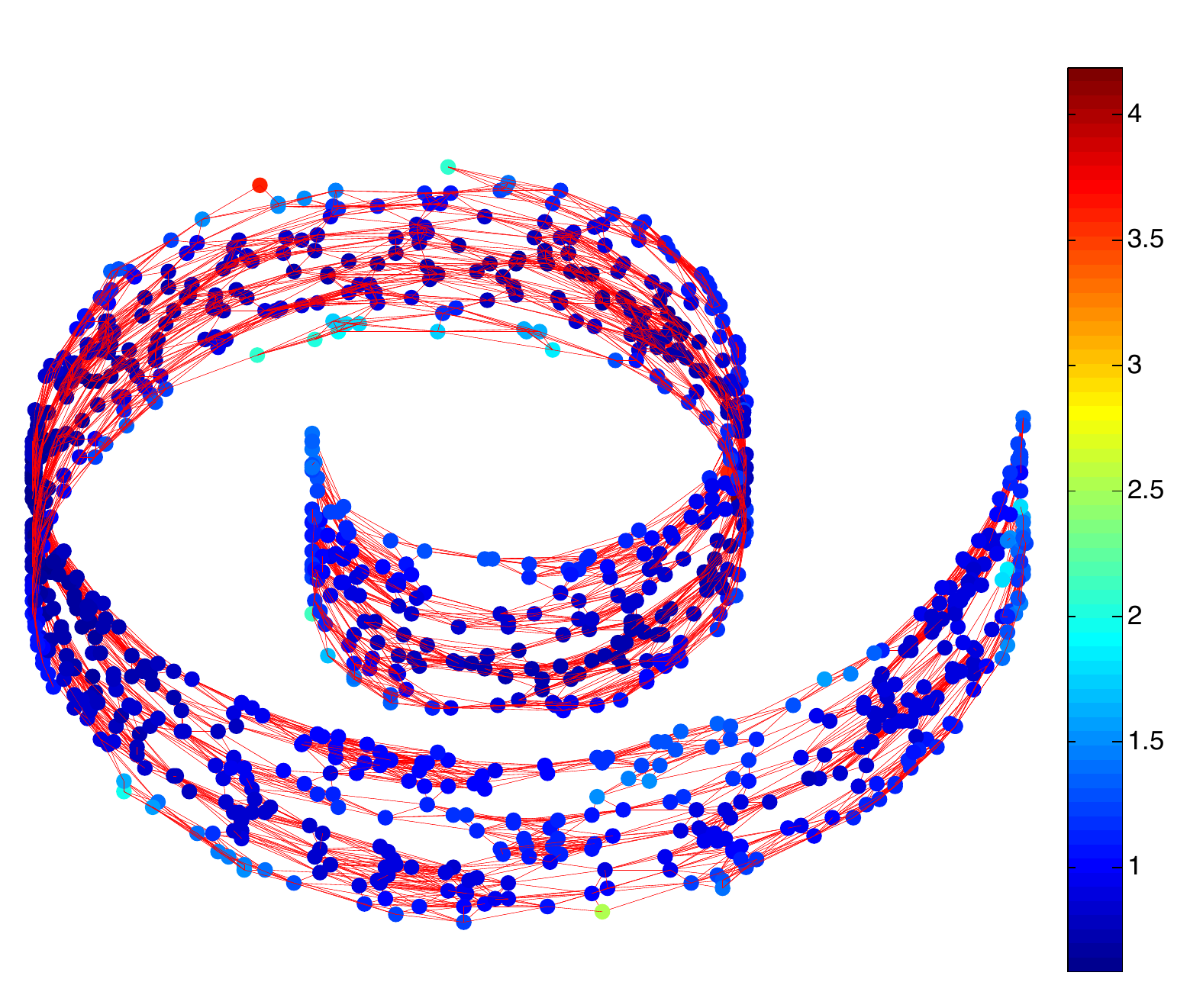}} 
\centerline{\small{(c)}~~~~}
\end{minipage} \hfill \\ 
\hfill
\begin{minipage}[b]{.31\linewidth}
   \centering
   \centerline{\includegraphics[width=\linewidth]{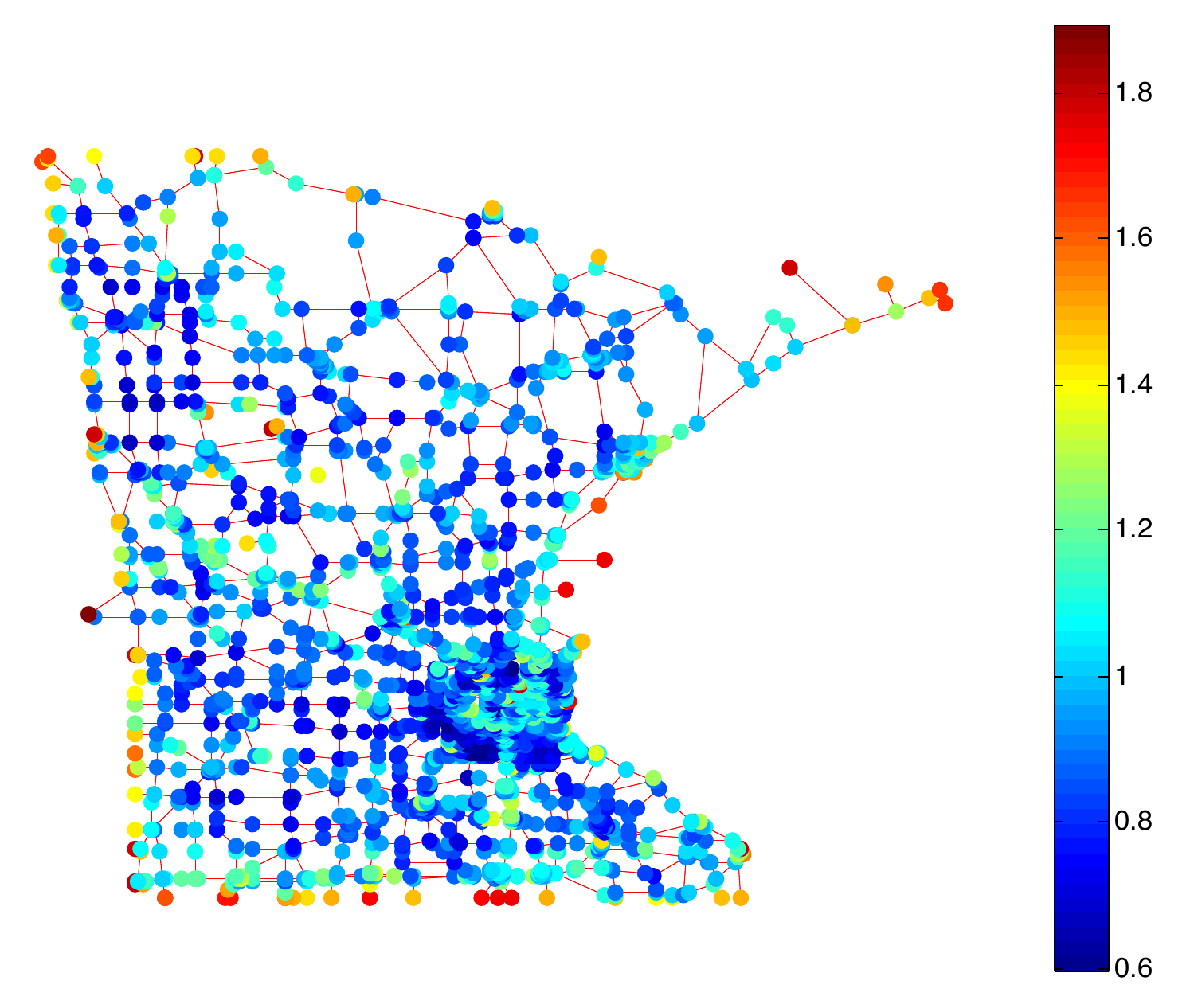}} 
\centerline{\small{(d)}~~}
\end{minipage}
\hfill
\begin{minipage}[b]{.31\linewidth}
   \centering
   \centerline{\includegraphics[width=\linewidth]{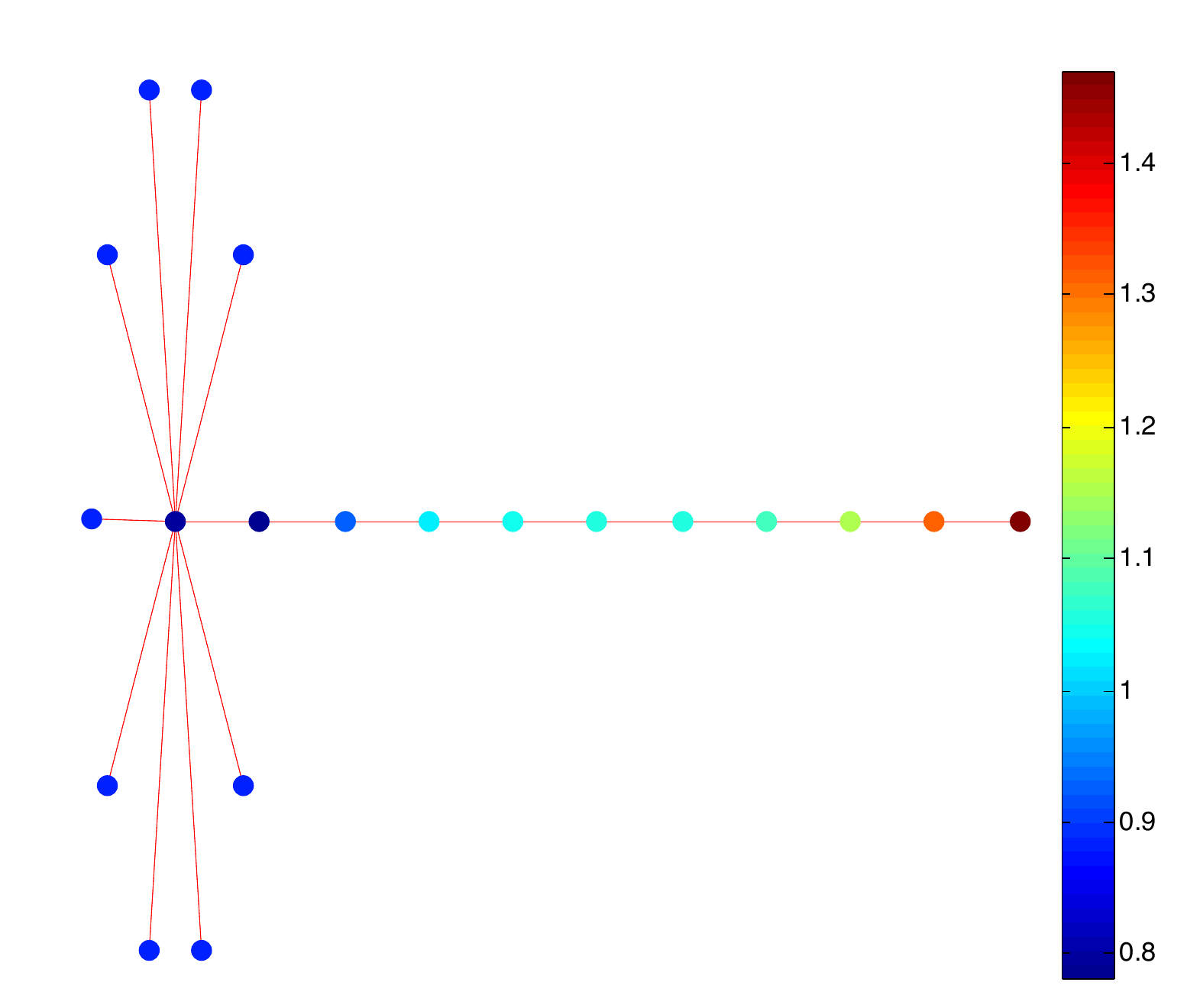}} 
\centerline{\small{(e)}~~~~}
\end{minipage}
\hfill
\begin{minipage}[b]{.31\linewidth}
   \centering
   \centerline{\includegraphics[width=\linewidth]{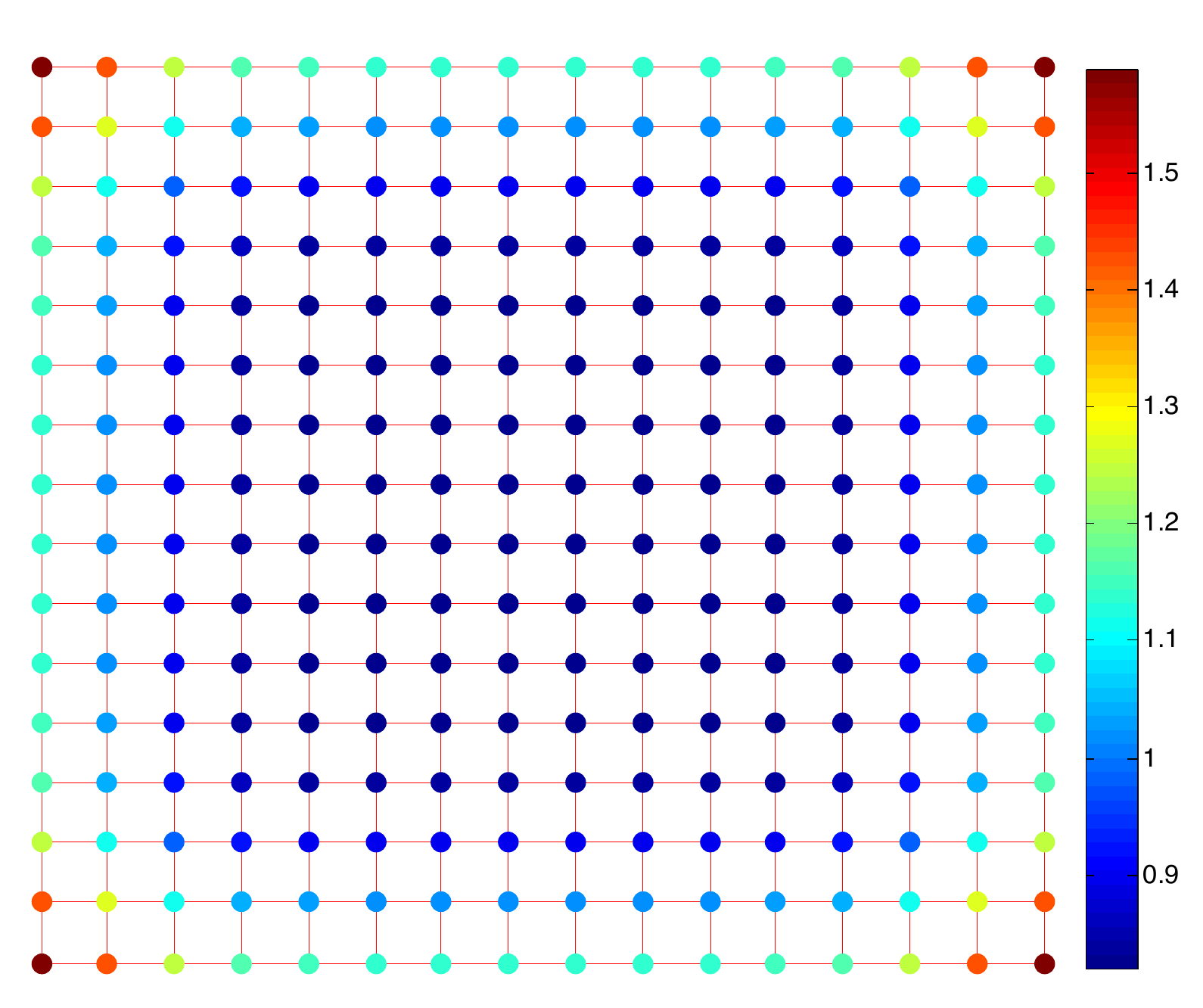}} 
\centerline{\small{(f)}~~~~}
\end{minipage} \hfill
\caption {Norms of translated normalized heat kernels with $\tau=2$. (a) A normalized heat kernel $\hat{f}(\lambda_{\l})=Ce^{-2\lambda_{\l}}$ on the sensor network graph shown in (b). (b)-(f) The value at each vertex $i$ represents $\norm{T_i f}_2$. The edges of the graphs in (b) and (c) are weighted by a thresholded Gaussian kernel weighting function based on the physical distance between nodes \eqref{Eq:gkw}, whereas the edges of the graphs in (d)-(f) all have weights equal to one. In all cases, the norms of the translated windows are not too close to zero, and the larger norms tend to be located at the ``boundary'' vertices in the graph. The lower bound $|\hat{f}(0)|$ and upper bound $\sqrt{N}\mu\norm{f}_2$ of Lemma \ref{Le:trans_norm_bounds} are (b) [0.27,21.38]; (c) [0.20,29.22]; (d) [0.07,42.88]; (e) [0.62,4.26]; (f) [0.36,3.25].} 
  \label{Fig:trans_norms_low}
\end{figure}
\begin{figure}[h]
\hfill
\begin{minipage}[b]{.44\linewidth}
   \centering
   \centerline{\includegraphics[width=\linewidth]{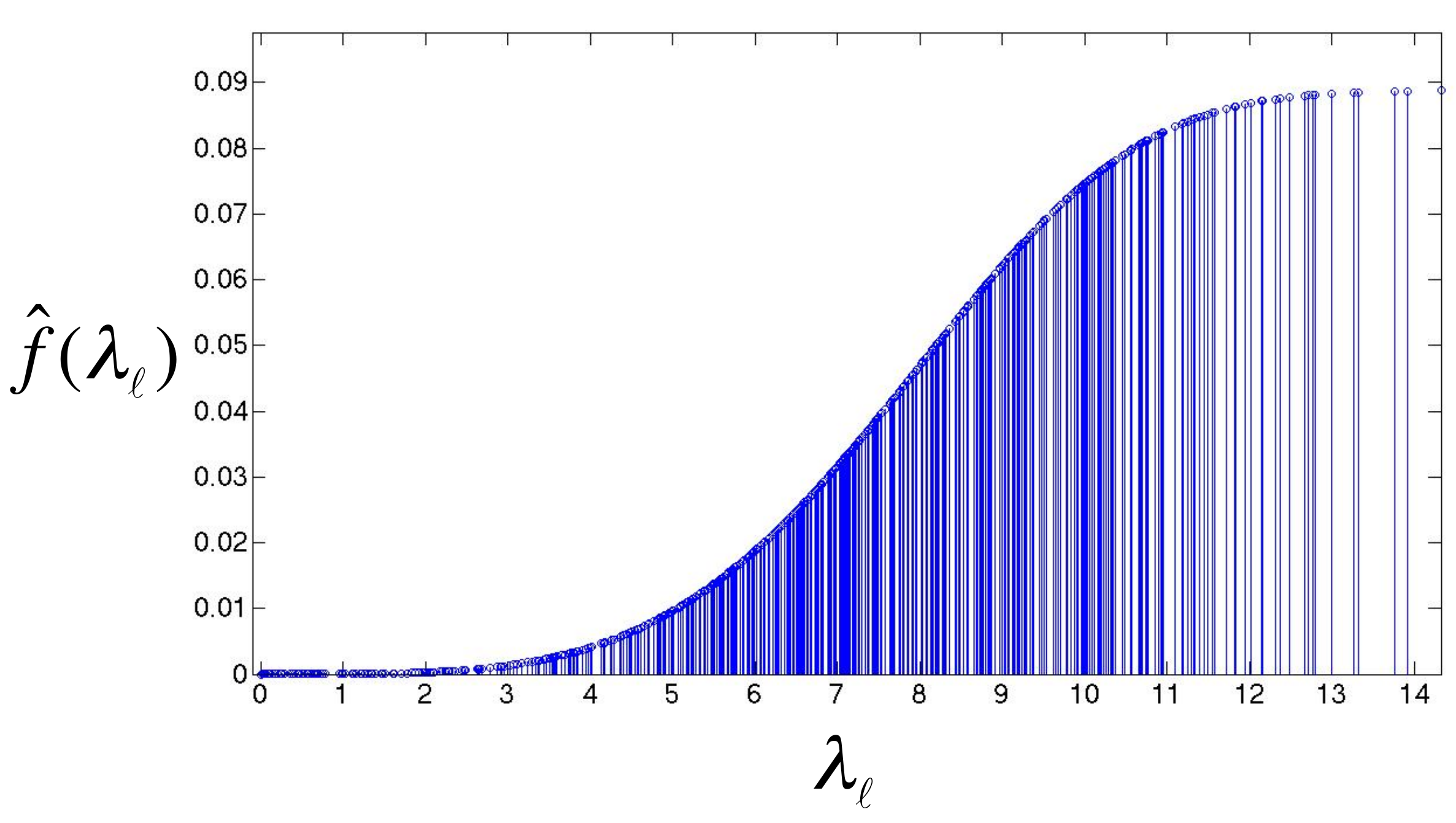}} 
\centerline{\small{~~~~~~~~~~(a)}}
\end{minipage}
\hfill
\begin{minipage}[b]{.31\linewidth}
   \centering
   \centerline{\includegraphics[width=\linewidth]{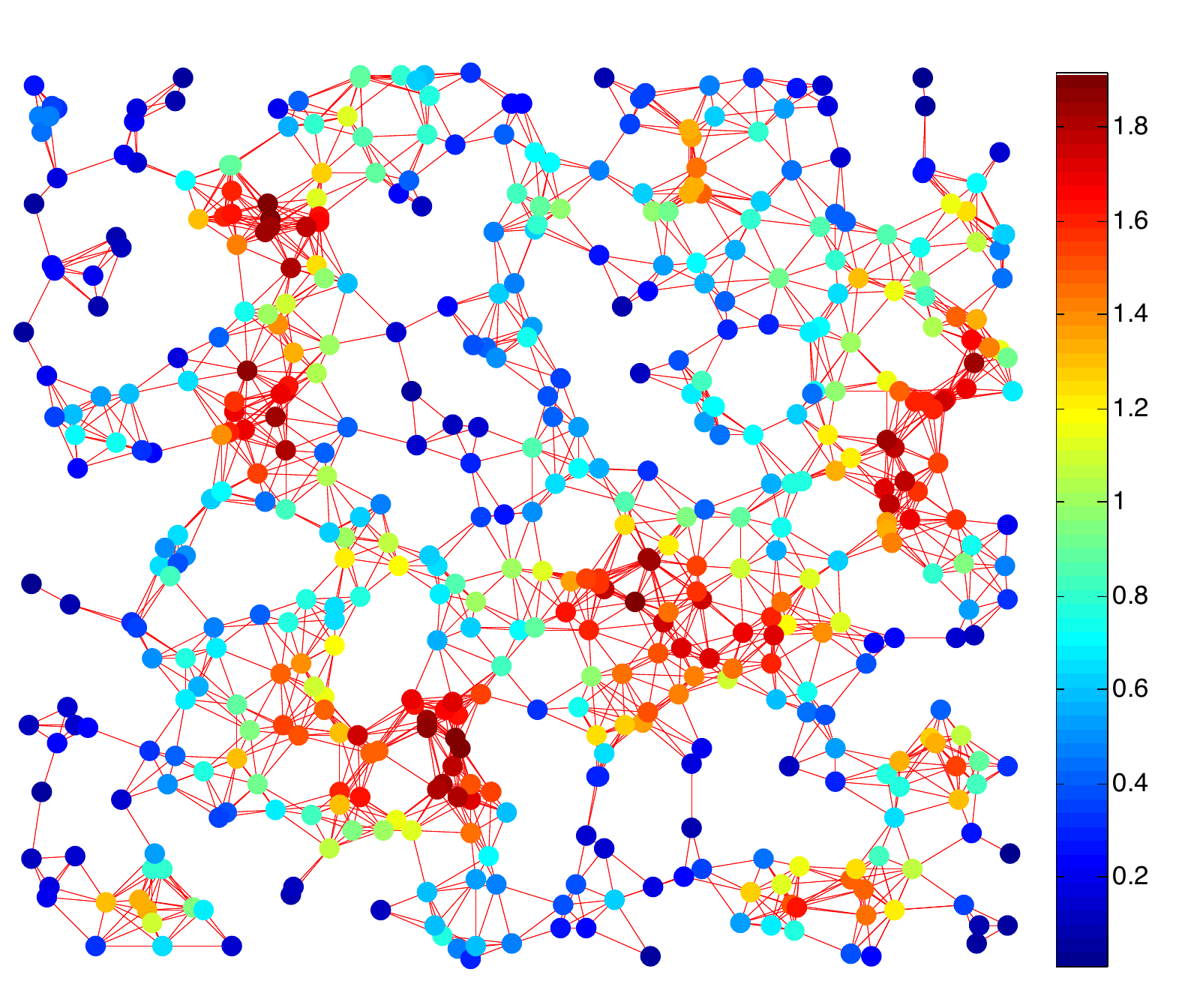}} 
\centerline{\small{(b)}}
\end{minipage} \hfill \hfill
\caption {(a) A smooth kernel with its energy concentrated on the higher frequencies of the sensor network's spectrum. (b) The values of $\norm{T_i f}_2$. Unlike the normalized heat kernels of Figure \ref{Fig:trans_norms_low}, the translated kernels may have norms close to zero. In this example, $\norm{f}_2=1$ and the minimum norm of a translated window is 0.013.} 
  \label{Fig:trans_norms_high}
\end{figure}

\subsection{Localization of Translated Kernels in the Vertex Domain} \label{Se:trans_loc}
We now examine to what extent translated kernels are localized in the vertex domain. First, we note 
that a polynomial kernel with degree $K$ that is translated to a given center vertex is strictly localized in a ball of radius $K$ around the center vertex, where the distance $d_{\G}(\cdot,\cdot)$ used to define the ball is the \emph{geodesic} or \emph{shortest path} distance (i.e., the distance between two vertices is the minimum number of edges in any path connecting them). Note that this choice of distance measure ignores the weights of the edges and only depends on the unweighted adjacency matrix of the graph.
\begin{lemma} \label{Le:Lap_power}
Let $\widehat{p_K}$ be a polynomial kernel with degree $K$; i.e., 
\begin{align}\label{Eq:poly_kern}
\widehat{p_K}(\l)=\sum_{k=0}^K a_k \lambda_{\l}^k 
\end{align}
for some coefficients $\{a_k\}_{k=0,1,\ldots,K}$. If $d_{\G}(i,n)>K$, then $(T_i p_K)(n) = 0$.
\end{lemma}
\begin{proof}
By \cite[Lemma 5.2]{sgwt}, $d_{\G}(i,n)>K$ implies $(\L^K)_{i,n}=0$. Combining this with the fact that
\begin{align*}
(\L^k)_{i,n}=\sum_{\l=0}^{N-1} \lambda_{\l}^k \chi_{\l}^*(i)\chi_{\l}(n),
\end{align*}
and with the definitions \eqref{Eq:new_translation}
 and \eqref{Eq:poly_kern} of the generalized translation and the polynomial kernel, we have
\begin{align*}
\left(T_i p_K\right)(n) &= \sqrt{N}\sum_{\l=0}^{N-1}\widehat{p_K}(\lambda_{\l})\chi_{\l}^*(i)\chi_{\l}(n) \\
&= \sqrt{N}\sum_{\l=0}^{N-1}\sum_{k=0}^K a_k \lambda_{\l}^k \chi_{\l}^*(i)\chi_{\l}(n) \\
&= \sqrt{N} \sum_{k=0}^K a_k (\L^k)_{i,n} = 0,~\forall~i,n \in \V \hbox{ s.t. } d_{\G}(i,n)>K. 
\end{align*}
\end{proof}

More generally, as seen in Figure \ref{Fig:trans}, if we translate a smooth kernel to a given center vertex $i$, the magnitude of the translated kernel at another vertex $n$ decays as the distance between $i$ and $n$ increases.
In Theorem \ref{Th:trans_loc}, we provide one estimate of this localization by combining the strict localization of polynomial kernels
with
the following 
upper bound on the minimax polynomial approximation error. 

\begin{lemma}[{{\cite[Equation (4.6.10)]{atkinson}}}] 
If a function $f(x)$ is $(K+1)$-times continuously differentiable on an interval $[a,b]$, then 
\begin{align}\label{Eq:atkinson}
\inf_{q_K}\bigl\{\norm{f-q_K}_{\infty}\bigr\} \leq \frac{\left[\frac{b-a}{2}\right]^{K+1}}{(K+1)!~2^K}\norm{f^{(K+1)}}_{\infty},
\end{align}
where the infimum in \eqref{Eq:atkinson} is taken over all polynomials $q_K$ of degree $K$.
\end{lemma}

\begin{theorem} \label{Th:trans_loc}
Let 
$\hat{g}:[0,\lambda_{\max}]\rightarrow \Rbb$ be a kernel, 
 and define $d_{in}:=d_{\G}(i,n)$ and $K_{in}:=d_{in}-1$. 
 Then 
\begin{align} \label{Eq:loc_bound0}
|(T_i g)(n)| \leq \sqrt{N} \inf_{\widehat{p_{K_{in}}}}\left\{\sup_{\lambda \in [0,\lambda_{\max}]} \left|\hat{g}(\lambda)-\widehat{p_{K_{in}}}(\lambda)
\right|\right\}=\sqrt{N} \inf_{\widehat{p_{K_{in}}}} \left\{ \norm{\hat{g}-\widehat{p_{K_{in}}}
}_{\infty}\right\}, 
\end{align}
where the infimum is taken over all polynomial kernels of degree $K_{in}$, as defined in \eqref{Eq:poly_kern}. 
More generally, for $p,q \geq 1$ such that $\frac{1}{p}+\frac{1}{q}=1$, 
\begin{align} \label{Eq:loc_bound0}
|(T_i g)(n)| \leq \sqrt{N} \mu^{\frac{2(q-1)}{q}} \inf_{\widehat{p_{K_{in}}}} \left\{ \norm{\hat{g}-\widehat{p_{K_{in}}}
}_{p}\right\}.
\end{align}
%
%
\end{theorem}

\begin{proof}
From Lemma \ref{Le:Lap_power}, $(T_i p_{K_{in}})(n) = 0$ for all polynomial kernels $\widehat{p_{K_{in}}}$ of degree $K_{in}$. Thus, we have
\begin{align} \label{Eq:ratio_num_bounda}
|(T_i g)(n)|=\inf_{\widehat{p_{K_{in}}}}|(T_i g)(n)-(T_i p_{K_{in}})(n)| &= \inf_{\widehat{p_{K_{in}}}} \left|\sqrt{N}\sum_{\l=0}^{N-1}\left[\hat{g}(\lambda_{\l})-\widehat{p_{K_{in}}}(\lambda_{\l})\right]\chi_{\l}^*(i)\chi_{\l}(n)\right| \nonumber \\
& \leq \sqrt{N} \inf_{\widehat{p_{K_{in}}}} \left\{ \sum_{\l=0}^{N-1}\Bigl|\left[\hat{g}(\lambda_{\l})-\widehat{p_{K_{in}}}(\lambda_{\l})\right]\chi_{\l}^*(i)\chi_{\l}(n)\Bigr| \right\} \nonumber \\
& \leq \sqrt{N} \inf_{\widehat{p_{K_{in}}}} \left\{ \left(\sum_{\l=0}^{N-1}\left|\hat{g}(\lambda_{\l})-\widehat{p_{K_{in}}}(\lambda_{\l})\right|^p \right)^{\frac{1}{p}} \left(\sum_{\l=0}^{N-1}\left|\chi_{\l}^*(i)\chi_{\l}(n)\right|^q\right)^{\frac{1}{q}} \right\} \\
& \leq \sqrt{N} \mu^{\frac{2(q-1)}{q}}\inf_{\widehat{p_{K_{in}}}} \left\{ \norm{\hat{g}-\widehat{p_{K_{in}}}
}_{p}\right\}, \label{Eq:ratio_num_bound} 
\end{align}
where \eqref{Eq:ratio_num_bounda} follows from H\"{o}lder's inequality, and \eqref{Eq:ratio_num_bound} follows from
\begin{align*}
\sum_{\l=0}^{N-1} \left| \chi_{\l}^*(i)\chi_{\l}(n) \right|^q \leq \mu^{2(q-1)} \sum_{\l=0}^{N-1} \left| \chi_{\l}^*(i)\chi_{\l}(n) \right|\leq \mu^{2(q-1)} \sqrt{\sum_{\l=0}^{N-1} \left| \chi_{\l}(i)\right|^2}\sqrt{\sum_{\l=0}^{N-1} \left| \chi_{\l}(n)\right|^2 } =\mu^{2(q-1)} .
\end{align*}

\end{proof}
Substituting  \eqref{Eq:atkinson} into \eqref{Eq:ratio_num_bound} yields the following corollary to Theorem \ref{Th:trans_loc}.
\begin{corollary}\label{Co:trans_loc_diff}
If $\hat{g}(\cdot)$ is $d_{in}$-times continuously differentiable on $[0,\lambda_{\max}]$, then 
\begin{align} \label{Eq:dis_loc_bound0}
{|(T_i g)(n)|} \leq \left[
\frac{2\sqrt{N}}{d_{in}!}\left(\frac{\lambda_{\max}}{4}\right)^{d_{in}} \right] \sup_{\lambda \in [0,\lambda_{\max}]} \left|\hat{g}^{(d_{in})}(\lambda)\right|.
\end{align}
\end{corollary}

When the kernel $\hat{g}(\cdot)$ has a significant DC component, as is the case for the most logical candidate window functions, such as the heat kernel, then we can combine Theorem \ref{Th:trans_loc} and Corollary \ref{Co:trans_loc_diff} with the lower bound on the norm of a translated kernel from Lemma \ref{Le:trans_norm_bounds} to show that the energy of the translated kernel is localized around the center vertex $i$.

\begin{corollary} \label{Co:loc_extra}
If $\hat{g}(0)\neq0$, then 
\begin{align} \label{Eq:loc_bound}
\frac{|(T_i g)(n)|}{\norm{T_i g}_2} \leq \frac{\sqrt{N}}{|\hat{g}(0)|}~\inf_{\widehat{p_{K_{in}}}}\left\{\sup_{\lambda \in [0,\lambda_{\max}]} \left|\hat{g}(\lambda)-\widehat{p_{K_{in}}}(\lambda)
\right|\right\}. 
\end{align}
Moreover, if $\hat{g}(\cdot)$ is $d_{in}$-times continuously differentiable on $[0,\lambda_{\max}]$, then 
\begin{align} \label{Eq:dis_loc_bound}
\frac{|(T_i g)(n)|}{\norm{T_i g}_2} \leq \left[
\frac{2\sqrt{N}}{d_{in}!~|\hat{g}(0)|}\left(\frac{\lambda_{\max}}{4}\right)^{d_{in}} \right] \sup_{\lambda \in [0,\lambda_{\max}]} \left|\hat{g}^{(d_{in})}(\lambda)\right|.
\end{align}
\end{corollary}

\begin{example} If $\hat{g}(\lambda)=e^{-\tau \lambda}$, then $\hat{g}(0)=1$, $\sup_{\lambda \in [0,\lambda_{\max}]} \left|\hat{g}^{(d_{in})}(\lambda)\right|=\tau^{d_{in}}$, and \eqref{Eq:dis_loc_bound} becomes
\begin{align} \label{Eq:heat_kernel_spread_bound}
\frac{|(T_i g)(n)|}{\norm{T_i g}_2} \leq 
\frac{2\sqrt{N} }{d_{in}!} \left(\frac{\tau \lambda_{\max}}{4}\right)^{d_{in}} \leq \sqrt{\frac{2N}{d_{in} \pi}} e^{-\frac{1}{12d_{in}+1}} \left(\frac{\tau \lambda_{\max} e}{4 d_{in}}\right)^{d_{in}},
\end{align}
where the second inequality follows from Stirling's approximation: $m! \geq \sqrt{2\pi m} \left(\frac{m}{e}\right)^m e^{\frac{1}{12m+1}}$ (see, e.g. \cite[p.~257]{abram}). 
Interestingly, a term of the form $\left(\frac{C \tau e}{d_{in}}\right)^{d_{in}}$ also appears in the upper bound of \cite[Theorem 1]{metzger}, which is specifically tailored to the heat kernel and derived via a rather different approach.

Agaskar and Lu \cite{agaskar_spie,agaskar_icassp} define the graph spread of a signal $f$ around a given vertex $i$ as 
\begin{align*}
\Delta_i^2(f):=\frac{1}{\norm{f}_2^2} \sum_{n=1}^N[d_{\G}(i,n)]^2 [f(n)]^2.
\end{align*} 
We can now use \eqref{Eq:heat_kernel_spread_bound} to bound the spread of a translated heat kernel around the center vertex $i$ in terms of the diffusion parameter $\tau$ as follows:
\begin{align}\label{Eq:heat_spread_vertex1}
\Delta_i^2(T_i g)&~=\frac{1}{\norm{T_i g}_2^2} \sum_{n=1}^N d_{in}^2 |T_i g(n)|^2 \nonumber \\
&~= \sum_{n\neq i} d_{in}^2 \left[\frac{|T_i g(n)|}{\norm{T_i g}_2}\right]^2 \nonumber \\
& \stackrel{\eqref{Eq:heat_kernel_spread_bound}}\leq 4N \sum_{n\neq i} \frac{1}{[(d_{in}-1)!]^2} \left(\frac{\tau^2 \lambda_{\max}^2}{16}\right)^{d_{in}} \nonumber \\
&~= \frac{N \tau^2 \lambda_{\max}^2}{4} \sum_{r=1}^{diam(\G)} \left|{\cal R}(i,r)\right| \frac{1}{[(r-1)!]^2} \left(\frac{\tau^2 \lambda_{\max}^2}{16}\right)^{r-1}, 
\end{align}
where $\left|{\cal R}(i,r)\right|$ is the number of vertices whose distance from vertex $i$ is exactly $r$. Note that
\begin{align} \label{Eq:heat_spread_vertex2}
\left|{\cal R}(i,r)\right| \leq d_i(d_{\max}-1)^{r-1},
\end{align}
so, substituting \eqref{Eq:heat_spread_vertex2} into \eqref{Eq:heat_spread_vertex1}, we have
\begin{align}\label{Eq:heat_spread_vertex3}
\Delta_i^2(T_i g) &\leq \frac{N \tau^2 \lambda_{\max}^2}{4} \sum_{r=1}^{diam(\G)} \frac{d_i}{[(r-1)!]^2} \left(\frac{\tau^2 \lambda_{\max}^2}{16(d_{\max}-1)}\right)^{r-1} \nonumber \\
&= \frac{N \tau^2 \lambda_{\max}^2 d_i}{4} \sum_{r=0}^{diam(\G)-1} \frac{1}{[(r)!]^2} \left(\frac{\tau^2 \lambda_{\max}^2}{16(d_{\max}-1)}\right)^{r} \nonumber \\
&\leq \frac{N \tau^2 \lambda_{\max}^2 d_i}{4} \sum_{r=0}^{\infty} \frac{1}{r!} \left(\frac{\tau^2 \lambda_{\max}^2}{16(d_{\max}-1)}\right)^{r}  \nonumber \\
& = \frac{N \tau^2 \lambda_{\max}^2 d_i}{4} \exp\left(\frac{\tau^2 \lambda_{\max}^2}{16(d_{\max}-1)}\right),
\end{align}
where the final equality in \eqref{Eq:heat_spread_vertex3} follows from the Taylor series expansion of the exponential function around zero.
While the above bounds 
can be quite loose (in particular for weighted graphs as we have not incorporated the graph weights into the bounds), 
the analysis nonetheless shows that we can control the spread of translated heat kernels around their center vertices through the diffusion parameter $\tau$.
For any $\epsilon > 0$, in order to ensure $\Delta_i^2(T_i g) \leq \epsilon$, it suffices to take 
\begin{align}\label{Eq:Lambert}
\tau \leq \frac{4}{\lambda_{\max}}\sqrt{(d_{\max}-1)~\Omega\left(\frac{\epsilon}{4Nd_i(d_{\max}-1)}\right)},
\end{align}
where $\Omega(\cdot)$ is the Lambert W function. Note that the right-hand side of \eqref{Eq:Lambert} is increasing in $\epsilon$.

In the limit, as $\tau \rightarrow 0$, $T_i g \rightarrow \delta_i$, and the spread $\Delta_i^2(T_i g) \rightarrow 0$. On the other hand, as $\tau \rightarrow \infty$, $T_i g(n) \rightarrow \frac{1}{\sqrt{N}}$ for all $n$, $\norm{T_i g}_2 \rightarrow 1$, and $\Delta_i^2(T_i g) \rightarrow \frac{1}{N} \sum_{n=1}^{N} d_{in}^2$.
\end{example}
 
\section{Generalized Modulation of Signals on Graphs}
\label{Se:modulation}
Motivated by the fact that the classical modulation \eqref{Eq:classical_modulation} is a multiplication by a Laplacian eigenfunction, we define, for any $k \in \{0,1,\ldots,N-1\}$, a \emph{generalized modulation operator} $M_{k}: \Rbb^N \rightarrow \Rbb^N$ by
\begin{eqnarray}\label{Eq:mod_def1}
\left(M_{k}f\right)(n):=\sqrt{N}f(n)\chi_{k}(n).
\end{eqnarray}

\subsection{Localization of Modulated Kernels in the Graph Spectral Domain}
Note 
first that, for connected graphs, $M_0$ is the identity operator, as $\chi_{0}(n)=\frac{1}{\sqrt{N}}$ for all $n$. 
In the classical case, the modulation operator represents a translation in the Fourier domain:
\begin{align*}
\widehat{M_{\xi}f}(\omega)=\hat{f}(\omega-\xi), \forall \omega \in \Rbb.
\end{align*}
This property is not true in general for our modulation operator on graphs due to the discrete nature of the graph. However, we do have the nice property that if $\hat{g}(\l)=\delta_0(\lambda_{\l})$, then 
\begin{align*}
\widehat{M_k g}(\lambda_{\l}) &=\sum_{n=1}^N \chi_{\l}^*(n) (M_k g)(n) \\
&=  \sum_{n=1}^N \chi_{\l}^*(n) \sqrt{N} \chi_k(n) \frac{1}{\sqrt{N}} =\delta_0(\lambda_{\l}-\lambda_k)=
\begin{cases}
1,&\hbox{ if }\lambda_{\l}=\lambda_k \\
0,&\hbox{ otherwise}
\end{cases}
, 
\end{align*}
so $M_k$ maps the DC component of any signal ${f} \in \Rbb^N$ to 
$\hat{f}(0)\chi_k$.
Moreover, if we start with a function ${f}$ that is localized around the eigenvalue 0 in the graph spectral domain, as in Figure \ref{Fig:mod}, then $M_k {f}$ 
will be localized around the eigenvalue $\lambda_k$ in the graph spectral domain. 
\begin{figure}[h]
\centering
{\hfill
\begin{minipage}[b]{.4\linewidth}
   \centering
   \centerline{\includegraphics[width=.92\linewidth]{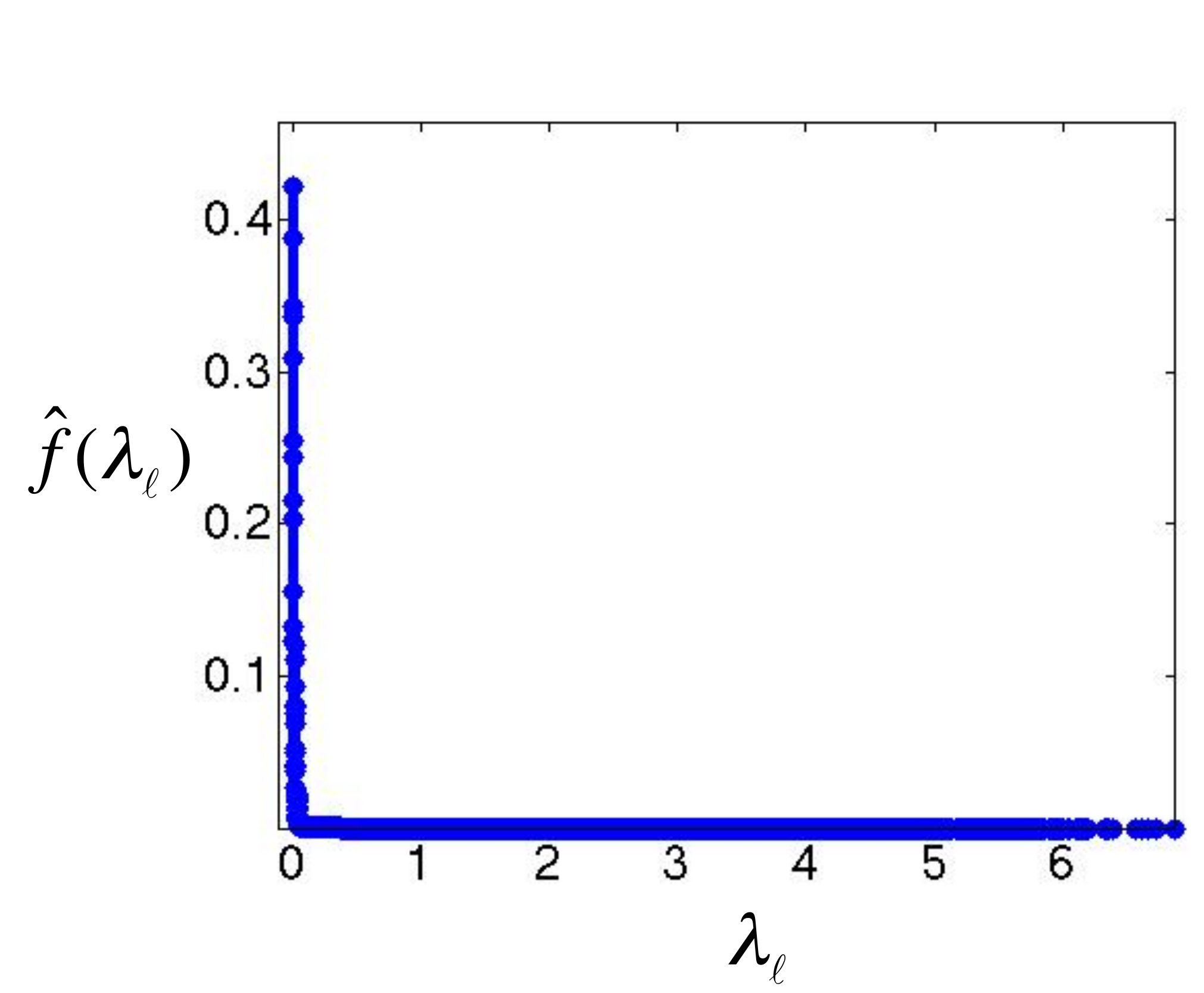}} 
\centerline{\small{~~~~~~~~~~~~~~(a)}}
\end{minipage}
\hfill
\begin{minipage}[b]{.45\linewidth}
   \centering
   \centerline{\includegraphics[width=\linewidth]{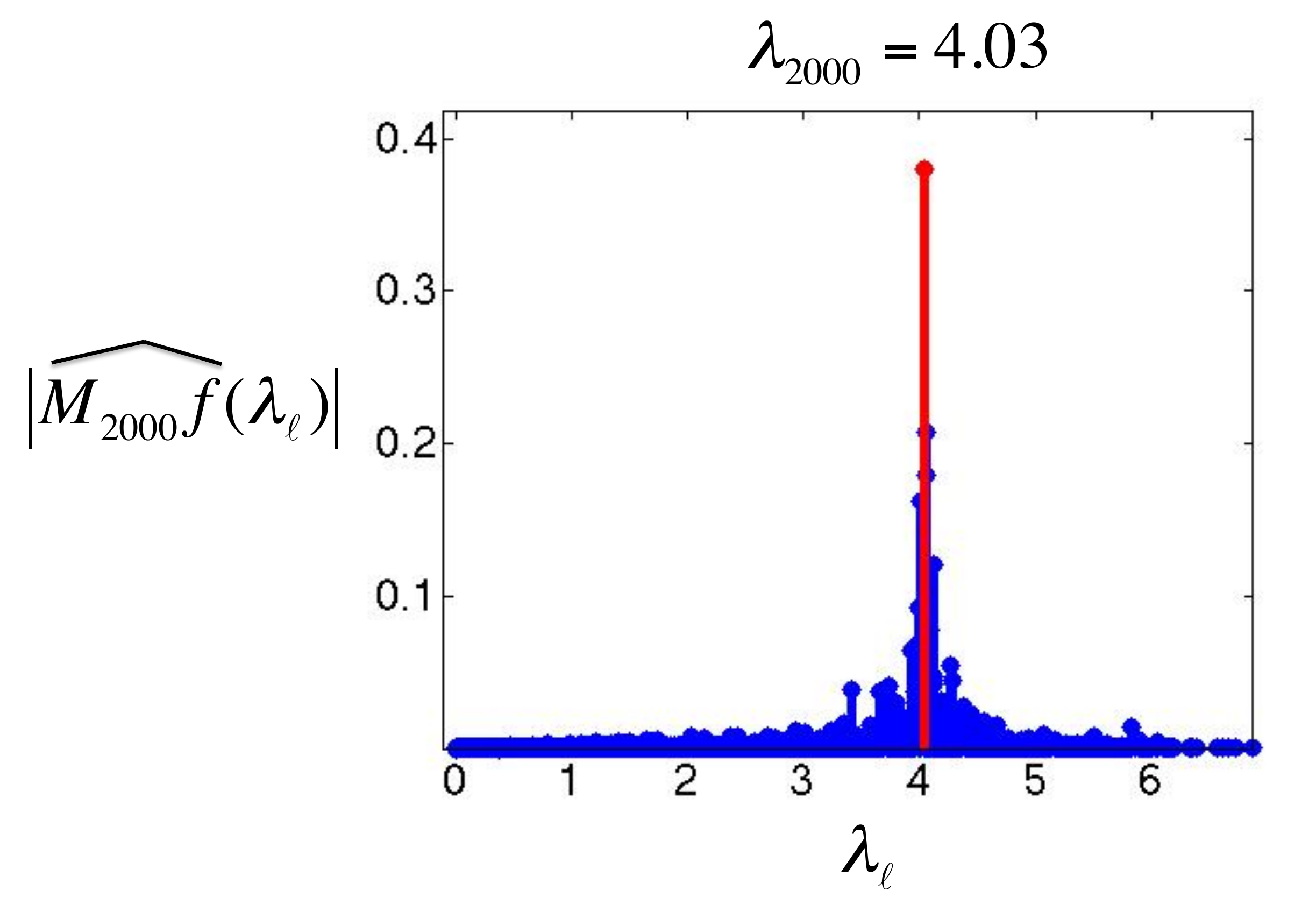}} 
\centerline{\small{~~~~~~~~~~~~~~~~~~~~~~(b)}}
\end{minipage}
\hfill}
\caption {(a) The graph spectral representation of a signal $f$ with $\hat{f}(\l)=Ce^{-100 \lambda_{\l}}$ on the Minnesota graph, where the constant $C$ is chosen such that $\norm{f}_2=1$. 
 (b) The graph spectral representation $\widehat{M_{2000} f}$ of the modulated signal $M_{2000} f$. Note that the modulated signal is localized around $\lambda_{2000}=4.03$ in the graph spectral domain.} 
  \label{Fig:mod}
\end{figure}

We quantify this localization in the next theorem, which is an improved version of \cite[Theorem 1]{shuman_SSP_2012}.
\begin{theorem} \label{Th:mod_trans}
Given a weighted graph $\G$ with $N$ vertices, 
if for some $\gamma > 0$, a 
kernel $\hat{f}$ 
satisfies 
\begin{align} \label{Eq:gsum_cond}
\sqrt{N}\sum_{\l=1}^{N-1}{\mu_{\l}|\hat{f}(\lambda_{\l})|} \leq \frac{|\hat{f}(0)|}{1+\gamma},
\end{align}
 then
 \begin{align}\label{Eq:mod_trans_result}
 |\widehat{M_k f}(\lambda_{k})| \geq \gamma |\widehat{M_k f}(\lambda_{\l})|~\hbox{ for all }\l\neq k.
 \end{align}
\end{theorem}
\begin{proof} 
\begin{align}\label{Eq:mod_1}
\widehat{M_k f}(\lambda_{\l^{\prime}}) 
=\sum_{n=1}^N \sqrt{N} \chi_{\l^{\prime}}^*(n)\chi_k(n)f(n) 
&
=\sum_{n=1}^N \sqrt{N} \chi_{\l^{\prime}}^*(n) \chi_k(n)\sum_{\l^{\prime \prime}=0}^{N-1}\chi_{{\l^{\prime \prime}}}(n) \hat{f}(\lambda_{\l^{\prime \prime}}) \nonumber \\
&
=\sum_{n=1}^N \sqrt{N} \chi_{\l^{\prime}}^*(n) \chi_k(n) \left[\frac{\hat{f}(0)}{\sqrt{N}}+\sum_{\l^{\prime \prime}=1}^{N-1}\chi_{\l^{\prime \prime}}(n) \hat{f}(\lambda_{\l^{\prime \prime}})\right] \nonumber \\
&=\hat{f}(0) \delta_{\l^{\prime} k} + \sum_{n=1}^N \sqrt{N} \chi_{\l^{\prime}}^*(n) \chi_k(n) \sum_{\l^{\prime \prime}=1}^{N-1}\chi_{\l^{\prime \prime}}(n) \hat{f}(\lambda_{\l^{\prime \prime}}).
\end{align}
Therefore, we have
\begin{align}\label{Eq:mod_2a}
|\widehat{M_k f}(\lambda_{k})| 
&
=\left|\hat{f}(0)+ \sum_{n=1}^N \sqrt{N} |\chi_{k}(n)|^2 \sum_{\l^{\prime \prime}=1}^{N-1}\chi_{\l^{\prime \prime}}(n) \hat{f}(\lambda_{\l^{\prime \prime}})\right|  \\
&
\geq |\hat{f}(0)| - \left|\sum_{n=1}^N \sqrt{N} |\chi_{k}(n)|^2 \sum_{\l^{\prime \prime}=1}^{N-1}\chi_{\l^{\prime \prime}}(n) \hat{f}(\lambda_{\l^{\prime \prime}})\right| \nonumber \\
&
\geq |\hat{f}(0)| -\sum_{n=1}^N \sqrt{N} |\chi_{k}(n)|^2 \sum_{\l^{\prime \prime}=1}^{N-1}\left|\chi_{\l^{\prime \prime}}(n)\right| |\hat{f}(\lambda_{\l^{\prime \prime}})| \nonumber \\
&
\geq |\hat{f}(0)| - \sqrt{N}  \sum_{\l^{\prime \prime}=1}^{N-1}\mu_{\l^{\prime \prime}} |\hat{f}(\lambda_{\l^{\prime \prime}})| \nonumber \\
& \geq |\hat{f}(0)| \left(1-\frac{1}{1+\gamma}\right) \label{Eq:mod_2}
\end{align}
where the last two inequalities follow from \eqref{Eq:mu_l_def} and \eqref{Eq:gsum_cond}, respectively. 
Returning to \eqref{Eq:mod_1} for $\l \neq k$, we have
\begin{align}\label{Eq:mod_3}
\gamma | \widehat{M_k f}(\lambda_{\l})|
&=\gamma \left|\sum_{n=1}^N \sqrt{N}  \chi_{\l}^*(n)\chi_k(n)\sum_{\l^{\prime \prime}=1}^{N-1}\chi_{\l^{\prime \prime}}(n) \hat{f}(\lambda_{\l^{\prime \prime}})\right| \nonumber \\
& \leq \gamma \sum_{n=1}^N \left|\sqrt{N}  \chi_{\l}^*(n)\chi_k(n)\right|\sum_{\l^{\prime \prime}=1}^{N-1}|\chi_{\l^{\prime \prime}}(n)| |\hat{f}(\lambda_{\l^{\prime \prime}})| \nonumber \\
& \leq \gamma \sum_{n=1}^N \left|\sqrt{N}  \chi_{\l}^*(n)\chi_k(n)\right|\sum_{\l^{\prime \prime}=1}^{N-1}\mu_{\l^{\prime \prime}} |\hat{f}(\lambda_{\l^{\prime \prime}})| \nonumber \\
& \leq \gamma \sqrt{N} \sum_{\l^{\prime \prime}=1}^{N-1}\mu_{\l^{\prime \prime}} |\hat{f}(\lambda_{\l^{\prime \prime}})| \nonumber \\
& \leq |\hat{f}(0)| \left(\frac{\gamma}{1+\gamma}\right)
\end{align}
where the last three inequalities follow from \eqref{Eq:mu_l_def}, H\"{o}lder's inequality,  
and \eqref{Eq:gsum_cond}, respectively. Combining \eqref{Eq:mod_2} and \eqref{Eq:mod_3} yields \eqref{Eq:mod_trans_result}.
\end{proof}

\begin{corollary}
Given a weighted graph $\G$ with $N$ vertices, if for some $\gamma > 0$, a 
kernel $\hat{f}$ satisfies \eqref{Eq:gsum_cond}, then
\begin{align*}
\frac{|\widehat{M_k f} (\lambda_{k})|^2}{\norm{M_k f}_2^2} \geq \frac{\gamma^2}{N+3+4\gamma+\gamma^2}.
\end{align*}
\end{corollary}
\begin{proof}
We lower bound the numerator by squaring \eqref{Eq:mod_2}. We upper bound the denominator as follows:
\begin{align*}
\norm{M_k f}_2^2 = |\widehat{M_k f}(\lambda_k)|^2 + \sum_{\l\neq k}|\widehat{M_k f}(\lambda_{\l})|^2 \leq |\hat{f}(0)|^2 \left(\frac{2+\gamma}{1+\gamma}\right)^2 + (N-1) |\hat{f}(0)|^2 \left(\frac{1}{1+\gamma}\right)^2
\end{align*}
which follows from squaring \eqref{Eq:mod_3} and from applying the triangle inequality at \eqref{Eq:mod_2a}, following the same steps until \eqref{Eq:mod_2}, and squaring the result.
\end{proof}
\begin{remark}
It would also be interesting to find conditions on $\hat{f}$ such that 
\begin{align*}
\frac{\sum\limits_{\l: \lambda_{\l} \in {\cal B}(\lambda_k,r)} \left|\widehat{M_{k} f} (\lambda_{\l})\right|^2}{\norm{M_k f}_2^2} \geq \gamma_2(r),
\end{align*}
or such that the spread of $\widehat{M_k f}$ around $\lambda_k$, defined as either \cite[p.~93]{shuman_SPM}
\begin{align}\label{Eq:correct_spectral_spread}
\frac{\sum\limits_{\l=0}^{N-1} (\lambda_{\l}-\lambda_{k})^2 \left|\widehat{M_{k} f} (\lambda_{\l})\right|^2}{\norm{M_k f}_2^2} 
\hbox{~~~or~~~}
\frac{\sum\limits_{\l=0}^{N-1} (\sqrt{\lambda_{\l}}-\sqrt{\lambda_{k}})^2 \left|\widehat{M_{k} f} (\lambda_{\l})\right|^2}{\norm{M_k f}_2^2} 
\end{align}
is upper bounded.\footnote{Agaskar and Lu \cite{agaskar_spie,agaskar_icassp} suggest to define the spread of $\hat{f}$ in the graph spectral domain as $\frac{1}{\norm{f}_2^2}\sum\limits_{\l=0}^{N-1} \lambda_{\l} |\hat{f} (\lambda_{\l})|^2$. With this definition, the ``spread'' is always taken around the eigenvalue 0, as opposed to the mean of the signal, and, as a result, the signal with the highest possible spread is actually a (completely localized) Kronecker delta at $\lambda_{\max}$ in the graph spectral domain, or, equivalently, $\chi_{\max}$ in the vertex domain.} In Section \ref{Se:mod_alt} of the Appendix, we present an alternative definition of the generalized modulation and localization results regarding that modulation operator that 
contain a spread form similar to \eqref{Eq:correct_spectral_spread}.
\end{remark}

\section{Windowed Graph Fourier Frames}
\label{Se:frame}
Equipped with these generalized notions of translation and modulation of signals on graphs, we can now define windowed graph Fourier atoms and a windowed graph Fourier transform analogously  
to \eqref{Eq:classical_atom} and \eqref{Eq:classical_STFT} in the classical case. 

\subsection{Windowed Graph Fourier Atoms and a Windowed Graph Fourier Transform}
For a window $g \in \Rbb^N$, 
we define a windowed graph Fourier atom by\footnote{An alternative definition of an atom is $g_{i,k}:=T_i M_k g$. In the classical setting, the translation and modulation operators do not commute, but the difference between the two definitions of a windowed Fourier atom is a phase factor \cite[p. 6]{groechenig}. In the graph setting, it is more difficult to characterize the difference between the two definitions. In our numerical experiments, defining the atoms as $g_{i,k}:=M_k T_i g$ tended to lead to more informative analysis when using the modulation definition \eqref{Eq:mod_def1}, but defining the atoms as $g_{i,k}:=T_i {M}_k  g$ tended to lead to more informative analysis when using the alternative modulation definition presented Section \ref{Se:mod_alt}. In this paper, we always use $g_{i,k}:=M_k T_i g$.} 
\begin{align}\label{Eq:wgft_atom_comp}
g_{i,k}(n):=\left(M_{k} T_{i} g\right)(n)= 
{N} \chi_{k}(n) \sum_{\l=0}^{N-1}\hat{g}(\lambda_{\l})\chi_{\l}^*(i)\chi_{\l}(n),
\end{align} 
and the windowed graph Fourier transform of a function $f\in \Rbb^N$ by 
\begin{align}\label{Eq:wgft_comp}
Sf(i,k):=\ip{f}{g_{i,k}}. 
\end{align}
Note that as discussed in Section \ref{Se:translation}, we usually define the window directly in the graph spectral domain, as $\hat{g}(\lambda): [0,\lambda_{\max}] \rightarrow \Rbb$.

\begin{example} \label{Ex:wgft1}

We consider a random sensor network graph with $N=64$ vertices, and thresholded Gaussian kernel edge weights \eqref{Eq:gkw} with
$\sigma_1=\sigma_2=0.2$. We consider the signal $f$ shown in Figure \ref{Fig:wgft1}(a), and wish to compute the windowed graph Fourier transform coefficient $Sf(27,11)=\ip{f}{g_{27,11}}$ using a window kernel $\hat{g}(\lambda_{\l})=Ce^{-\tau\lambda_{\l}}$, where $\tau=3$ and $C=1.45$ is chosen such that $\norm{g}_2=1$. The windowed graph Fourier atom $g_{27,11}$ is shown in the vertex and graph spectral domains in Figure \ref{Fig:wgft1}(b) and \ref{Fig:wgft1}(c), respectively.

\begin{figure}[h]
\centering
\begin{minipage}[b]{.32\linewidth}
\hspace{.3in}\centerline{{$f$}} \vspace{.15in}

\centerline{\hspace{.5in}\includegraphics[width=\linewidth]{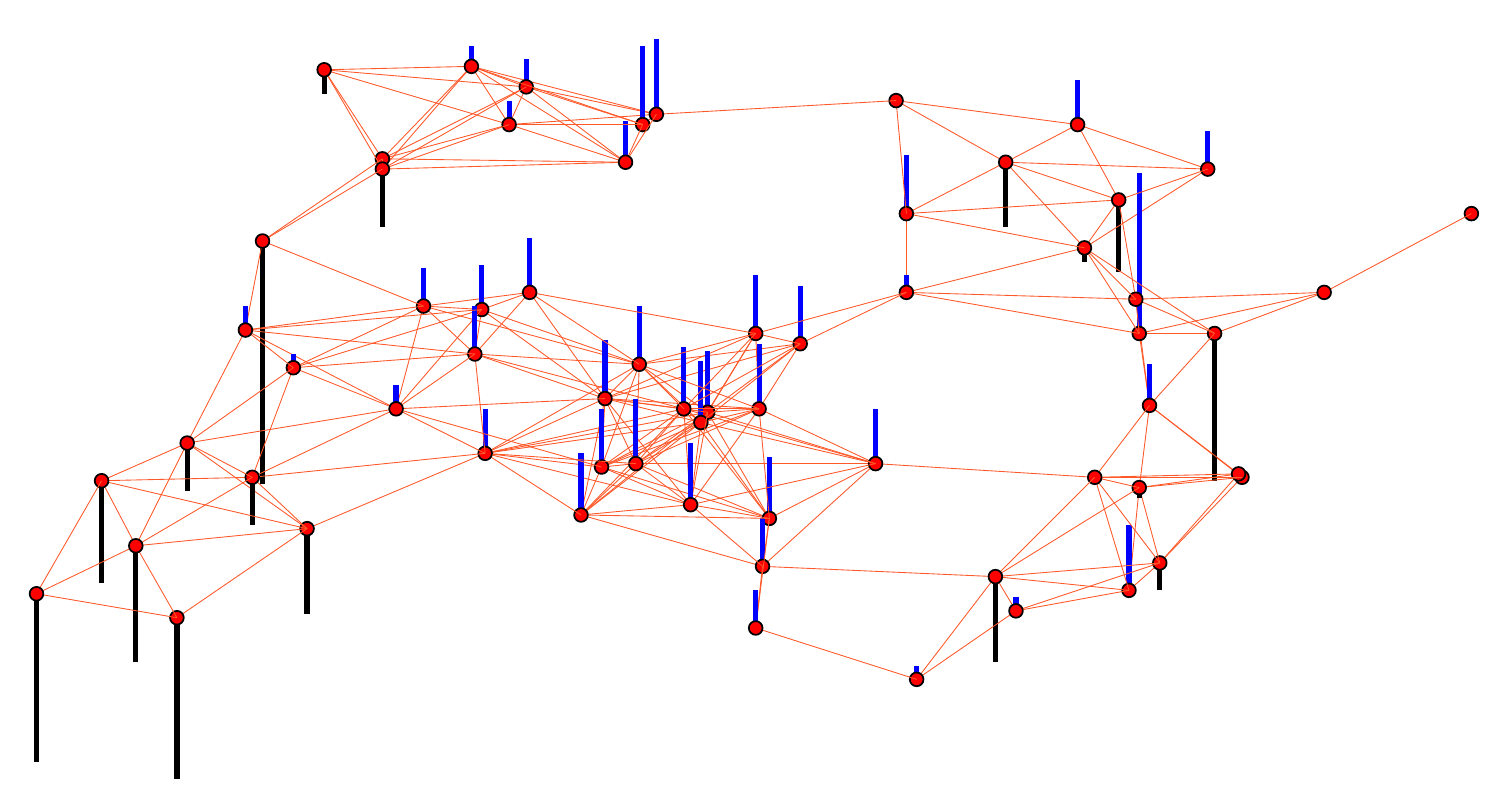}} 
\vspace{.18in}
\centerline{\small{~~~~~~~~~~~~~~(a)}}
\end{minipage} 
\hfill
\begin{minipage}[b]{.32\linewidth}
\hspace{.3in}\centerline{{$g_{27,11}$}} \vspace{.015cm}

\centerline{\hspace{.5in}\includegraphics[width=\linewidth]{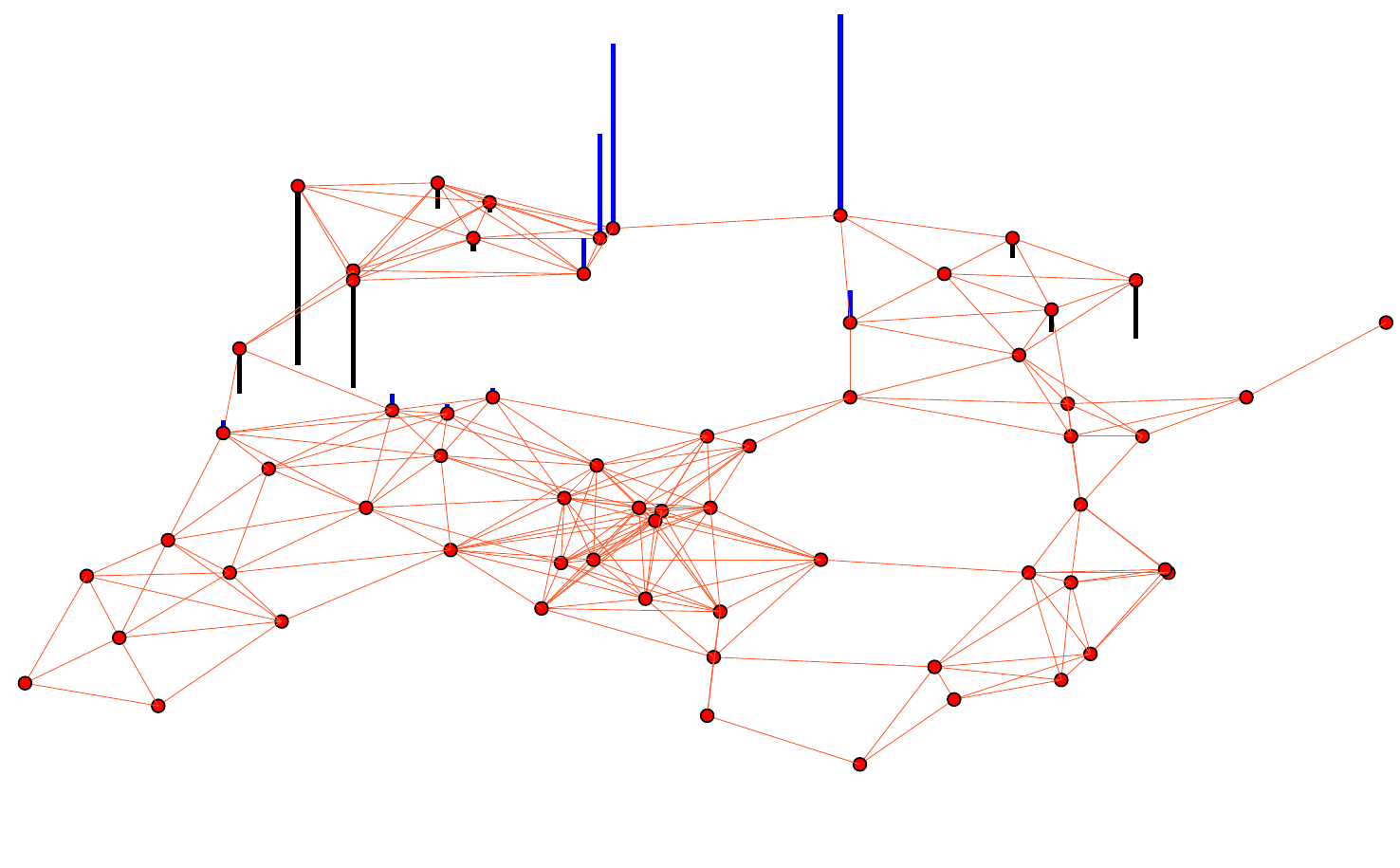}}   
\vspace{.18in}
\centerline{\small{~~~~~~~~~~~~~~(b)}}  
\end{minipage} 
\hfill
\begin{minipage}[b]{.32\linewidth}
\hspace{.3in} \centerline{{$\hat{g}_{27,11}$}}
\centerline{\hspace{.5in}\includegraphics[width=.85\linewidth]{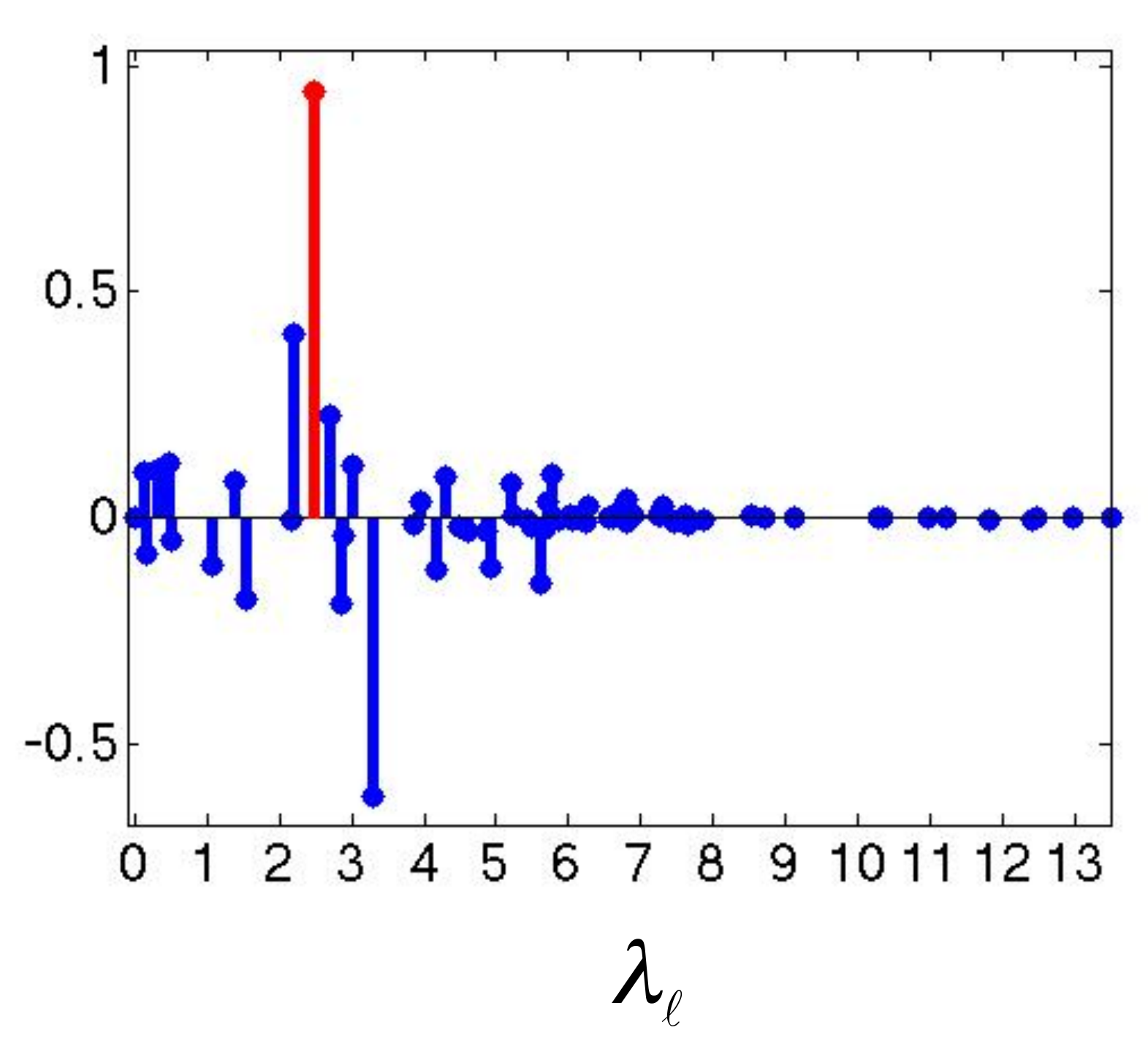}} 
\centerline{~~~~~~~~~~~~~~~\small{(c)}}
\end{minipage} 
\caption {Windowed graph Fourier transform example. (a) A signal $f$ on a random sensor network with 64 vertices. (b) A windowed graph Fourier atom $g_{27,11}$ centered at vertex 27 and frequency $\lambda_{11}=2.49$. (c) The same atom plotted in the graph spectral domain. Note that the atom is localized in both the vertex domain and the graph spectral domain. To determine the corresponding windowed graph Fourier transform coefficient, we take an inner product of the signals in (a) and (b): $S f(27,11)=\ip{f}{g_{27,11}}=0.358$.} 
 \label{Fig:wgft1}
\end{figure}

\end{example}

As with the classical windowed Fourier transform described in Section \ref{Se:classical}, we can interpret the computation of the windowed graph Fourier transform coefficients in a second manner. Namely, we can first multiply the signal $f$ by (the complex conjugate of) a translated window:
\begin{align}\label{Eq:dot_star}
f_{windowed,i}(n):=((T_i g)^* .* f)(n):=[T_i g(n)]^*[f(n)],
\end{align}
and then compute the windowed graph Fourier transform coefficients as $\sqrt{N}$ times the graph Fourier transform of the windowed signal:
\begin{align*}
\sqrt{N}\hat{f}_{windowed,i}(k) =\sqrt{N}\ip{f_{windowed,i}}{\chi_k} 
&=\sqrt{N}\sum_{n=1}^N [T_i g(n)]^*[f(n)] \chi_k^*(n) \\
&=\sum_{n=1}^N f(n) \sum_{\l=0}^{N-1} N \hat{g}(\lambda_{\l}) \chi_{\l}(i) \chi_{\l}^*(n) \chi_k^*(n) \\
&=\sum_{n=1}^N f(n) g_{ik}^*(n) = \ip{f}{g_{ik}} = Sf(i,k). 
\end{align*}
Note that the $.*$ in \eqref{Eq:dot_star} represents component-wise multiplication.

\addtocounter{example}{-1}
\begin{example}[cont.]
In Figure \ref{Fig:wgft2}, we illustrate this alternative interpretation of the windowed graph Fourier transform by first multiplying the signal $f$ of Figure \ref{Fig:wgft1}(a) by the window $T_{27}g$ (componentwise), and then taking the graph Fourier transform of the resulting windowed signal.
\begin{figure}[h]
\centering
\begin{minipage}[b]{.32\linewidth}
\hspace{.3in}\centerline{{$T_{27} g$}} 

\centerline{\hspace{.5in}\includegraphics[width=\linewidth]{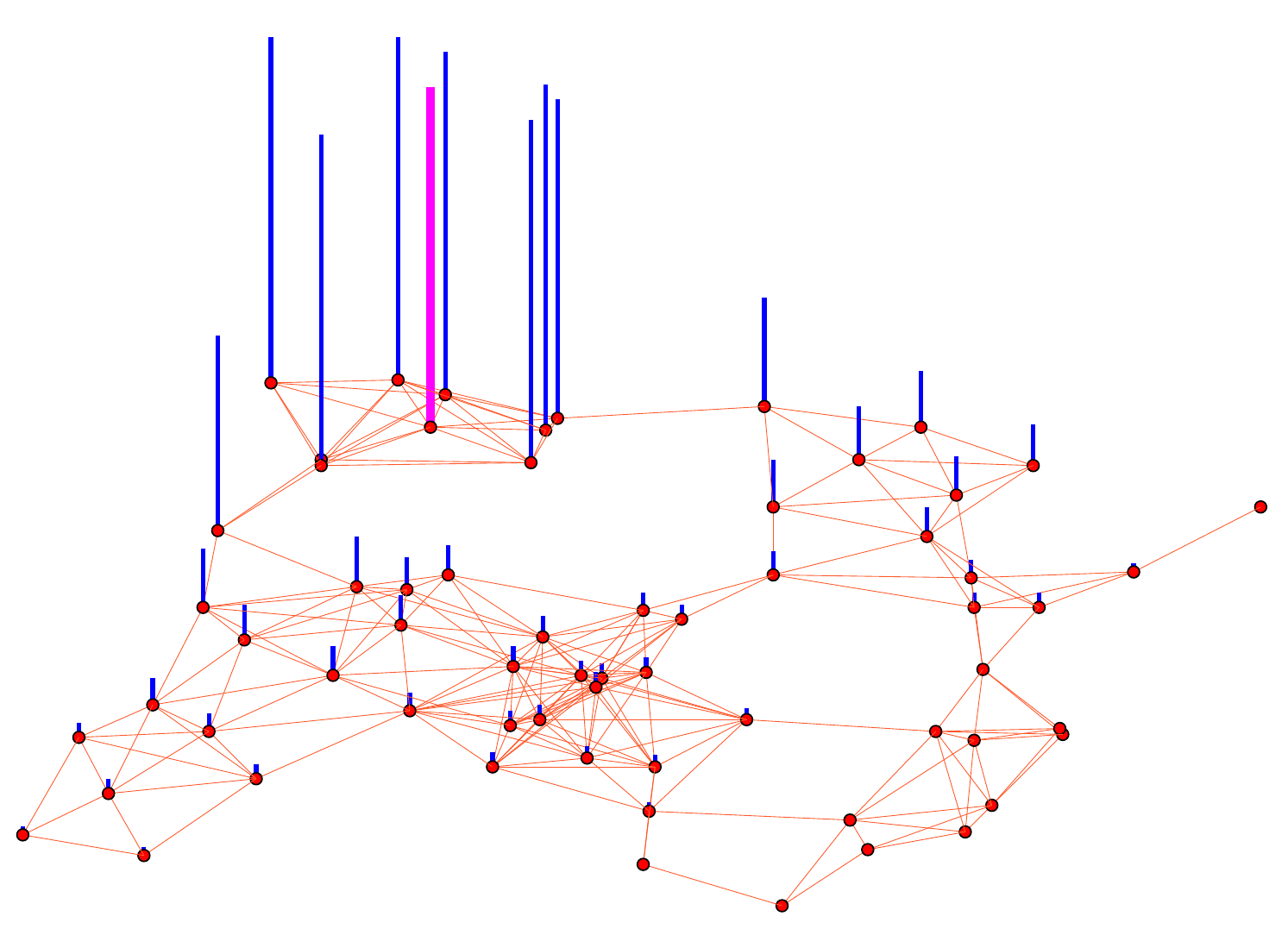}} 
\vspace{.24in}
\centerline{\small{~~~~~~~~~~~~~~(a)}}
\end{minipage} 
\hfill
\begin{minipage}[b]{.32\linewidth}
\hspace{.3in}\centerline{{$(T_{27} g).*f$}} \vspace{.1cm}

\centerline{\hspace{.5in}\includegraphics[width=\linewidth]{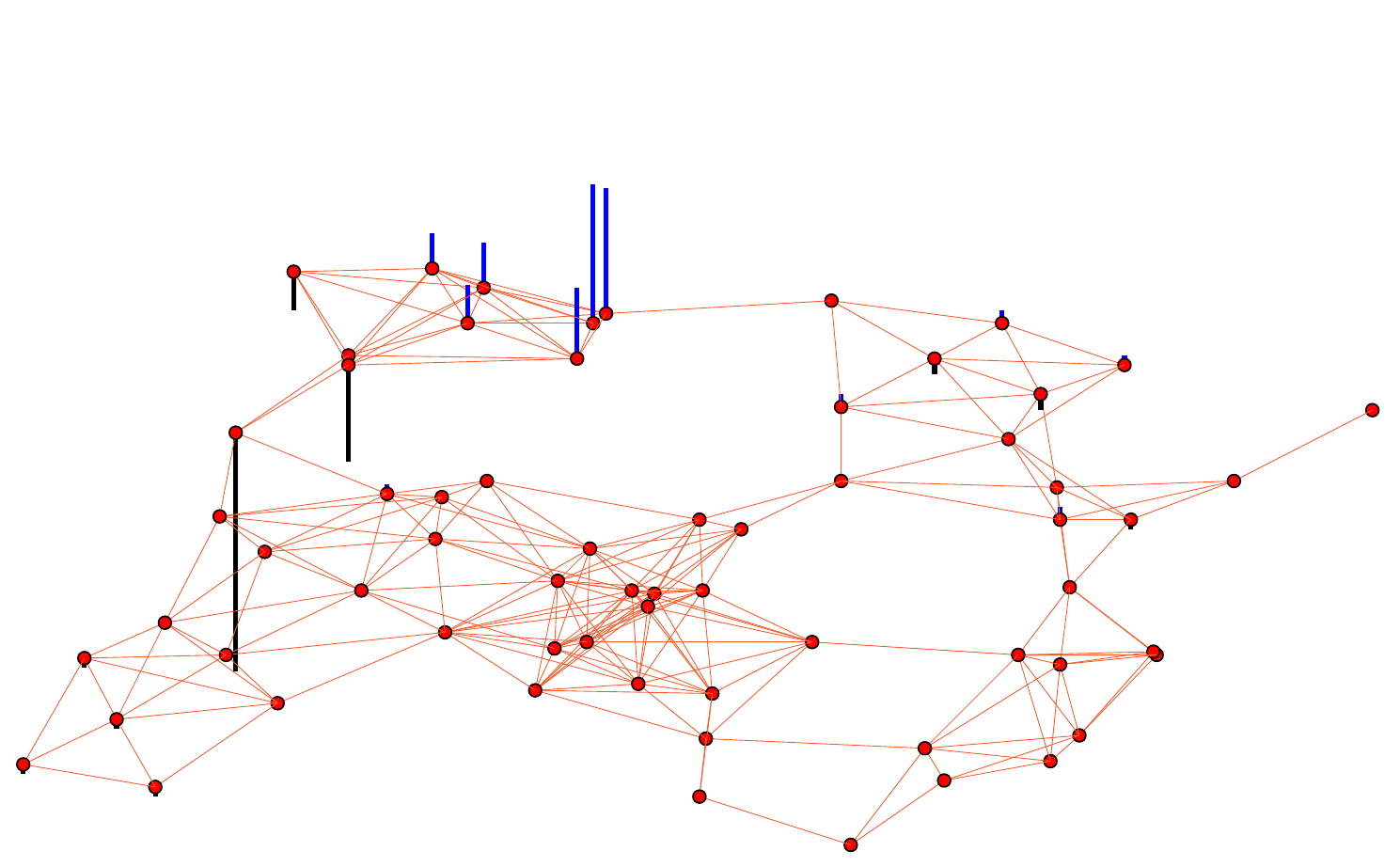}}   
\vspace{.24in}
\centerline{\small{~~~~~~~~~~~~~~(b)}}  
\end{minipage} 
\hfill
\begin{minipage}[b]{.32\linewidth}
\hspace{.3in} \centerline{{$\sqrt{N}\widehat{\left[(T_{27} g).*f\right]}$}} \vspace{.0005cm}

\centerline{\hspace{.5in}\includegraphics[width=.85\linewidth]{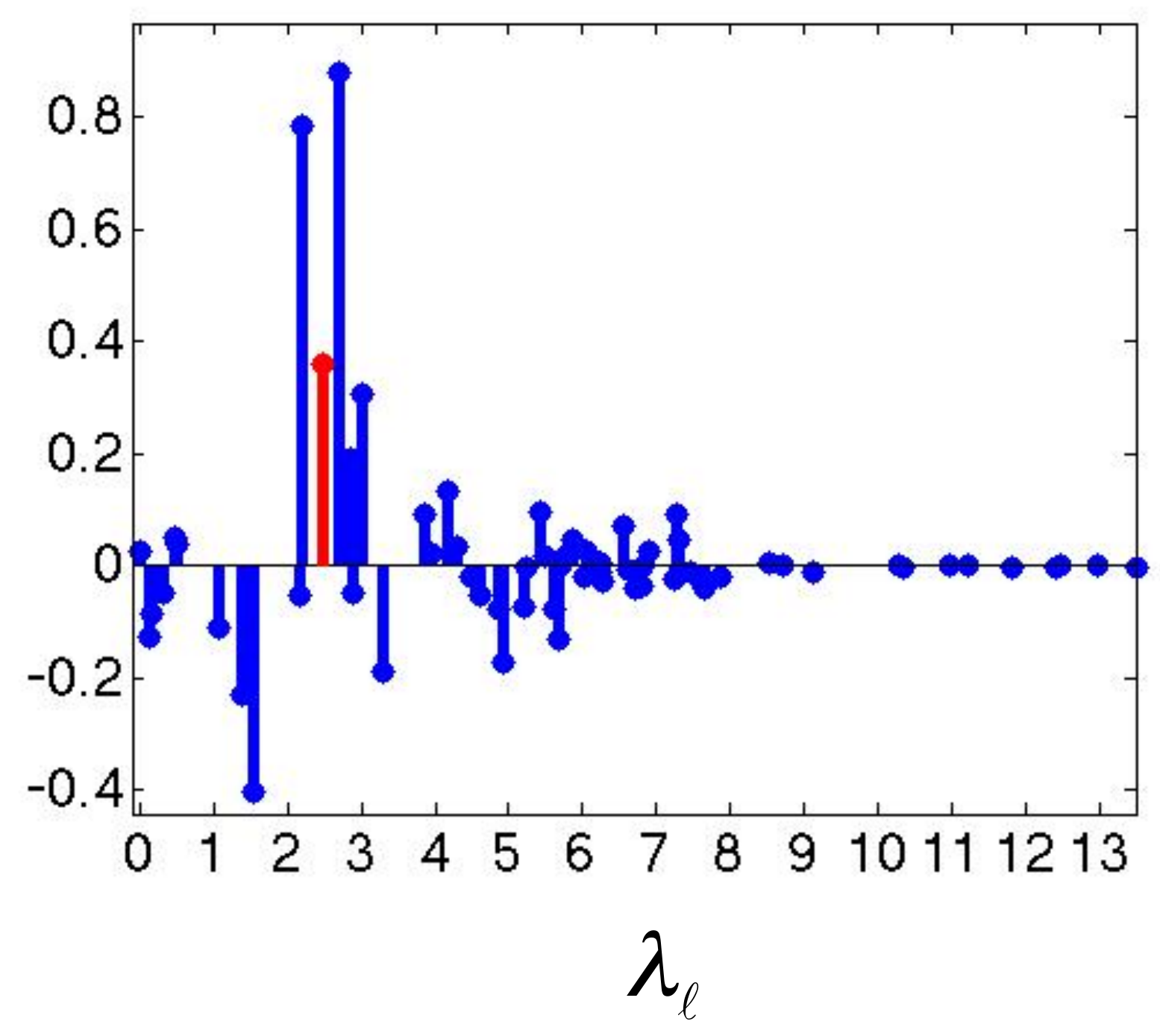}} 
\centerline{~~~~~~~~~~~~~~~~\small{(c)}}
\end{minipage} 
\caption {Alternative interpretation of the windowed graph Fourier transform. (a) A translated (real) window centered at vertex 27. (b) We multiply the window by the signal (pointwise) to compute the windowed signal in the vertex domain. (c) The windowed signal in the graph spectral domain (multiplied by a normalizing constant). The corresponding windowed graph Fourier transform coefficient is $\sqrt{N}\widehat{(T_{27} g).*f}(\lambda_{11})=0.358$, the same as $\ip{f}{g_{27,11}}$, which was computed in Figure \ref{Fig:wgft1}.} 
 \label{Fig:wgft2}
\end{figure}

\end{example}

\subsection{Frame Bounds} \label{Se:frame_bounds}
We now provide a simple sufficient condition for the collection of windowed graph Fourier atoms to form a \emph{frame} (see, e.g., \cite{christensen,kovacevic_frames1,kovacevic_frames2})
\begin{theorem} \label{Th:frame_bounds} 
If $\hat{g}(0) \neq 0$, then
$\left\{g_{i,k}\right\}_{i=1,2,\ldots,N;~k=0,1,\ldots,N-1}$ is a frame; i.e., for all $f\in \Rbb^N$,
\begin{align*}
A \norm{f}_2^2 \leq \sum_{i=1}^N \sum_{k=0}^{N-1} \left|\ip{f}{g_{i,k}}\right|^2 \leq B \norm{f}_2^2,
\end{align*}
where 
\begin{align*}
0 < N |\hat{g}(0)|^2 \leq A :=\min_{n\in \{1,2,\ldots,N\}} \left\{N \norm{T_n g}_2^2\right\} \leq B :=\max_{n\in \{1,2,\ldots,N\}} \left\{N \norm{T_n g}_2^2\right\} \leq N^2 \mu^2 \norm{g}_2^2.
\end{align*}
\end{theorem}

\begin{proof}
\begin{align}\label{Eq:frame_bound_4}
\sum_{i=1}^N \sum_{k=0}^{N-1} \left|\ip{f}{g_{i,k}}\right|^2 
&=\sum_{i=1}^N \sum_{k=0}^{N-1} \left|\ip{f}{M_k T_i g}\right|^2 \nonumber \\
&= N \sum_{i=1}^N \sum_{k=0}^{N-1} \left|\ip{f (T_i g)^*}{\chi_k}\right|^2 \nonumber \\
&=N \sum_{i=1}^N \ip{f (T_i g)^*}{f (T_i g)^*} \\
&=N \sum_{i=1}^N \sum_{n=1}^N \left|f(n)\right|^2 \left|(T_i g)(n)\right|^2 \nonumber \\ \label{Eq:frame_bound_5}
&=N \sum_{i=1}^N \sum_{n=1}^N \left|f(n)\right|^2 \left|(T_n g)(i)\right|^2 \\ \label{Eq:frame_bound_6}
&=N\sum_{n=1}^N \left|f(n)\right|^2 \norm{T_n g}_2^2, 
\end{align}
where \eqref{Eq:frame_bound_4} is due to Parseval's identity, and \eqref{Eq:frame_bound_5} follows from the symmetry of $\L$ and the definition \eqref{Eq:new_translation} of $T_i$. Moreover, under the hypothesis that 
$\hat{g}(0)\neq 0$,
we have
\begin{align} \label{Eq:frame_bound_7}
\norm{T_n g}_2^2 = N \sum_{\l=0}^{N-1} \left|\hat{g}(\lambda_{\l})\right|^2 \left|\chi_{l}(n)\right|^2\geq {|\hat{g}(0)|^2}>0.
\end{align}
Combining \eqref{Eq:frame_bound_6} and \eqref{Eq:frame_bound_7}, for $f \neq 0$,
\begin{align*}
0 < A \norm{f}_2^2 \leq \sum_{i=1}^N \sum_{k=0}^{N-1} \left|\ip{f}{g_{i,k}}\right|^2  \leq B \norm{f}_2^2 < \infty. 
\end{align*}
The upper bound on $B$ follows directly from \eqref{Eq:trans_operator_2}.
\end{proof}
\begin{remark}
If $\mu=\frac{1}{\sqrt{N}}$, then $\left\{g_{i,k}\right\}_{i=1,2,\ldots,N;~k=0,1,\ldots,N-1}$ is a tight frame with $A=B=N\norm{g}_2^2$.
\end{remark}

In Table \ref{Ta:empirical_frame_bounds}, we compare the lower and upper frame bounds derived in Theorem \ref{Th:frame_bounds} to the empirical optimal frame bounds, $A$ and $B$, for different graphs.

\begin{table}[htb] 
\centering
\begin{tabular}{|l*{6}{>{\centering\arraybackslash}m{1.4cm}}|}
\hline
 Graph               & $\mu$ & $\tau$ & $N|\hat{g}(0)|^2$ & A  & B & $N^2 \mu^2 \norm{g}_2^2$ \\
\hline
 &  & 0.5 & 3.2 &  498.4 & 846.2 &  1000.0  \\
Path graph  & 0.063 & 5.0 & 11.0 &  494.5 & 976.5 &  1000.0  \\
  &  & 50.0 & 34.2 &  482.9 & 964.6 &  1000.0  \\
  \hline
 &  & 0.5 & 150.8 &  465.0 & 591.8 &  12676.9  \\
Random regular graph  & 0.225 & 5.0 & 500.0 &  500.0 & 500.0 &  12676.9  \\
  &  & 50.0 & 500.0 &  500.0 & 500.0 &  12676.9  \\
\hline
  &  & 0.5 & 13.5 & 142.7 & 3702.2 & 228607.0  \\
Random sensor network  & 0.956 & 5.0 & 71.7 &  158.4 & 2530.7 &  228607.0  \\
  &  & 50.0 & 387.1 & 389.3 & 1185.6 &  228607.0  \\
\hline
  & & 0.5 & 3.0 &  7.4 & 786.3 &  248756.2  \\
Comet graph  & 0.998 & 5.0 &   17.9 & 43.9 & 1584.1 & 248756.2  \\
  &  & 50.0 & 52.9 &  119.1 & 1490.8 &  248756.2  \\
\hline
\end{tabular}
\caption {Comparison of the empirical optimal frame bounds, $A$ and $B$, to the theoretical lower and upper bounds derived in Theorem \ref{Th:frame_bounds}. The number of vertices for all graphs is $N=500$, and in all cases, the window is $\hat{g}(\lambda_{\l})=Ce^{-\tau \lambda_{\l}}$, where the constant $C$ is chosen such that $\norm{g}_2=1$. In the random regular graph, every vertex has degree 8. The random sensor network is the one shown in Figure \ref{Fig:trans_norms_low}(b). In the comet graph, which is structured like the one shown in Figure \ref{Fig:trans_norms_low}(e), the degree of the center vertex is 200.} 
 \label{Ta:empirical_frame_bounds}
\end{table}

\subsection{Reconstruction Formula}
Provided the window $g$ has a non-zero mean, a signal $f\in \Rbb^N$ can be recovered from its windowed graph Fourier transform coefficients.
\begin{theorem}\label{Th:reconstruction_formula}
If $\hat{g}(0)\neq 0$, then for any $f \in \Rbb^N$,
\begin{align*}
f(n)=\frac{1}{N\norm{T_n g}_2^2} \sum_{i=1}^N \sum_{k=0}^{N-1} Sf(i,k) g_{i,k}(n).
\end{align*}
\end{theorem}

\begin{proof} 
\begin{align*}
&\sum_{i=1}^N \sum_{k=0}^{N-1} Sf(i,k) g_{i,k}(n) \nonumber \\
&=\sum_{i=1}^N \sum_{k=0}^{N-1} 
\left( N \sum_{m=1}^N f(m) \chi_k^*(m) \sum_{\l=0}^{N-1} \hat{g}(\lambda_{\l}) \chi_{\l}(i)\chi_{\l}^*(m)\right)
\left({N} \chi_{k}(n) \sum_{\l^{\prime}=0}^{N-1}\hat{g}(\lambda_{\l^{\prime}})\chi_{\l^{\prime}}^*(i)\chi_{\l^{\prime}}(n) \right) \nonumber \\
&= N^2 \sum_{m=1}^N f(m)  \sum_{\l=0}^{N-1} \sum_{\l^{\prime}=0}^{N-1} \hat{g}(\lambda_{\l})\hat{g}(\lambda_{\l^{\prime}}) \chi_{\l}^*(m) \chi_{\l^{\prime}}(n) \sum_{i=1}^N \chi_{\l}(i) \chi_{\l^{\prime}}^*(i) \sum_{k=0}^{N-1} \chi_k^*(m) \chi_{k}(n)\nonumber \\
&=N^2 \sum_{m=1}^N f(m) \sum_{\l=0}^{N-1} \sum_{\l^{\prime}=0}^{N-1} \hat{g}(\lambda_{\l})\hat{g}(\lambda_{\l^{\prime}}) \chi_{\l}^*(m) \chi_{\l^{\prime}}(n)~\delta_{\l \l^{\prime}}~\delta_{mn} 
=N^2 f(n) \sum_{\l=0}^{N-1} |\hat{g}(\lambda_{\l})|^2 |\chi_{\l}^*(n)|^2 
=N \norm{T_n g}_2^2 ~f(n),
\end{align*}
where the last equality follows from \eqref{Eq:trans_operator_1}.
\end{proof}

\begin{remark}
In the classical case, $\norm{T_n g}_2^2=\norm{g}_2^2=1$, so this term does not appear in the reconstruction formula (see, e.g., \cite[Theorem 4.3]{mallat}).
\end{remark}

\subsection{Spectrogram Examples} \label{Se:spec_examples}
\label{Se:examples}

As in classical time-frequency analysis, we can now examine $|Sf(i,k)|^2$, the squared magnitudes of the windowed graph Fourier transform coefficients of a given signal $f$. In the classical case (see, e.g., \cite[Theorems 4.1 and 4.3]{mallat}), the windowed Fourier atoms form a tight frame, and therefore this \emph{spectrogram} of squared magnitudes can be viewed as an energy density function of the signal across the time-frequency plane. In the graph setting, the windowed graph Fourier atoms do not always form a tight frame, and we cannot therefore, in general, interpret the graph spectrogram as an energy density function. Nonetheless, it can still be a useful tool to elucidate underlying structure in graph signals. In this section,    
we present 
some examples to illustrate this concept and provide further intuition behind the proposed windowed graph Fourier transform. 
\begin{example} \label{Example:path}
We 
consider a path graph of 180 vertices, with all the weights equal to one. 
The graph Laplacian eigenvectors, given in \eqref{Eq:DCT_eigs}, are
the basis vectors 
in the DCT-II transform.
We compose the signal shown in Figure \ref{Fig:path}(a) on the path graph by summing three signals: $\chi_{10}$ restricted to the first 60 vertices, 
$\chi_{60}$ restricted to the next 60 vertices, and $\chi_{30}$ restricted to the final 60 vertices. We design a window $g$ by setting $\hat{g}(\lambda_{\l})=Ce^{-\tau\lambda_{\l}}$ with $\tau=300$ and $C$ chosen such that 
$\norm{g}_2=1$. The ``spectrogram'' in Figure \ref{Fig:path}(b) shows 
$\left|Sf(i,k)\right|^2$ for all $i \in \{1,2,\ldots,180\}$ and $k \in \{0,1,\ldots,179\}$. 
Consistent with 
intuition from discrete-time signal processing (see, e.g., \cite[Chapter 2.1]{groechenig}), the spectrogram shows the discrete cosines at different frequencies with the appropriate spatial localization.    


\begin{figure}[htb]
\centering
\hfill
\begin{minipage}[b]{.5\linewidth}
\centerline{\includegraphics[width=\linewidth]{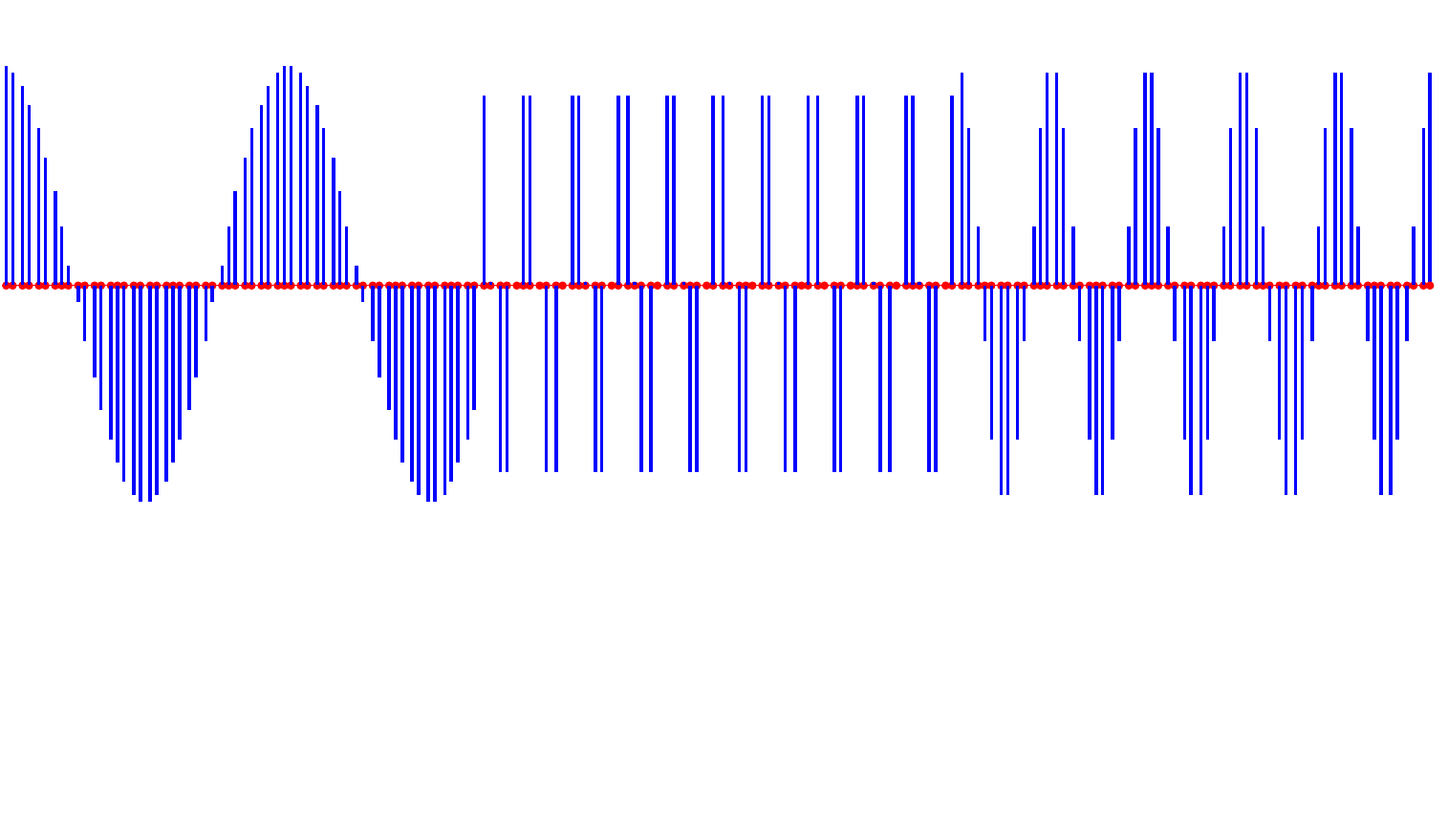}}
\centerline{\small{(a)}}
\end{minipage} 
\hfill
\begin{minipage}[b]{.42\linewidth}
   \centering
   \centerline{\includegraphics[width=\linewidth]{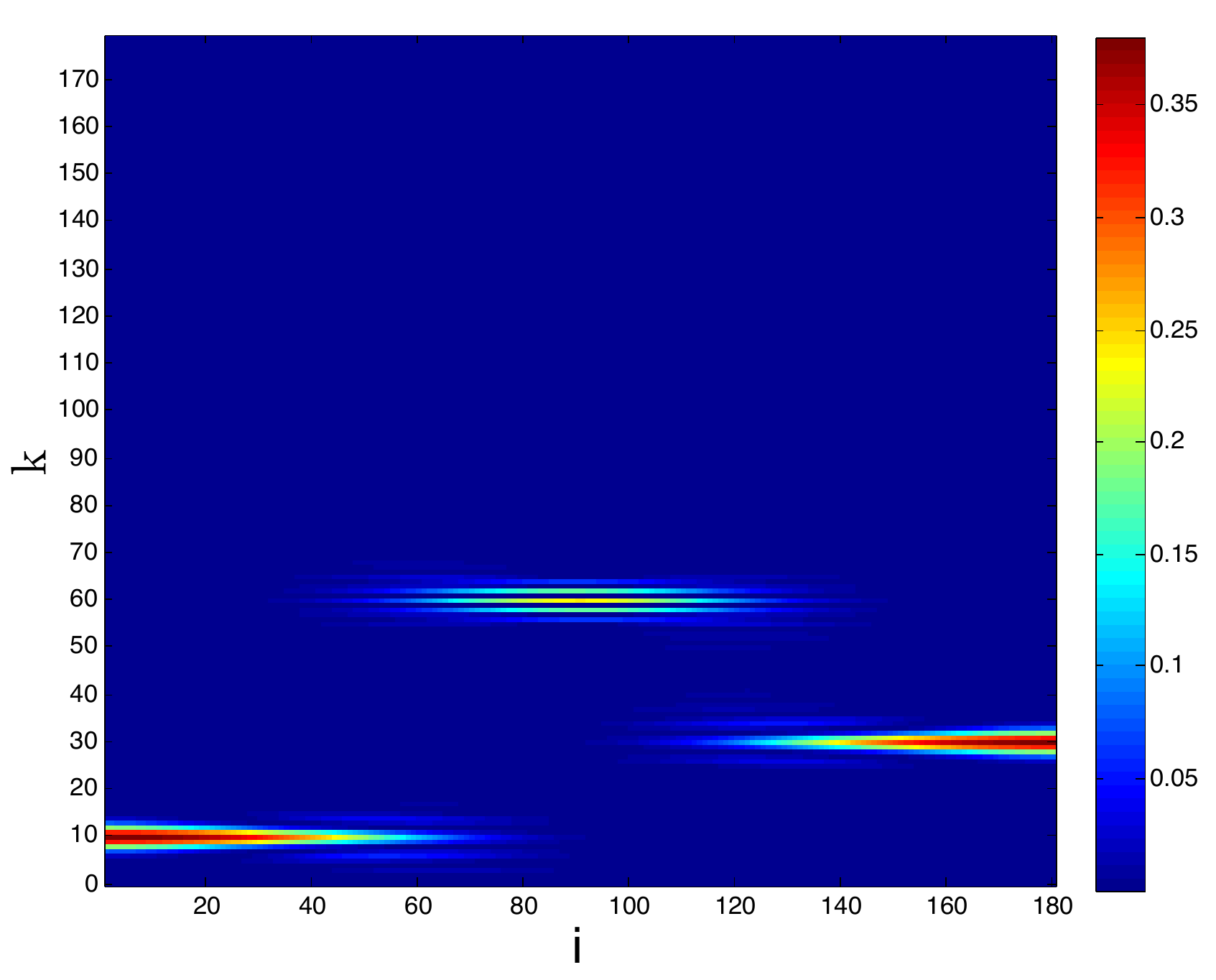}}
\centerline{\small{(b)}~~~~}
\end{minipage} \hfill \hfill
\caption {Spectrogram example on a path graph. (a) A signal $f$ on the path graph that is comprised of three different graph Laplacian eigenvectors restricted to three different segments of the graph. (b) A spectrogram of $f$. 
The vertex indices are on the horizontal axis, and the frequency indices 
are on the vertical axis.}  \label{Fig:path}
\end{figure}
\end{example}

\begin{example} \label{Ex:sensor}
We now compute the spectrogram of the signal $f$ on the random sensor network of Example \ref{Ex:wgft1}, using the same window $\hat{g}(\lambda_{\l})=Ce^{-\tau\lambda_{\l}}$ from Example \ref{Ex:wgft1}, with $\tau=3$. It is not immediately obvious upon inspection that the signal $f$ shown in both Figure \ref{Fig:wgft1}(a) and Figure \ref{Fig:sensor}(a) is a highly structured signal; however,  
the spectrogram, shown in Figure \ref{Fig:sensor}(b), elucidates the structure of $f$, as we can clearly see  
three different frequency components present in three different areas of the graph. In order to mimic Example \ref{Example:path} on a more general graph, we constructed the signal $f$ by first using spectral clustering (see, e.g., \cite{spectral_clustering}) to partition the network into three sets of vertices, which are shown in red, blue, and green in Figure \ref{Fig:sensor}(a). Like Example \ref{Example:path}, we then took the signal $f$ to be the sum of three signals: $\chi_{10}$ restricted to the red set of vertices, $\chi_{27}$ restricted to the blue set of vertices, and $\chi_{5}$ restricted to the green set of vertices. 
%
%

\begin{figure}[htb]
\centering
\hfill
\begin{minipage}[b]{.5\linewidth}
   \centering
   \centerline{\includegraphics[width=\linewidth]{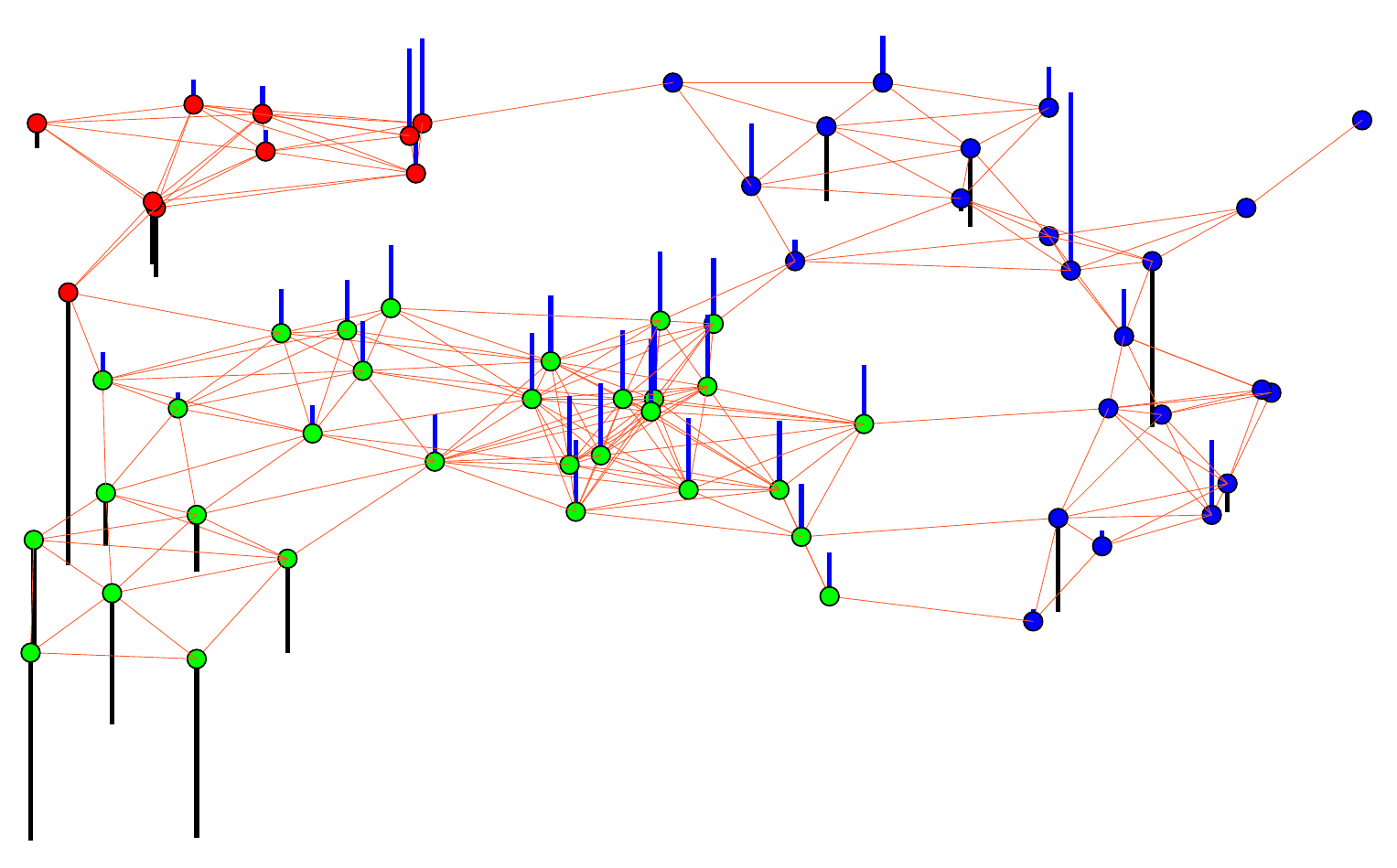}}
\centerline{\small{(a)}~~}
\end{minipage} 
\hfill
\begin{minipage}[b]{.42\linewidth}
   \centering
   \centerline{\includegraphics[width=\linewidth]{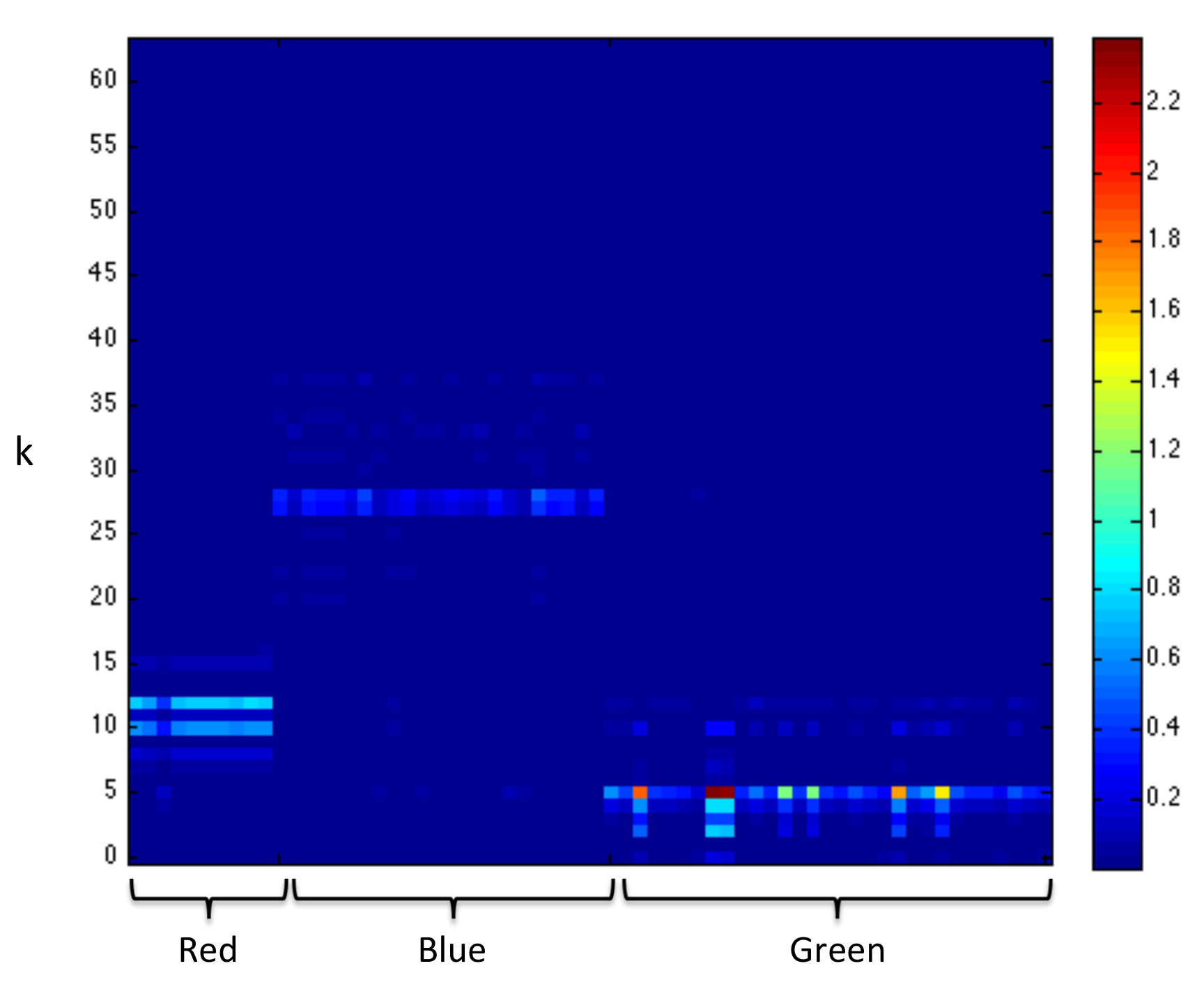}}
\centerline{\small{(b)}}
\end{minipage} \hfill \hfill
\caption {Spectrogram example on a random sensor network. (a) A signal comprised of three different graph Laplacian eigenvectors restricted to three different clusters of a random sensor network. 
(b) The spectrogram shows the different frequency components in the red, blue, and green clusters.} 
 \label{Fig:sensor}
\end{figure}

\end{example}


\begin{remark}
In Figure \ref{Fig:sensor}(b), in order to make the structure of the signal more evident, we arranged the vertices of the sensor graph on the horizontal axis according to clusters, so that vertices in the same cluster are close to each other. Of course, this would not be possible if we did not have an idea of the structure \emph{a priori}. A more general way to view the spectrogram of a graph signal without such \emph{a priori} knowledge is as a sequence of images, with one image per graph Laplacian eigenvalue and the sequence arranged monotonically according to the corresponding eigenvalue. Then, as we scroll through the sequence (i.e., play a frequency-lapse video), the areas of the graph where the signal contains low frequency components ``light up'' first, then the areas where the signal contains middle frequency components, and so forth. In Figure \ref{Fig:spectro_sequence}, we show a subset of the images that would comprise such a sequence for the signal from Example \ref{Ex:sensor}.  
\end{remark}

\begin{figure}[htb]
\centering
\centerline{\includegraphics[width=\linewidth]{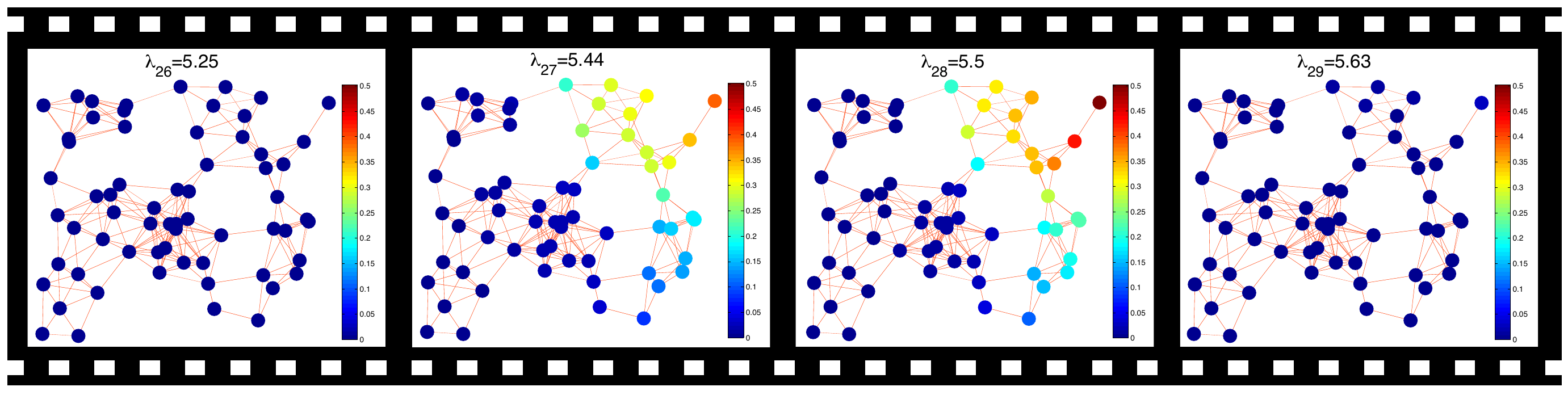}}
\caption {Frequency-lapse video representation of the spectrogram of a graph signal. The graph spectrogram can be viewed as a sequence of images, with each frame corresponding to the spectrogram coefficients at all vertices and a single frequency, $|Sf(:,k)|^2$. This particular subsequence of frames shows the spectrogram coefficients from eigenvalues $\lambda_{26}$ through $\lambda_{29}$. At frequencies close to $\lambda_{27}=5.44$, we can see the coefficients ``light up'' in the blue cluster of vertices from Figure \ref{Fig:sensor}(a), corresponding to the spectrogram coefficients in the middle of Figure \ref{Fig:sensor}(b).} 
 \label{Fig:spectro_sequence}
\end{figure}


\subsection{Application Example: Signal-Adapted Graph Clustering}
Motivated by Examples \ref{Example:path} and \ref{Ex:sensor}, we can also use the windowed graph Fourier transform coefficients as feature vectors to cluster a graph into sets of vertices, taking into account both the graph structure and a given signal $f$. In particular, for $i=1,2,\ldots,N$, we can define $y_i:=Sf(i,:)\in \Rbb^N$, and then use a standard clustering algorithm to cluster the points $\left\{y_i\right\}_{i=1,2,\ldots,N}$.\footnote{Note that if we take the window to be a heat kernel $\hat{g}(\lambda)=Ce^{-\tau \lambda}$, as $\tau \rightarrow 0$, $T_i g \rightarrow \delta_i$ and $Sf(i,k)\rightarrow \ip{\sqrt{N}M_k \delta_i}{f}=\sqrt{N} \chi_k(i) f(i)$. Thus, this method of clustering reduces to spectral clustering \cite{spectral_clustering} when (i) the signal $f$ is constant across all vertices, and (ii) each window is a delta in the vertex domain.}
 
\begin{example}\label{Ex:signal_adapted_clustering}
We generate a signal $f$ on the 500 vertex random sensor network of Example \ref{Ex:high_coherence} as follows. First, we generate four random signals $\left\{f_i\right\}_{i=1,2,3,4}$, with each random component uniformly distributed between 0 and 1. Second, we generate four graph spectral filters $\left\{\widehat{h_i}(\cdot)\right\}_{i=1,2,3,4}$ that cover different bands of the graph Laplacian spectrum, as shown in Figure \ref{Fig:signal_adapted_clustering}(a). Third, we generate four clusters, $\left\{C_i\right\}_{i=1,2,3,4}$, on the graph, taking the first three to be balls of radius 4 around different center vertices, and the fourth to be the remaining vertices. These clusters are shown in Figure \ref{Fig:signal_adapted_clustering}(b). Fourth, we generate a signal on the graph as
\begin{align*}
f=\sum_{i=1}^4 \left(f_i \ast h_i\right)|_{C_i};
\end{align*}
i.e., we filter each random signal by the corresponding filter (to shift its frequency content to a given band), and then restrict that signal to the given cluster by setting the components outside of that cluster to 0. The resulting signal $f$ is shown in Figure \ref{Fig:signal_adapted_clustering}(c). Fifth, we compute the windowed graph Fourier transform coefficients of $f$ using a window $\hat{g}(\lambda)=e^{-0.3\lambda}$. Sixth, we use a classical trick of applying a nonlinear transformation to the coefficients (see, e.g., \cite[Section 3]{jain}), and define the vectors $\left\{y_i\right\}_{i=1,2,\ldots,N}$ as
\begin{align*}
y_i(k):=\tanh\Bigl(\alpha |Sf(i,k)|\Bigr)=\frac{1-e^{-2\alpha |Sf(i,k)|}}{1+e^{-2\alpha |Sf(i,k)|}},
\end{align*}
with $\alpha=0.75$. Finally, we perform $k$-means clustering on the points $\left\{y_i\right\}_{i=1,2,\ldots,N}$, searching for 6 clusters to give the algorithm some extra flexibility. The resulting signal-adapted graph clustering is shown in Figure \ref{Fig:signal_adapted_clustering}(d).

\begin{figure}[h]
\hfill
\begin{minipage}[b]{.22\linewidth}
   \centering
      \centerline{\small{Filters}}
   \centerline{\includegraphics[width=\linewidth]{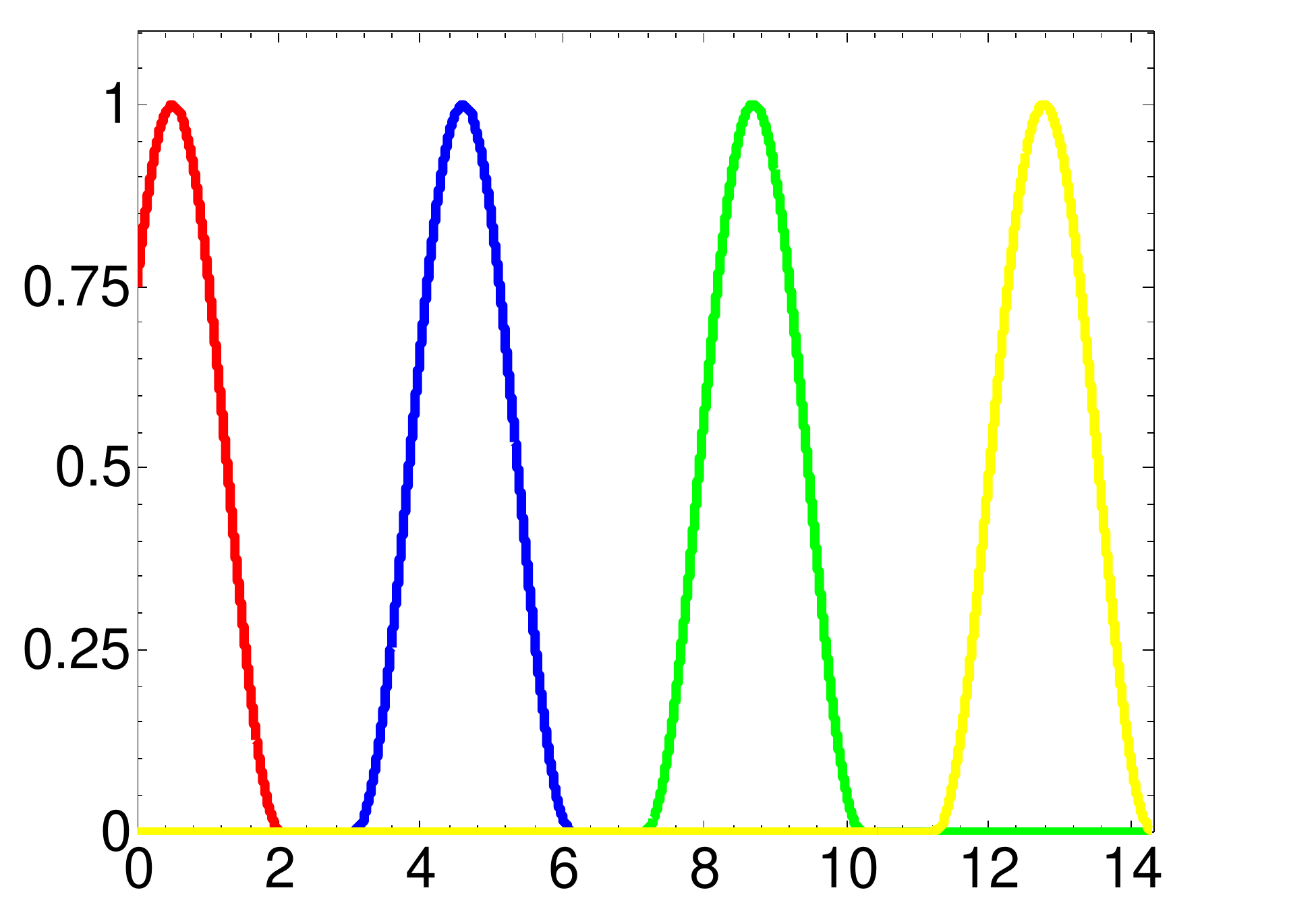}} 
   \centerline{\small{$\lambda$}}
\centerline{\small{(a)}}
\end{minipage}
\hfill
\begin{minipage}[b]{.22\linewidth}
   \centering
   \centerline{\small{Ground Truth Clusters}}
   \centerline{\includegraphics[width=\linewidth]{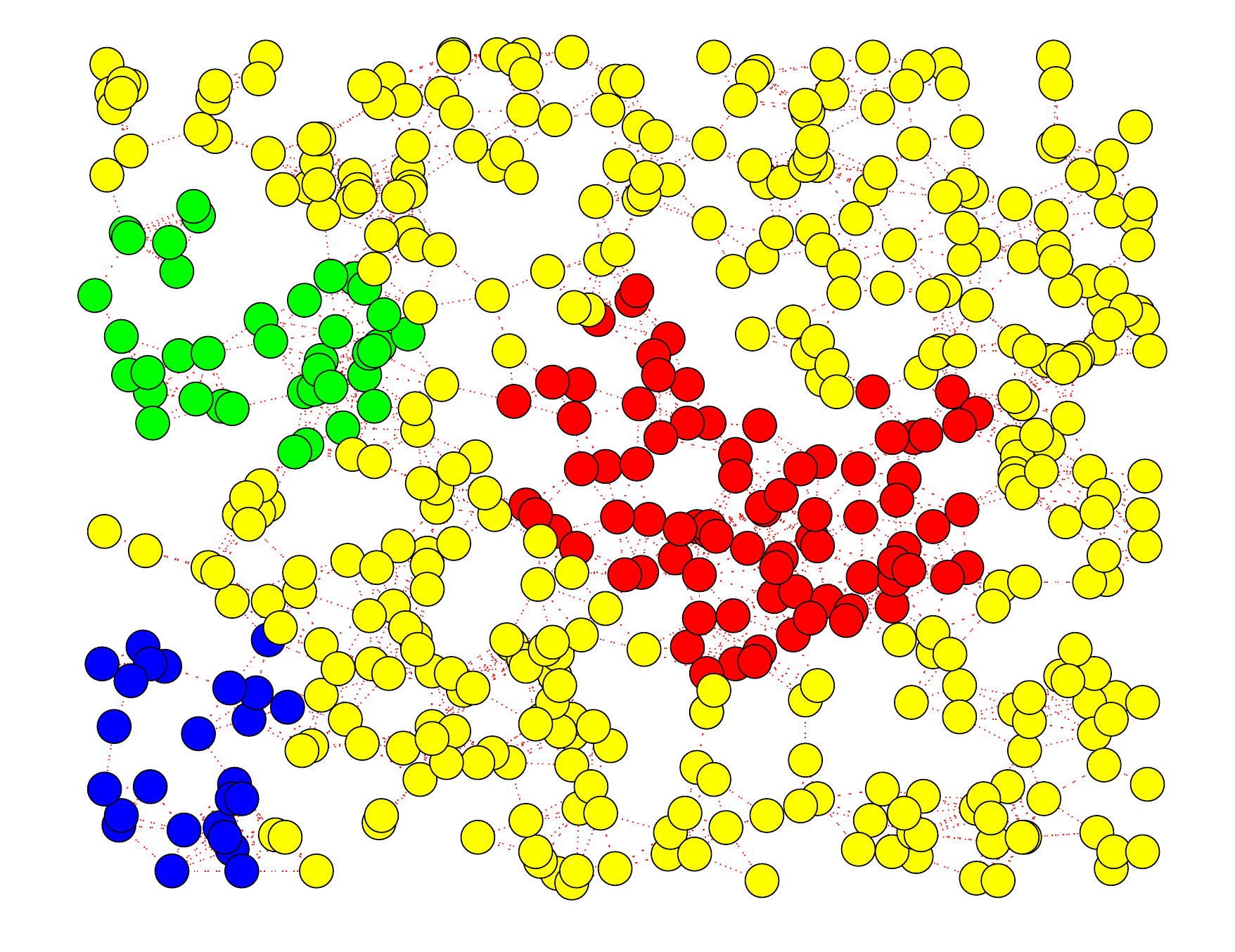}} 
      \vspace{.05in}
\centerline{\small{(b)}}
\end{minipage}
\hfill
\begin{minipage}[b]{.3\linewidth}
   \centering
      \centerline{\small{Signal}}
   \centerline{\includegraphics[width=\linewidth]{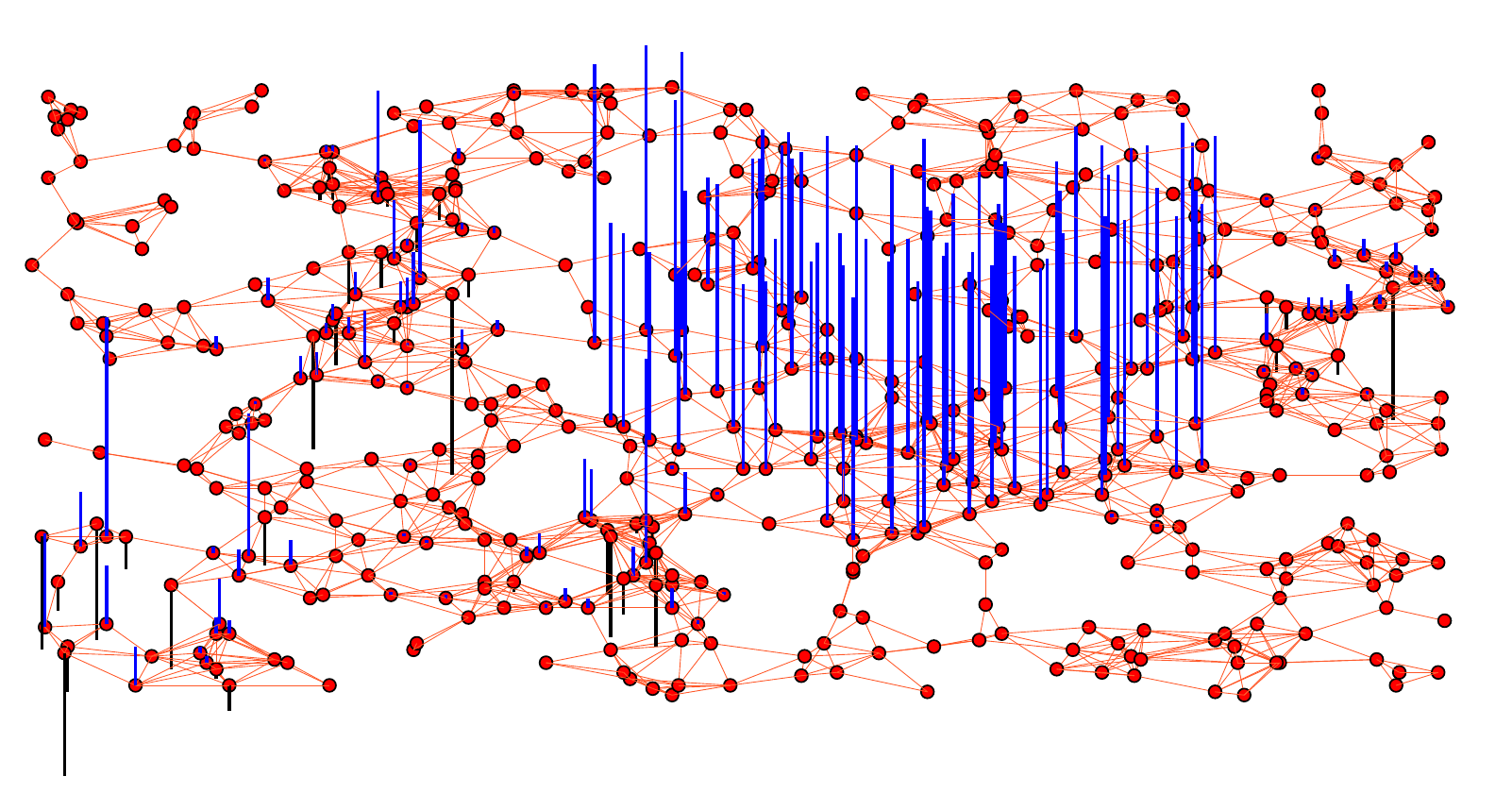}} 
   \vspace{.1in}
\centerline{\small{(c)}}  
\end{minipage}
\hfill
\begin{minipage}[b]{.22\linewidth}
   \centering
         \centerline{\small{Result of}} 
                  \centerline{\small{Signal-Adapted Clustering}}
   \centerline{\includegraphics[width=\linewidth]{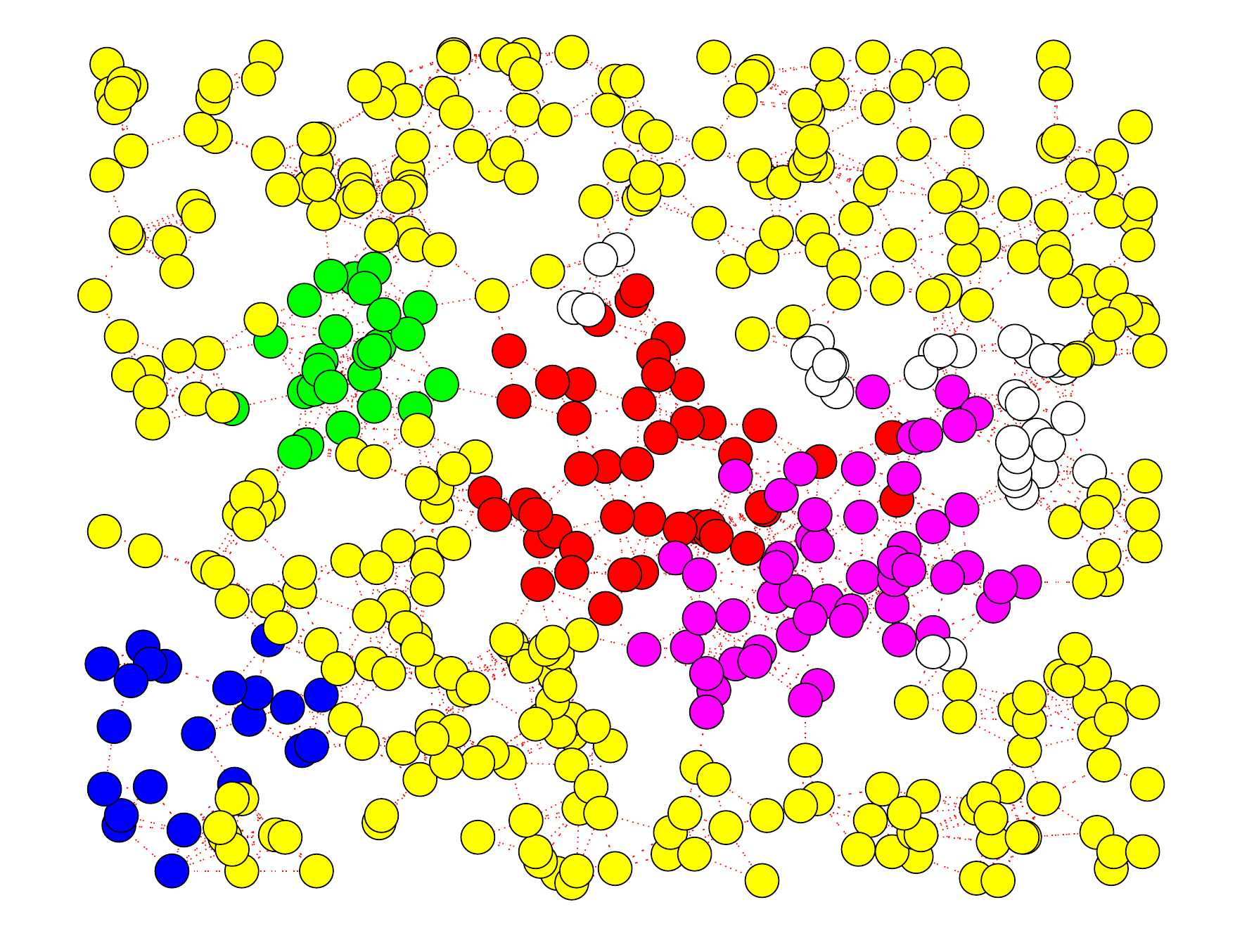}} 
      \vspace{.05in}
\centerline{\small{(d)}}
\end{minipage}
\hfill
\caption {Signal-adapted graph clustering example. (a) The graph spectral filters $\left\{\widehat{h_i}(\cdot)\right\}_{i=1,2,3,4}$ covering different bands of the graph Laplacian spectrum. (b) The four clusters used to generate the signal $f$. (c) The signal $f$ is a combination of signals in different frequency bands $(f_i \ast h_i)$ restricted to different areas of the graph. We can see for example that the red cluster is generated by the lowpass red filter in (a), and therefore the signal $f$ is relatively smooth around the red cluster of vertices. (d) The signal-adapted graph clustering results. We see that the red and magenta clusters roughly correspond to the original red cluster, and the white and yellow clusters roughly correspond to the original yellow cluster. There are some errors around the edges of the original red, blue, and green clusters, particularly in regions of low connectivity.} 
  \label{Fig:signal_adapted_clustering}
\end{figure}
\end{example}

\subsection{Tiling} \label{Se:tiling}
Thus far, we have generalized the classical notions of translation and modulation in order to mimic the construction of the classical windowed Fourier transform. We have seen from the examples in the previous subsections that the spectrogram may be an informative tool for signals on both regular and 
irregular graphs. In this section and the following section, we examine the extent to which our intuitions from classical time-frequency analysis carry over to the graph setting, and where they deviate due to the irregularity of the data domain, and, in turn, the possibility of localized Laplacian eigenfunctions.

First, we compare tilings of the time-frequency plane (or, in the graph setting, the vertex-frequency plane). Recall that Heisenberg boxes represent the time-frequency resolution of a given dictionary atom (including, e.g., windowed Fourier atoms or wavelets) in the time-frequency plane (see, e.g., \cite[Chapter 4]{mallat}). As shown in the tiling diagrams of Figure \ref{Fig:path_comparison}(a) and \ref{Fig:path_comparison}(b), respectively, the Heisenberg boxes of classical windowed Fourier atoms have the same size throughout the time-frequency plane, while the Heisenberg boxes of classical wavelets have different sizes at different wavelet scales. While one can trade-off time and frequency resolutions (e.g., change the length and width of the Heisenberg boxes of classical windowed Fourier atoms by changing the shape of the analysis window), the Heisenberg uncertainty principle places a lower limit on the area of each Heisenberg box.


\begin{figure}[h]
\centering
\hfill
\begin{minipage}[b]{.26\linewidth}
   \centering
   \centerline{\small{~~~~~Classical Windowed}}
      \centerline{\small{~~~~~~Fourier Atoms}}
   \centerline{\includegraphics[width=\linewidth]{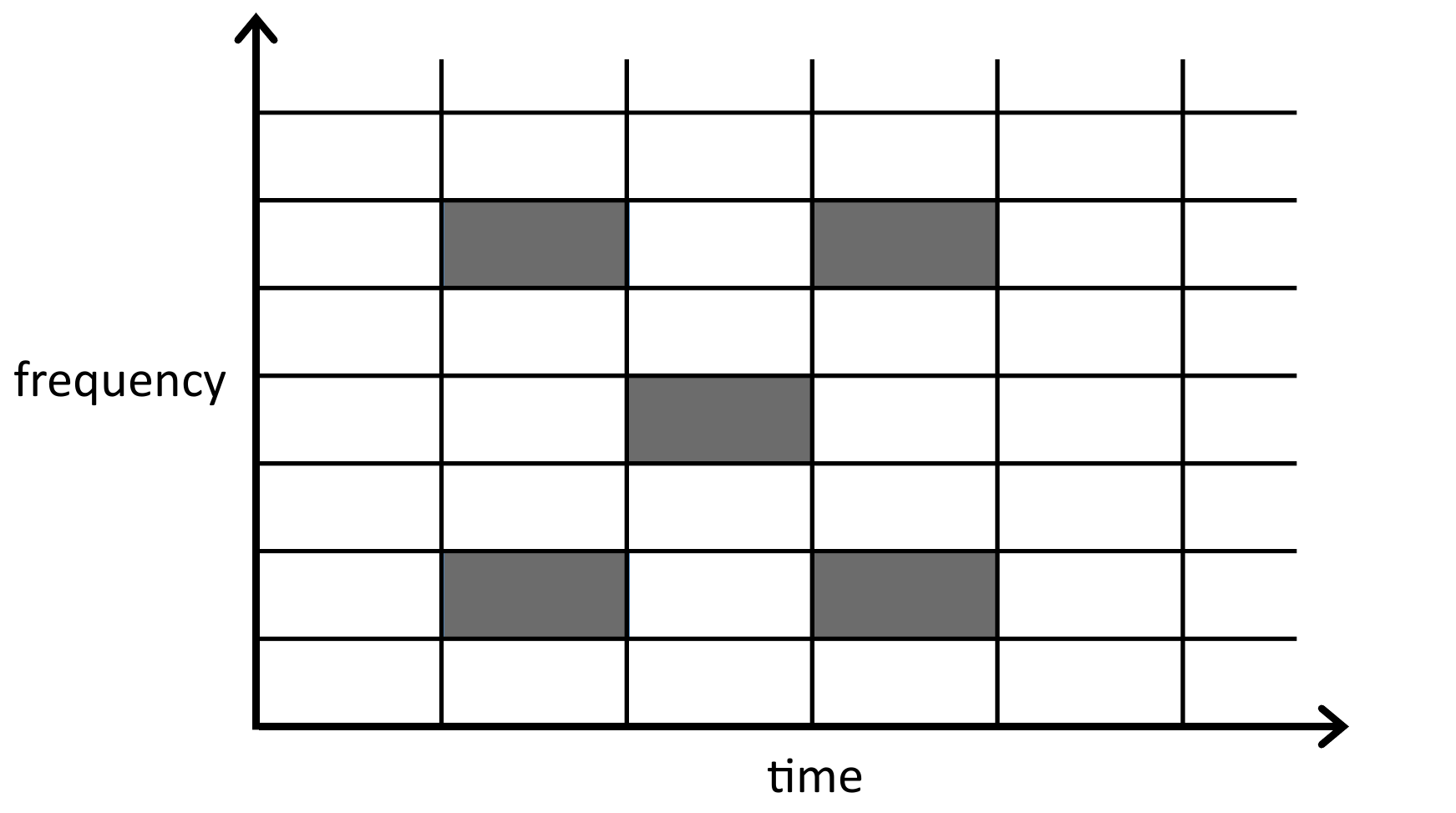}}
\centerline{\small{~~~~~~~~(a)}}
\end{minipage} 
\hfill
\begin{minipage}[b]{.26\linewidth}
   \centering
      \centerline{\small{~~~~Classical Wavelets}}
   \centerline{\includegraphics[width=\linewidth]{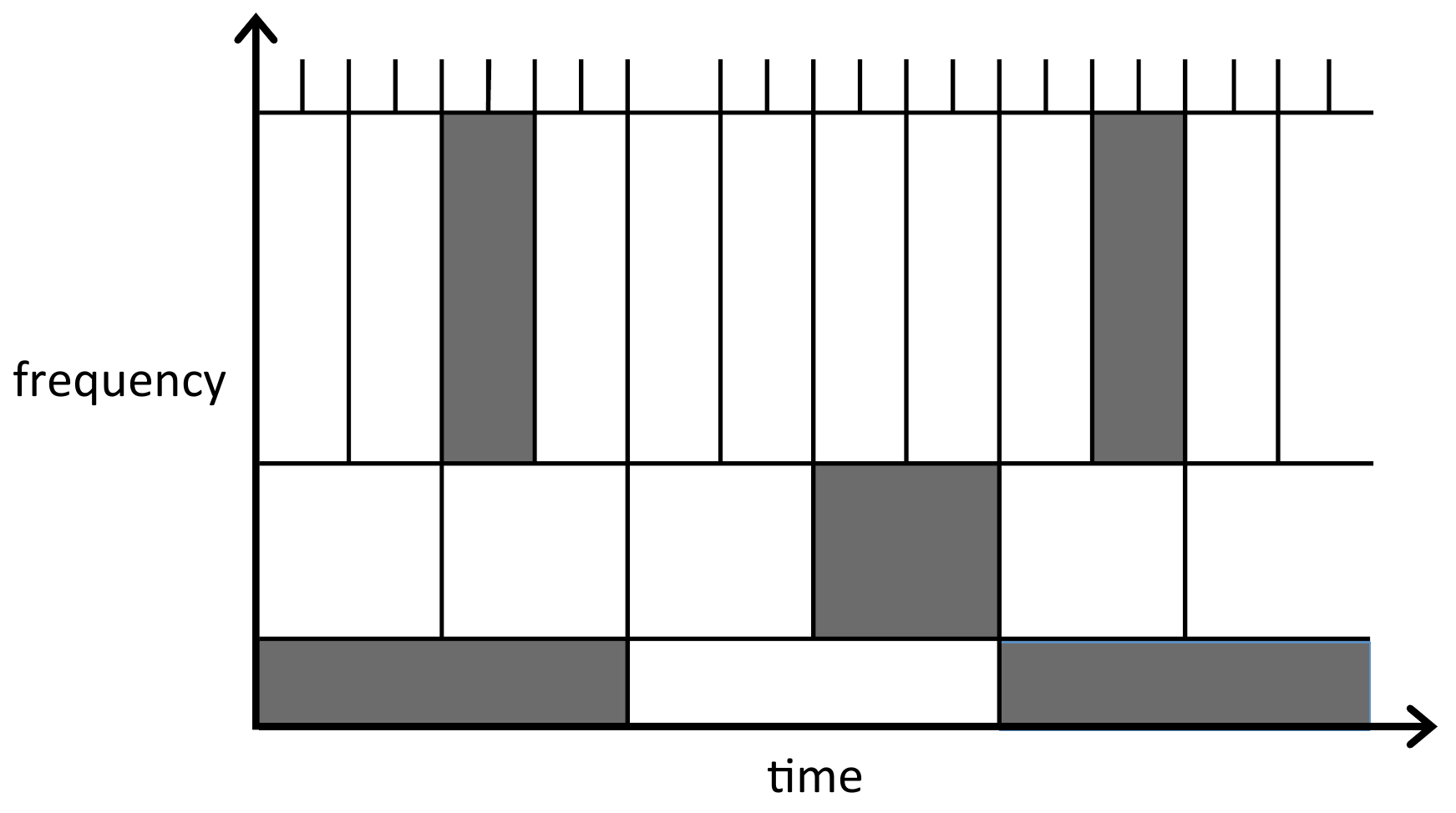}} 
\centerline{\small{~~~~~~~~(b)}}
\end{minipage} 
\hfill
\begin{minipage}[b]{.2\linewidth}
   \centering
         \centerline{\small{Windowed Graph}}
         \centerline{\small{Fourier Atoms}}
   \centerline{\includegraphics[width=\linewidth]{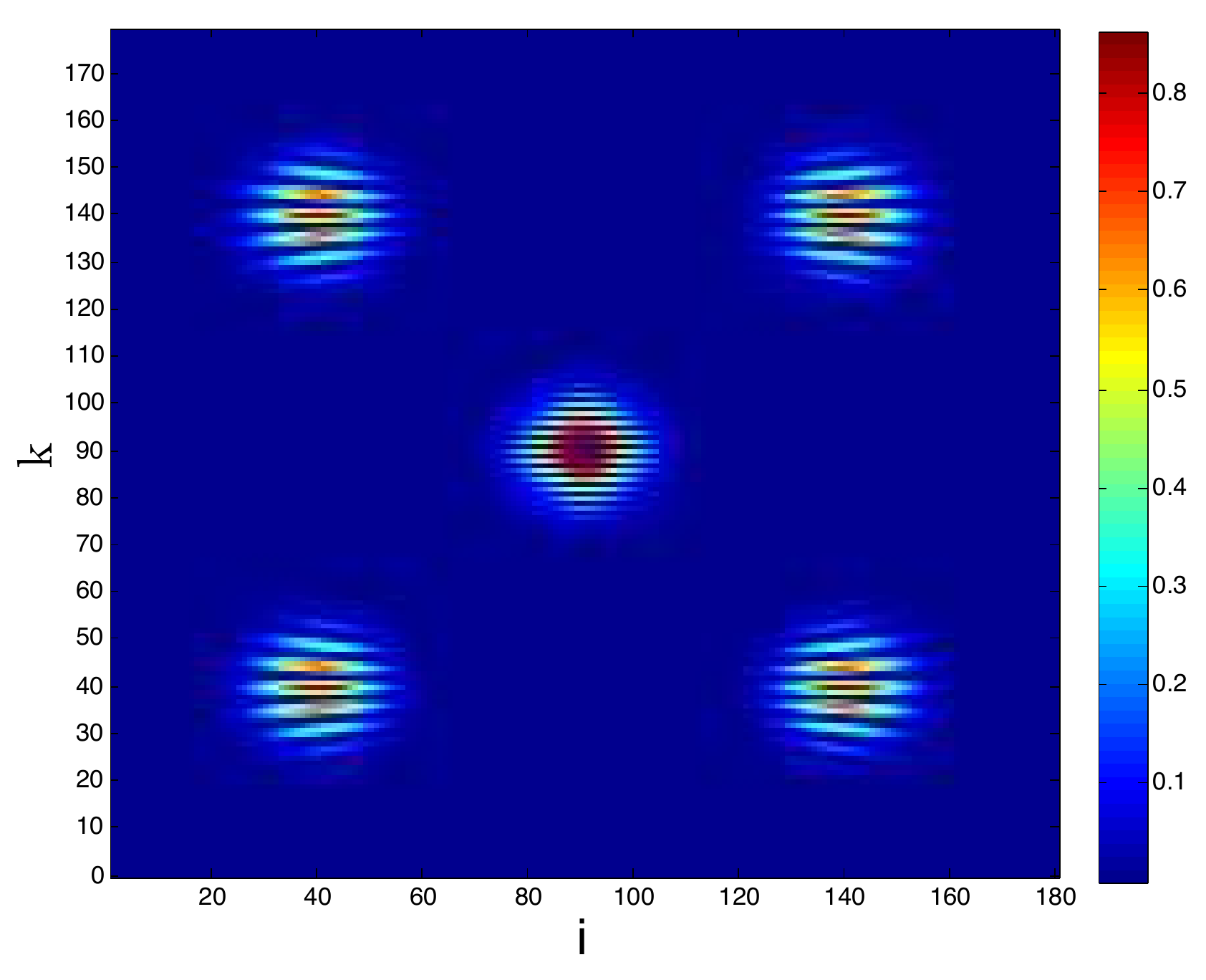}}
\centerline{\small{(c)~}}
\end{minipage} 
\hfill
\begin{minipage}[b]{.2\linewidth}
   \centering
           \centerline{\small{Spectral Graph}}
           \centerline{\small{Wavelets}}
   \centerline{\includegraphics[width=\linewidth]{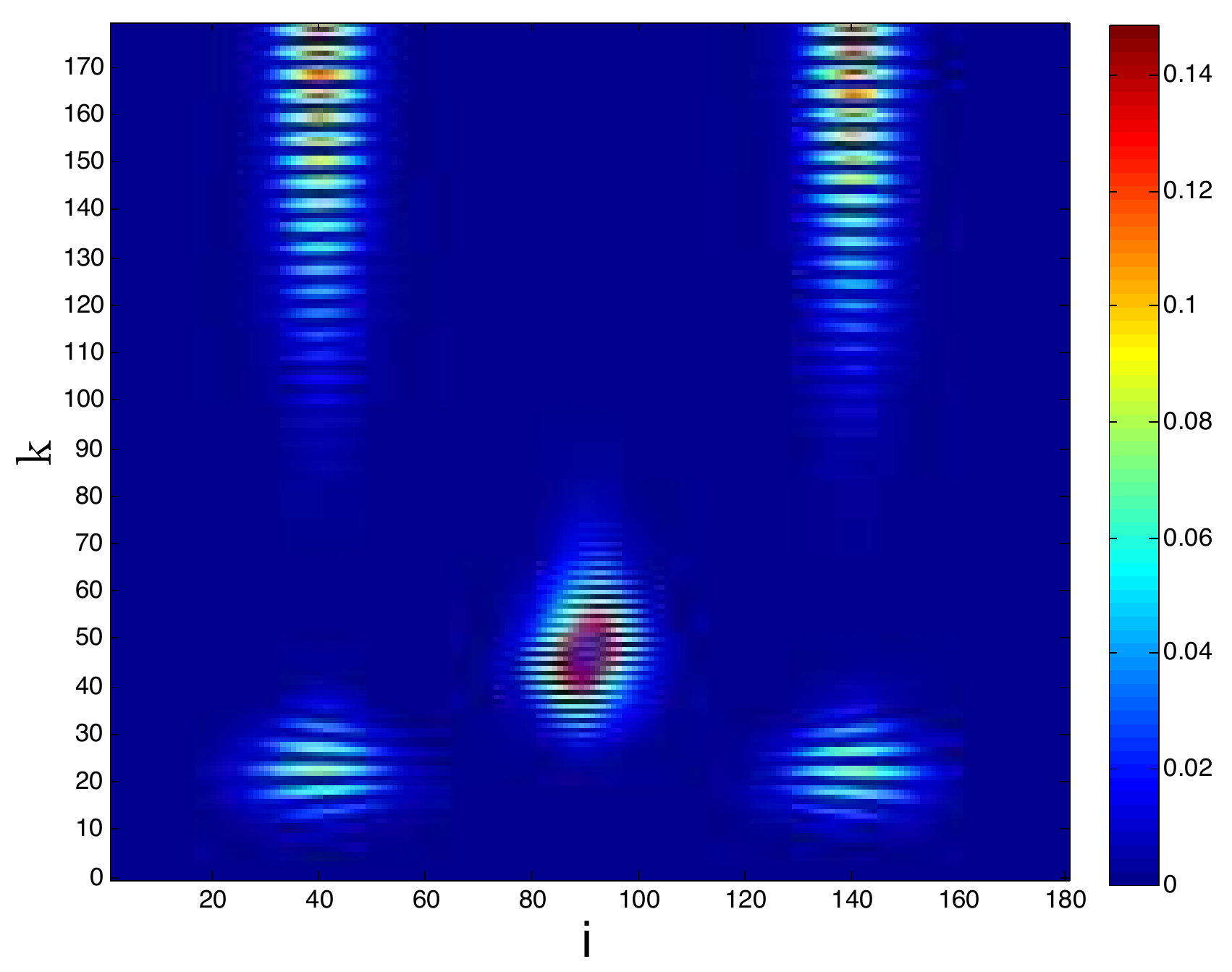}} 
\centerline{\small{(d)~}}
\end{minipage} \hfill 
\caption {(a) Tiling of the time-frequency plane by classical windowed Fourier atoms. (b) Tiling of the time-frequency plane by classical wavelets. (c) Sum of the spectrograms of five windowed graph Fourier atoms on the path graph with 180 vertices. (d) Sum of the spectrograms of five spectral graph wavelets on the path graph with 180 vertices.} 
 \label{Fig:path_comparison}
\end{figure}

In Figure \ref{Fig:path_comparison}(c) and \ref{Fig:path_comparison}(d), we use the windowed graph Fourier transform to show the sums of spectrograms of five windowed graph Fourier atoms on a path graph and five spectral graph wavelets on a path graph, respectively. The plots are not too different from what intuition from classical time-frequency analysis might suggest. Namely, the sizes of the Heisenberg boxes for different windowed graph Fourier atoms are roughly the same, while the sizes of the Heisenberg boxes of spectral graph wavelets are similar at a fixed scale, but vary across scales. 

In Figure \ref{Fig:atoms_jointly}, we plot three different windowed graph Fourier atoms -- all with the same center vertex -- on the Swiss roll graph. Note that all three atoms are jointly localized in the vertex domain around the center vertex 62, and in the graph spectral domain around the frequencies to which they have been respectively modulated. However, unlike the path graph example in Figure \ref{Fig:path_comparison}, the sizes of the Heisenberg boxes of the three atoms are quite different. In particular, the atom $g_{62,983}$ is extremely close to a delta function in both the vertex domain and the graph spectral domain, which of course is not possible in the classical setting due to the Heisenberg uncertainty principle. The reason this happens is that  the coherence of this Swiss roll graph is $\mu=0.94$, and the eigenvector $\chi_{983}$ is highly localized, with a value of -0.94 at vertex 62. The takeaway is that highly localized eigenvectors can limit the extent to which intuition from the classical setting carriers over to the graph setting.

\begin{figure}[h]
\centering
\hfill
\begin{minipage}[b]{.31\linewidth}
\hspace{.3in}\centerline{{$g_{62,100}$}}
\centerline{\hspace{.5in}\includegraphics[width=.85\linewidth]{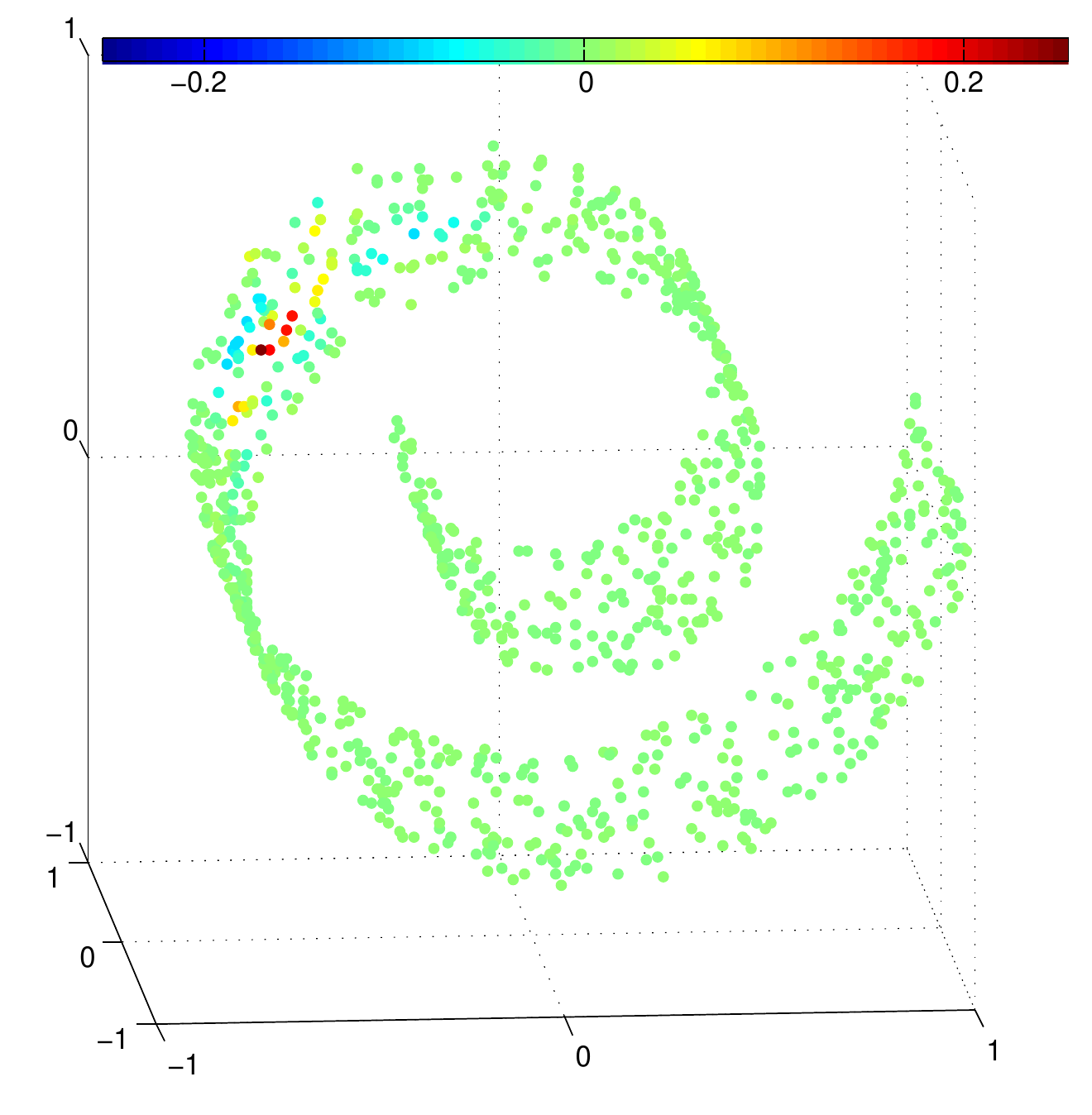}} 
\centerline{\includegraphics[width=\linewidth]{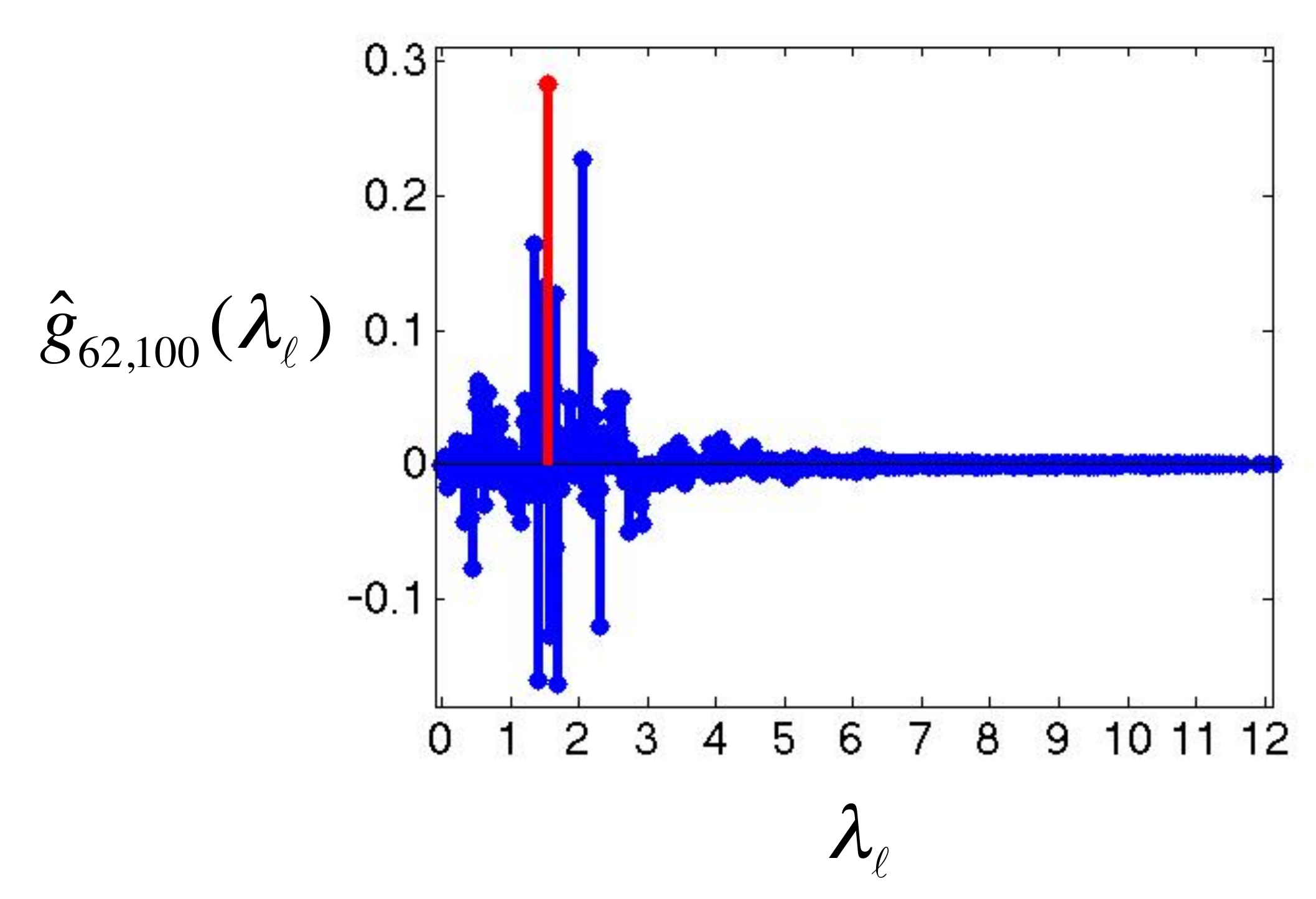}} 
\centerline{\small{~~~~~~~~~~~~~~(a)}}
\end{minipage} 
\hfill
\begin{minipage}[b]{.31\linewidth}
\hspace{.3in}\centerline{{$g_{62,450}$}}
\centerline{\hspace{.5in}\includegraphics[width=.85\linewidth]{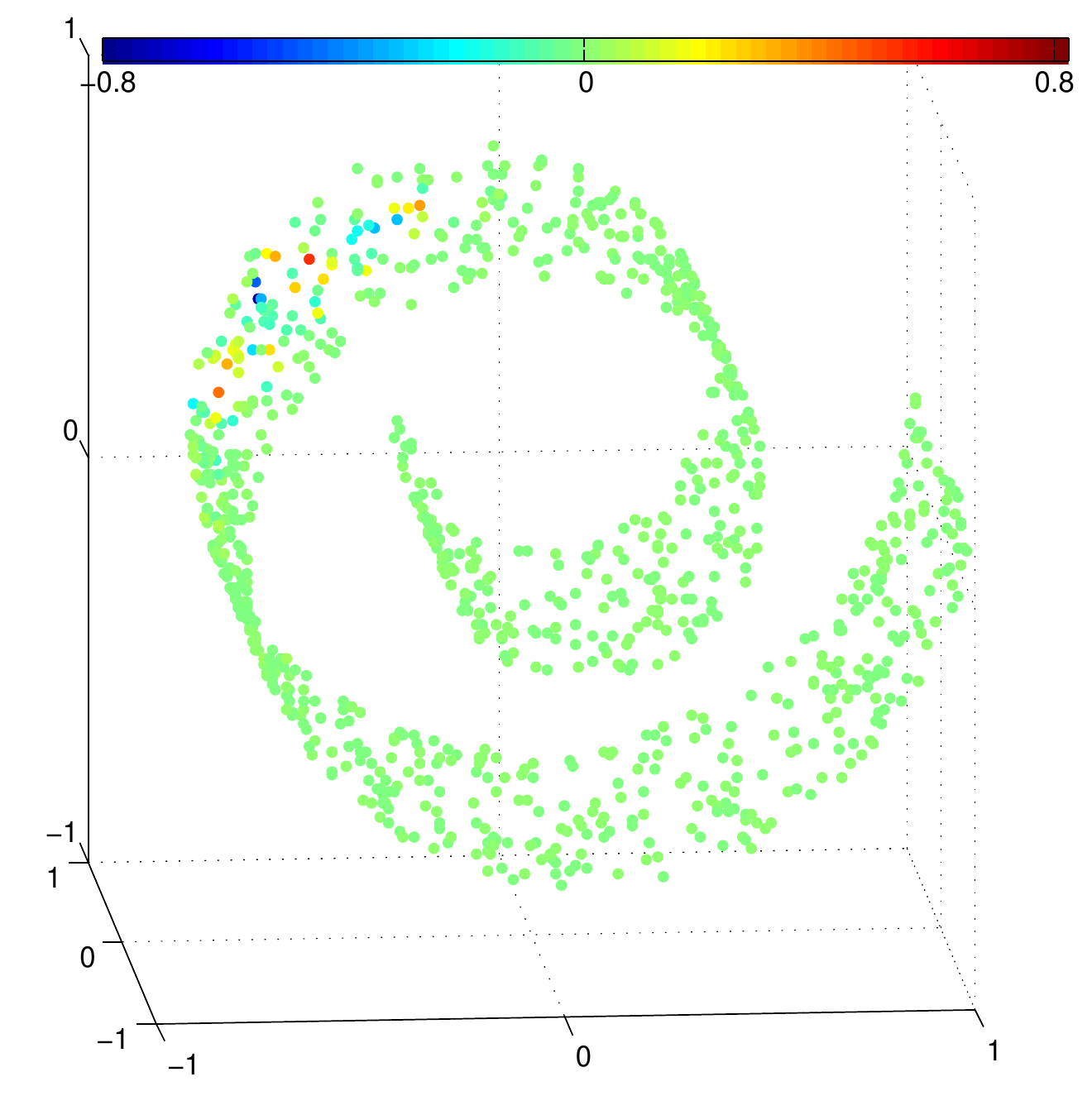}}   
\centerline{\includegraphics[width=\linewidth]{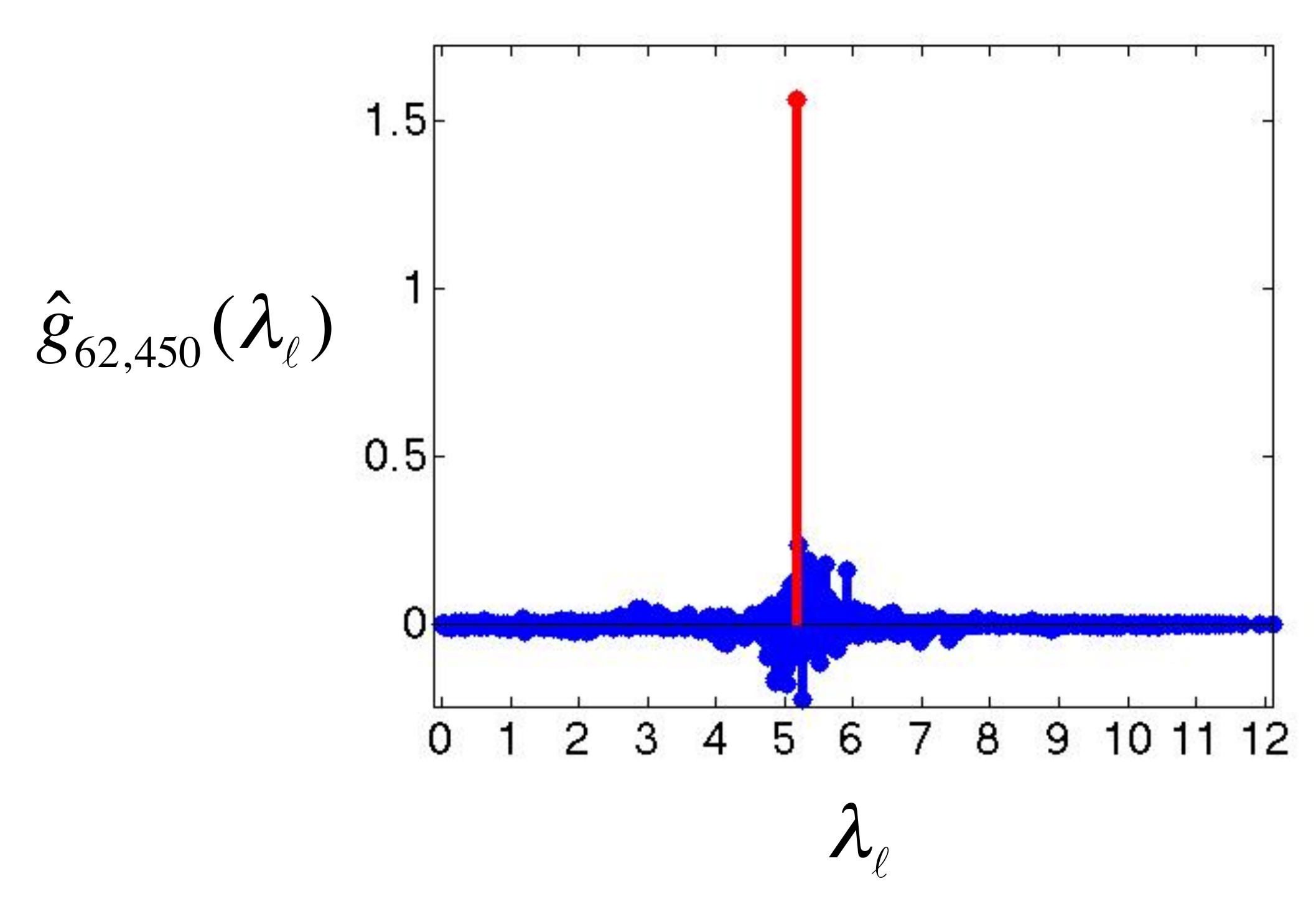}}  
\centerline{\small{~~~~~~~~~~~~~~(b)}}  
\end{minipage} 
\hfill
\begin{minipage}[b]{.31\linewidth}
\hspace{.3in} \centerline{{$g_{62,983}$}}
\centerline{\hspace{.5in}\includegraphics[width=.85\linewidth]{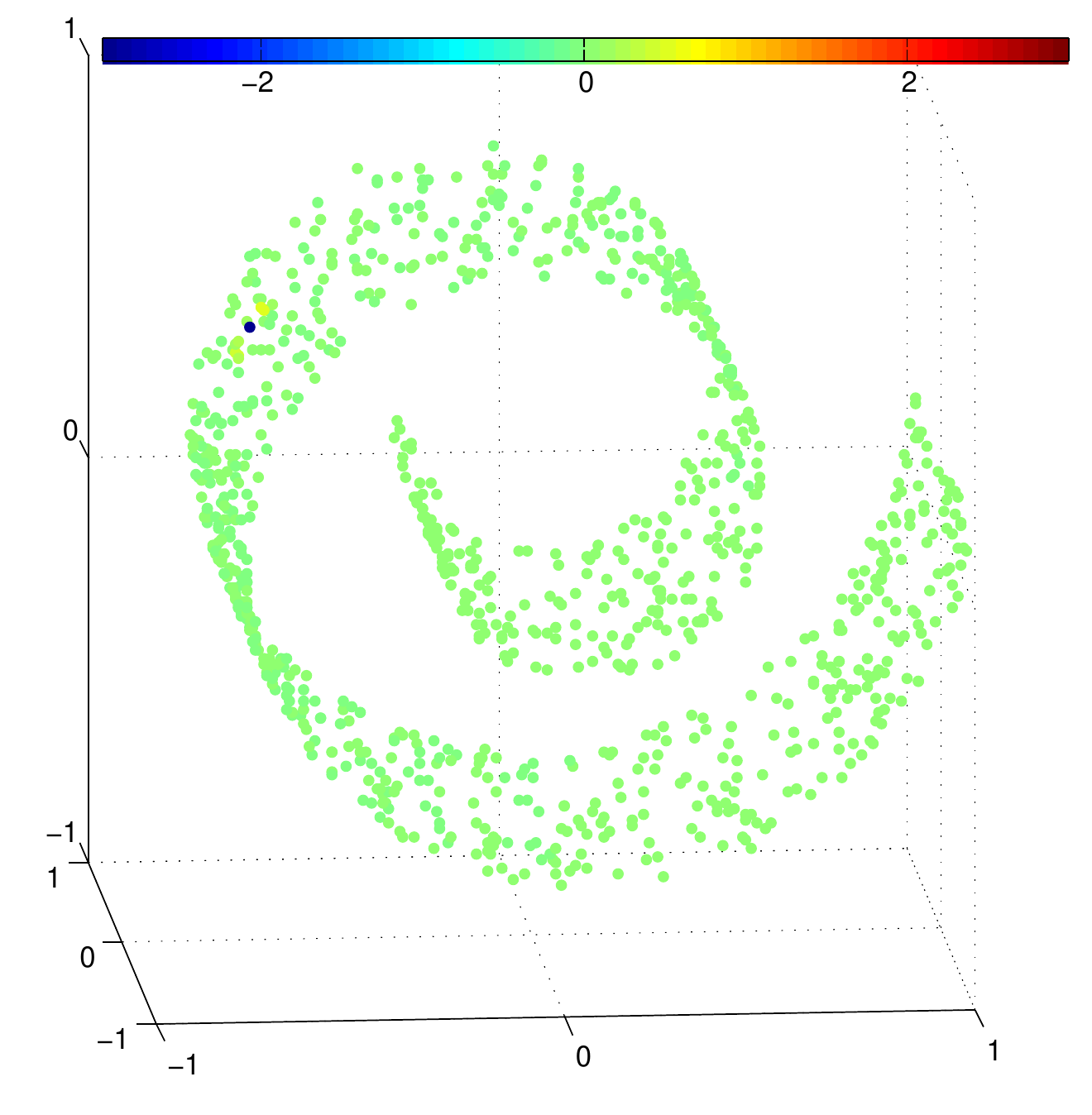}} 
\centerline{\includegraphics[width=\linewidth]{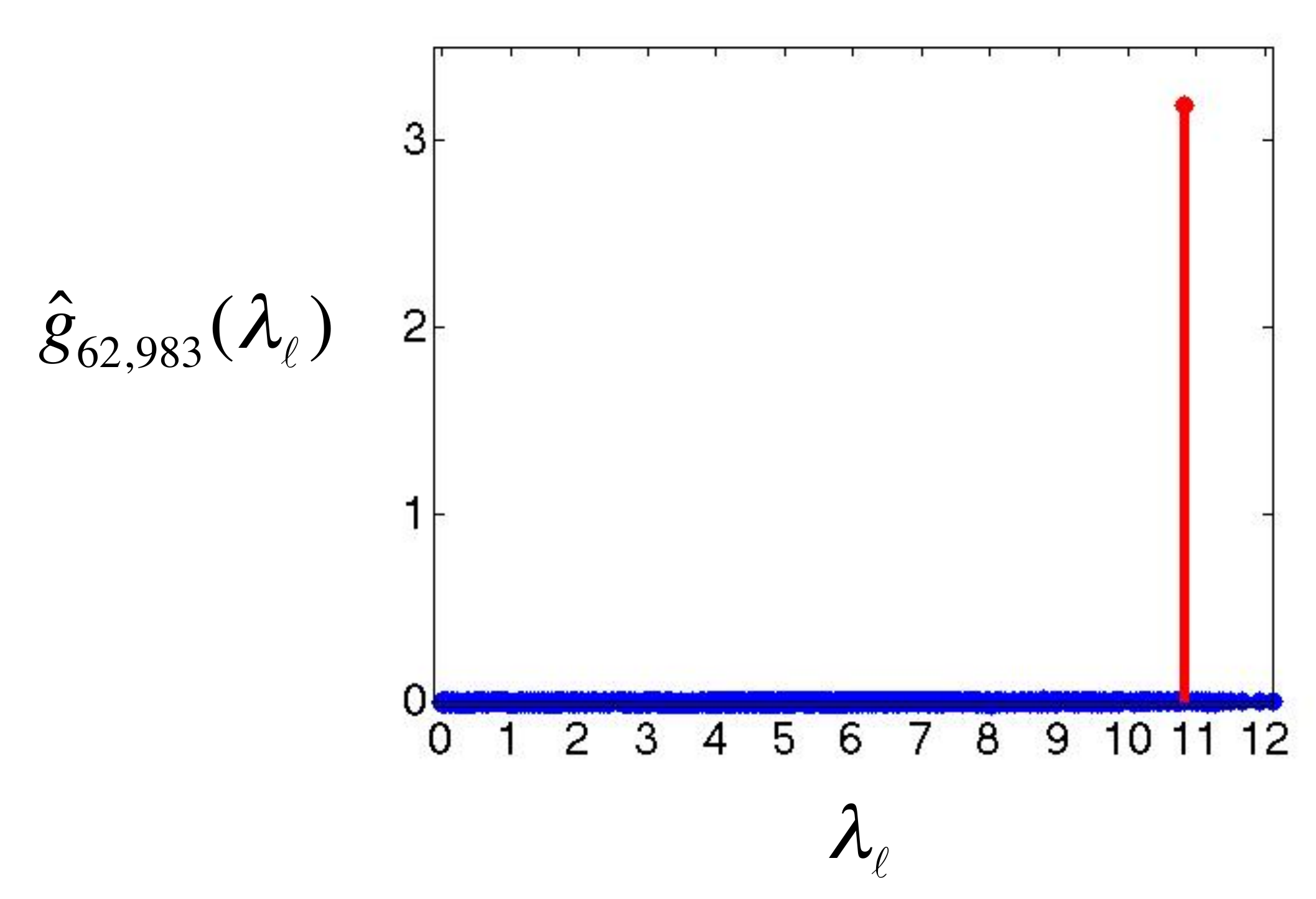}} 
\centerline{~~~~~~~~~~~~~\small{(c)}}
\end{minipage} 
\hspace{.15in}\hfill
\caption {Three different windowed graph Fourier atoms on the Swiss roll, shown in both domains.}
 \label{Fig:atoms_jointly}
\end{figure}





%
%



\subsection{Limitations}
\label{Se:limitations}
In this section, we briefly discuss a few limitations of the proposed windowed graph Fourier transform.

\subsubsection{Computational Complexity}\label{Se:comp}
While the exact computation of the windowed graph Fourier transform coefficients via \eqref{Eq:wgft_atom_comp} and \eqref{Eq:wgft_comp} is feasible for smaller graphs (e.g., less than 10,000 vertices), the computational cost may be prohibitive for much larger graphs. Therefore, it would be of interest to develop an approximate computational method that scales more efficiently with the size of the graph. Recall that 
\begin{align}\label{Eq:wgft_comp_summary}
Sf(i,k)=\ip{f}{g_{ik}}=\ip{f}{M_k T_i g}=\ip{f\left(T_i g\right)^*}{\chi_k}=\widehat{f\left(T_i g\right)^*}(\lambda_k).
\end{align}
The quantity $T_i g$ in the last term of \eqref{Eq:wgft_comp_summary} can be approximately computed in an efficient manner via the Chebyshev polynomial method of \cite[Section 6]{sgwt}, and therefore $f\left(T_i g\right)^*$ can be approximately computed in an efficient manner. Thus, if there was a fast approximate graph Fourier transform, we could apply that to the fast approximation of $f\left(T_i g\right)^*$ in order to approximately compute the windowed graph Fourier transform coefficients $\left\{Sf(i,k)\right\}_{k=0,1,\ldots,N-1}$. Unfortunately, we are not yet aware of a good fast approximate graph Fourier transform method.

\subsubsection{Lack of a Tight Frame}
As discussed in Sections \ref{Se:frame_bounds} and \ref{Se:spec_examples}, the collection of windowed graph Fourier atoms need not form a tight frame, meaning that the spectrogram can not always be interpreted as an energy density function. Furthermore, the lack of a tight frame may lead to (i) less numerical stability when reconstructing a signal from (potentially noisy) windowed graph Fourier transform coefficients \cite{christensen,kovacevic_frames1,kovacevic_frames2}, or (ii) slower computations, for example when computing proximity operators in convex regularization problems \cite{combettes_chapter}.

\subsubsection{No Guarantees on the Joint Localization of the Atoms in the Vertex and Graph Spectral Domains}

Thus far, we have seen that (i) if we translate a smooth kernel $\hat{g}$ to vertex $i$, the resulting signal $T_i g$ will be localized around vertex $i$ in the vertex domain (Section \ref{Se:trans_loc}); (ii) if we modulate a kernel $\hat{g}$ that is localized around 0 in the graph spectral domain, the resulting kernel $\widehat{M_k g}$ will be localized around $\lambda_k$ in the graph spectral domain (Section \ref{Se:modulation}); and (iii) a windowed graph Fourier atom, $g_{ik}$, is often jointly localized around vertex $i$ in the vertex domain and frequency $\lambda_k$ in the graph spectral domain (e.g., Figure \ref{Fig:atoms_jointly}). In classical time-frequency analysis, the windowed Fourier atoms $g_{u,\xi}$ are all jointly localized around time $u$ and frequency $\xi$. So we now ask 
whether the windowed graph Fourier atoms are always jointly localized around vertex $i$ and frequency $\lambda_k$? The answer is no, and the reason once again follows from the possibility of localized graph Laplacian eigenvectors. For a smooth window $\hat{g}$, the translated window $T_i g$ is indeed localized around vertex $i$; however, $g_{ik}=\sqrt{N}\chi_k .* (T_i g)$ may not be localized around vertex $i$ when $\chi_k(n)$ is close to zero for all vertices $n$ in a neighborhood around $i$. One such example is shown in Figure \ref{Fig:joint_counter}. Similarly, in order for $\hat{g}_{ik}$ to be localized around frequency $\lambda_k$ in the graph spectral domain, it suffices for $\widehat{T_i g}$ to be localized around 0 in the graph spectral domain. However, 
\begin{align*}
\widehat{T_i g}(\lambda_{\l})=\sqrt{N}\hat{g}(\lambda_{\l})\chi_{\l}^*(i),
\end{align*}
and, therefore, it is possible that the multiplication by a graph Laplacian eigenvector changes the localization of the translated window in the graph spectral domain. In classical time-frequency analysis, these phenomena never occur, because the complex exponentials are always delocalized.

\begin{figure}[h]
\centering
\hfill
\begin{minipage}[b]{.31\linewidth}
\centerline{{$~~T_{48} g$}}
\centerline{\includegraphics[width=\linewidth]{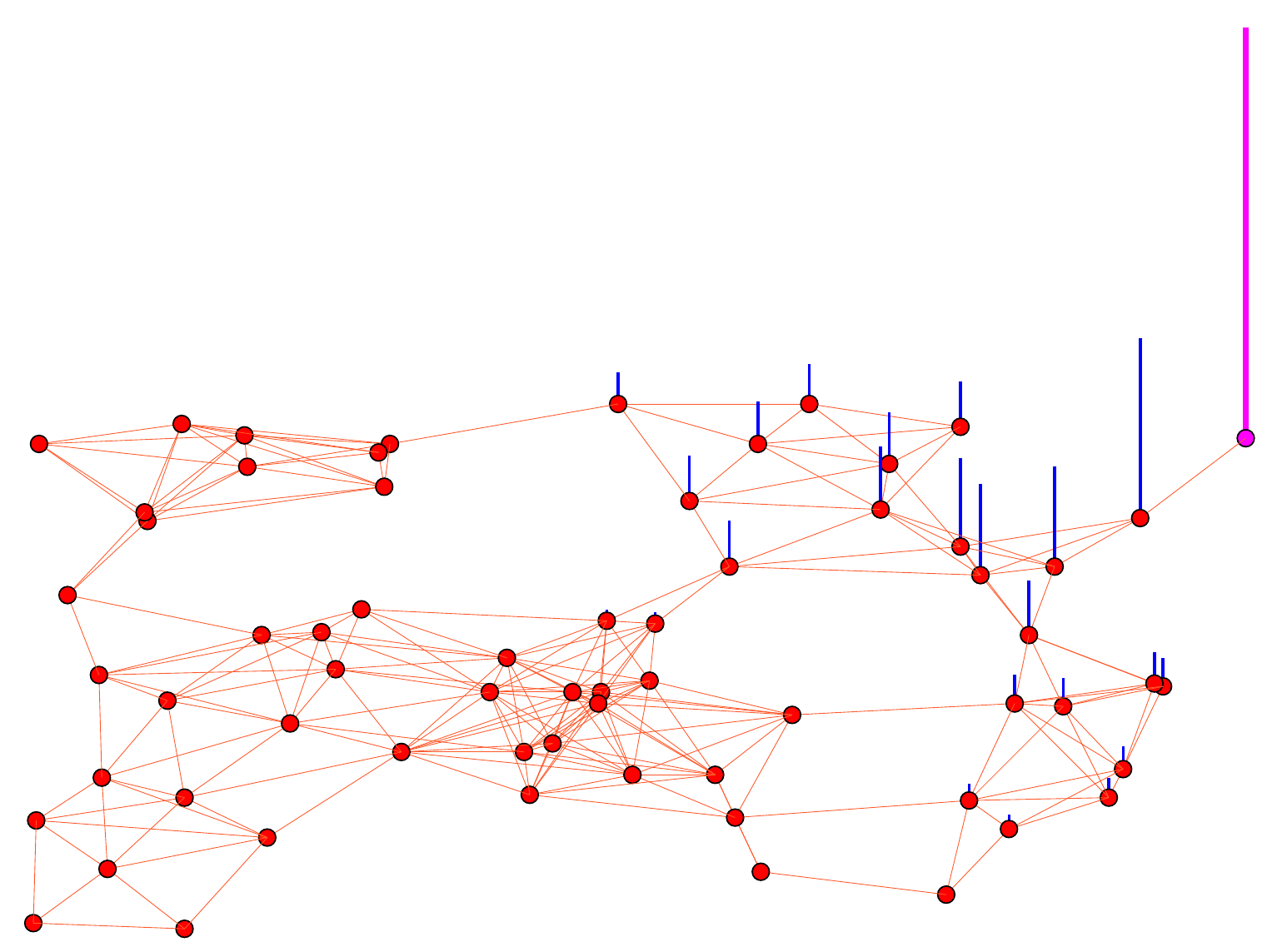}}
\centerline{\small{(a)}}
\end{minipage} 
\hfill
\begin{minipage}[b]{.31\linewidth}
\centerline{{$~~\chi_{N-2}$}}
\centerline{\includegraphics[width=\linewidth]{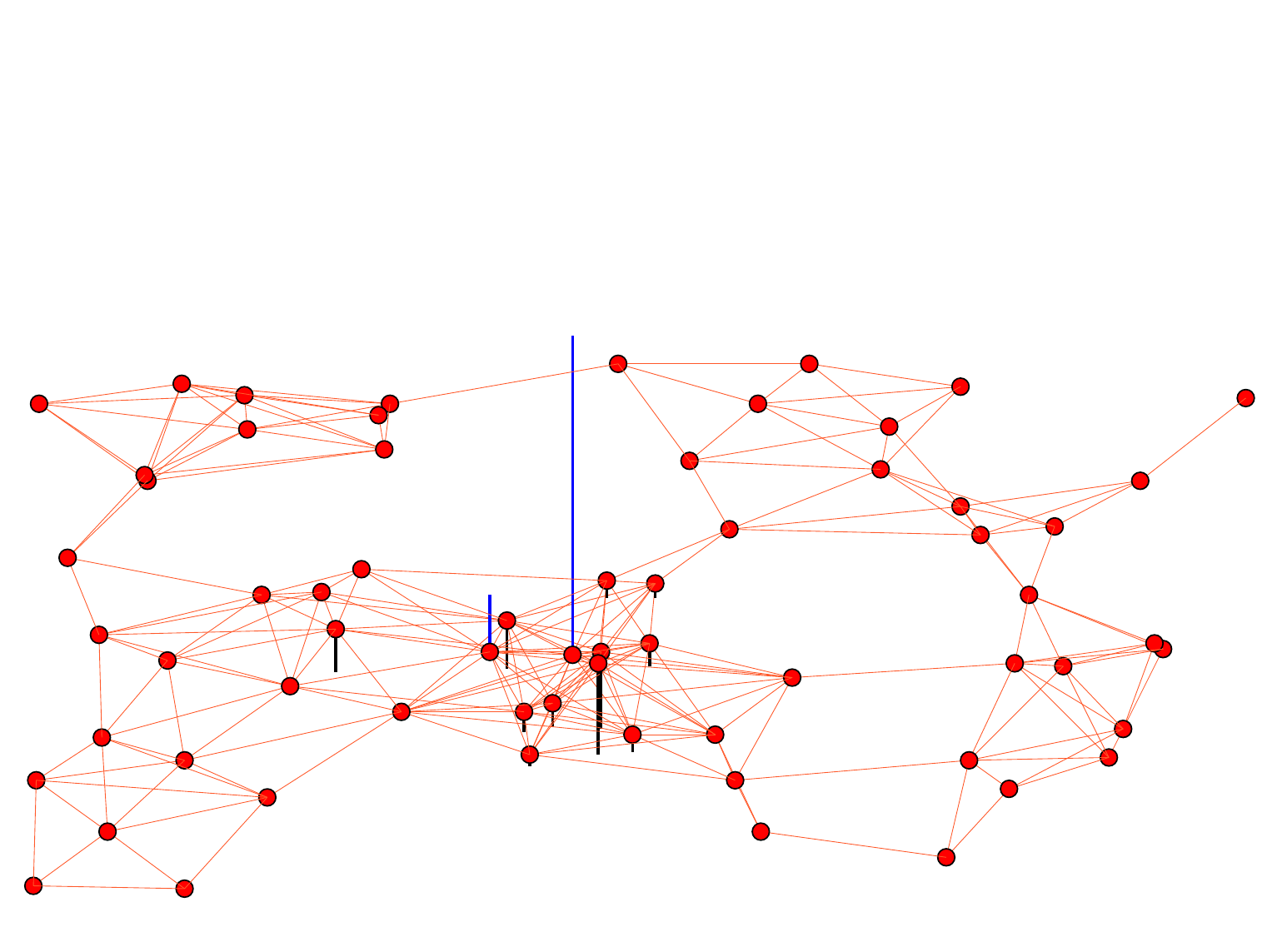}}  
\centerline{\small{(b)}}  
\end{minipage} 
\hfill
\begin{minipage}[b]{.31\linewidth}
\centerline{{$~~g_{48,N-2}$}}
\centerline{\includegraphics[width=\linewidth]{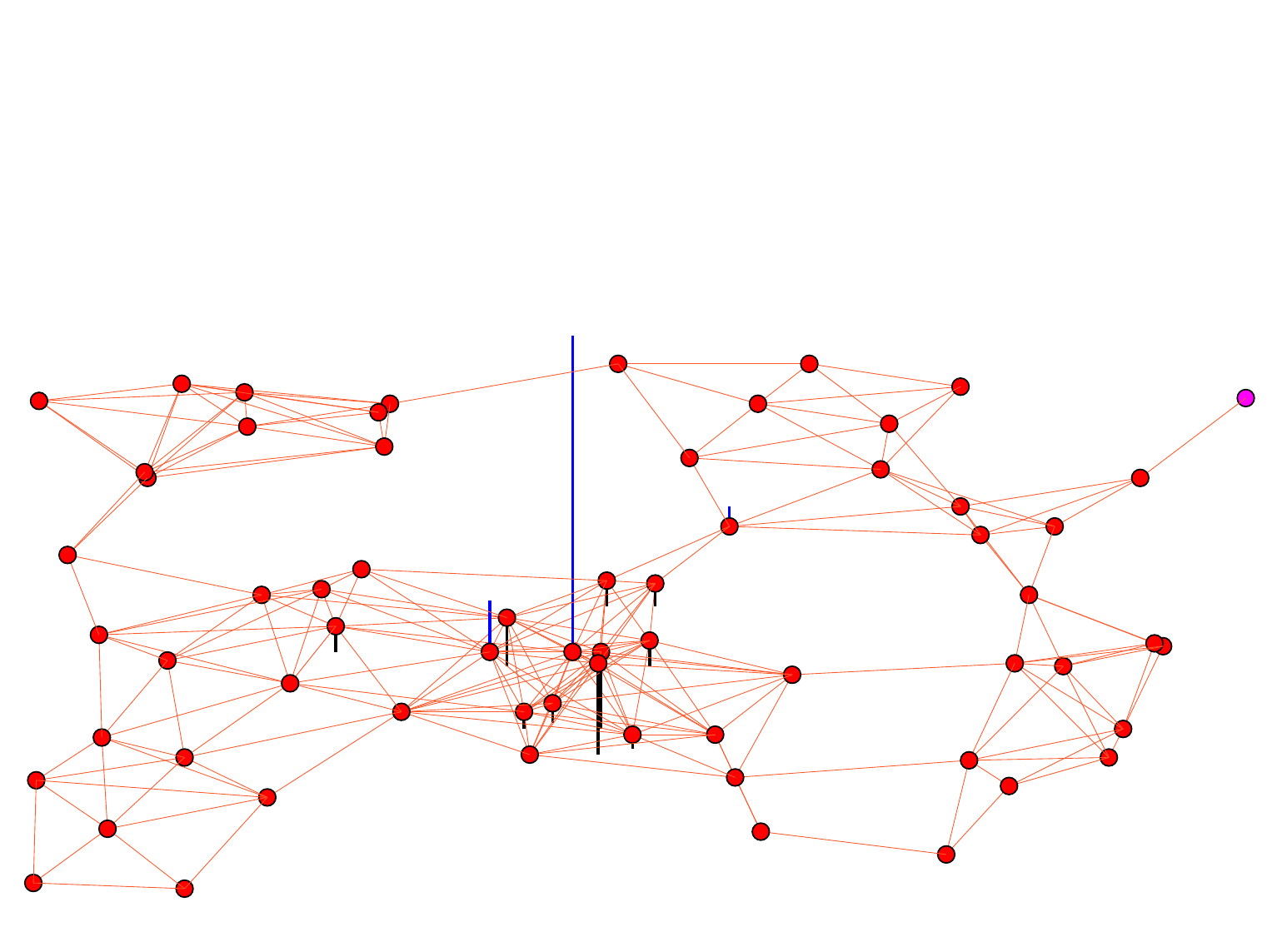}}
\centerline{\small{(c)}}
\end{minipage} 
\hfill
\caption {A windowed graph Fourier atom that is not localized around its center vertex. (a) The translated window $T_{48} g$ is localized around the center vertex 48, which is shown in magenta. (b) The graph Laplacian eigenvector $\chi_{N-2}$ is highly localized around a different vertex in the graph and is extremely close to zero in the neighborhood of vertex 48. (c) Therefore, when we modulate the translated window $T_{48}g$ by $\chi_{N-2}$, the resulting atom $g_{48,N-2}$ is not centered around vertex 48. Note that the three figures are drawn on different scales; the magnitudes of the largest components of the signals are 1.37, 0.59, and 0.15, respectively.}
 \label{Fig:joint_counter}
\end{figure}

Despite the lack of guarantees on joint localization, the windowed graph Fourier transform is still a useful analysis tool.
First, if the coherence $\mu$ is low (close to $\frac{1}{\sqrt{N}}$), the graph Laplacian eigenvectors are delocalized, the atoms are jointly localized in the vertex and graph spectral domains, and much of the intuition from classical time-frequency analysis carriers over to the graph setting. Even when the coherence is close to 1, however, it often happens that the majority of the atoms are in fact jointly localized in time and frequency. This is because only those atoms whose computations include highly localized eigenvectors are affected.
We have observed empirically that there tends to be only a few graph Laplacian eigenvectors, most commonly those associated with the higher frequencies (eigenvalues close to $\lambda_{\max}$). Moreover, if $\chi_k$ is highly localized around a vertex $j$, then $Sf(i,k)$ will be close to 0 for any vertex $i$ not close to $j$, so it is not particularly problematic that $g_{ik}$ may not be localized around $i$ in the vertex domain.




\section{Conclusion and Future Work}
\label{Se:conclusion}



We defined generalized notions of translation and modulation through multiplication with a graph Laplacian eigenvector in the graph spectral and vertex domains, respectively. We leveraged these generalized operators to design a windowed graph Fourier transform, which enables vertex-frequency analysis for signals on graphs. We showed that when the chosen window is smooth in the graph spectral domain, translated windows are localized in the vertex domain. Moreover, when the chosen window is localized around zero in the graph spectral domain, 
the modulation operator is close to a translation in the graph spectral domain.
If we apply this windowed graph Fourier transform to a signal with frequency components that vary along a path graph, the resulting spectrogram matches our intuition from classical discrete-time signal processing.
Yet, our construction is fully generalized and can be applied to analyze signals on any undirected, connected, weighted graph. The example in Figure \ref{Fig:sensor} 
shows that the windowed graph Fourier transform may be a valuable tool for extracting information from signals on graphs, as structural properties of the data that are hidden 
in the vertex domain may become obvious in the transform domain.    
%



One line of future work is to continue to improve the localization results for the translated kernels presented in Section \ref{Se:trans_loc}, preferably by incorporating the graph weights. This issue is closely related to the study of both the localization of eigenvectors and recent work in the theory of matrix functions \cite{higham}. In particular, it is related to the off-diagonal decay of entries of a matrix function $g(\L)$, as studied in \cite{benzi1,benzi2}. To our knowledge, existing results in this area also do not incorporate the entries of the matrix $\L$ (other than through the eigenvalues), but rather depend primarily on the sparsity pattern of $\L$. For more precise numerical localization results, numerical linear algebra researchers have also turned to quadrature methods to approximate the quantity $\delta_i^* g(\L) \delta_j$ (see, e.g., \cite{golub_quadrature} and references therein). 

Motivated by the spirit of vertex-frequency analysis introduced in this paper, we are also investigating a new, more computationally efficient dictionary design method to generate tight frames of atoms that are jointly localized in the vertex and graph spectral domains. 

\section{Appendix}

\subsection{The Normalized Laplacian Graph Fourier Basis Case}
We now briefly consider the case when the normalized graph Laplacian eigenvectors are used as the Fourier basis, and revisit the definitions and properties of the generalized translation and modulation operators from Sections \ref{Se:operators} and \ref{Se:modulation}. Throughout we use a ${\tilde{~}}$ to denote the corresponding quantities derived from the normalized graph Laplacian $\tilde{\L}$. To prove many of the following properties, we use the fact that for connected graphs
\begin{align*}
\tilde{\chi}_0(n)=\frac{\sqrt{d_n}}{\norm{\sqrt{d}}_2}=\frac{\sqrt{N}\sqrt{d_n}}{\norm{\sqrt{d}}_2}\chi_0(n),
\end{align*}
where the square root in $\sqrt{d}$ is applied component-wise. 

\subsubsection{Generalized Convolution and Translation in the Normalized Laplacian Graph Fourier Basis}


 We can keep the definition \eqref{Eq:gen_convolution} of the generalized convolution, with the normalized graph Laplacian eigenvalues and eigenvectors replacing those of the combinatorial graph Laplacian. Statements 1-7 in Proposition \ref{Prop:conv_prop} are still valid; however, \eqref{Eq:integration} becomes
\begin{align*}
\sum_{n=1}^N (f \ast g)(n) \sqrt{d_n} = \norm{\sqrt{d}}_2 \hat{f}(0)\hat{g}(0) = \frac{1}{\norm{\sqrt{d}}_2}\left[\sum_{n=1}^N f(n)\sqrt{d_n}  \right] \left[\sum_{n=1}^N g(n)\sqrt{d_n} \right].
\end{align*}
Accordingly, 
we can redefine the generalized translation operator as
\begin{align*}
(\tilde{T}_i f)(n):=\norm{\sqrt{d}}_2 (f \ast \delta_i)(n) = \norm{\sqrt{d}}_2 \sum_{\l=0}^{N-1} \hat{f}\left(\tilde{\lambda}_{\l}\right)\tilde{\chi}_{\l}^*(i)\tilde{\chi}_{\l}(n),
\end{align*}
so that Property 3 of Corollary \ref{Co:trans_conv} becomes
\begin{align*}
\sum_{n=1}^{N} (\tilde{T}_i f)(n) \sqrt{d_n} = \sqrt{d_i} \sum_{n=1}^N f(n) \sqrt{d_n},
\end{align*}
and for any $f \in \Rbb^N$, Lemma \ref{Le:trans_norm_bounds} becomes
\begin{align*}
\sqrt{d_i}|\hat{f}(0)| \leq \norm{\tilde{T}_i f}_2 \leq 
\tilde{\nu}_i \norm{\sqrt{d}}_2 \norm{f}_2 \leq  \tilde{\mu} \norm{\sqrt{d}}_2 \norm{f}_2.
\end{align*}
Because they only depend on the graph structure and not the specific graph weights, the localization results of Theorem \ref{Th:trans_loc}, Corollary \ref{Co:trans_loc_diff}, and Corollary \ref{Co:loc_extra} also hold, with the constant $\sqrt{N}$ replaced by $\norm{\sqrt{d}}_2$ and an extra factor of $\sqrt{d_i}$ in the denominator of \eqref{Eq:loc_bound} and \eqref{Eq:dis_loc_bound}.

\subsubsection{Generalized Modulation in the Normalized Laplacian Graph Fourier Basis}

When we use the normalized Laplacian eigenvectors as a graph Fourier basis instead of the combinatorial Laplacian eigenvectors, we can define a generalized modulation as
\begin{align*}
(\tilde{M}_k f)(n):=f(n)\frac{\tilde{\chi}_k(n)}{\tilde{\chi}_0(n)}= f(n) \frac{\tilde{\chi}_k(n)\norm{\sqrt{d}}_2}{\sqrt{d_n}},
\end{align*} 
so that $\tilde{M}_0$ is also the identity operator and $\widehat{\tilde{M}_k\delta_0}(\tilde{\lambda}_{\l})=\delta_0(\tilde{\lambda}_{\l}-\tilde{\lambda}_{k})$; i.e., $\tilde{M}_k$ maps the graph spectral component $\hat{f}(0)$ from eigenvalue $\tilde{\lambda}_0$ to eigenvalue $\tilde{\lambda}_k$ in the graph spectral domain. The following theorem on the localization of a modulated kernel in the graph spectral domain follows essentially the same line of argument as Theorem \ref{Th:mod_trans}, with $\frac{1}{\tilde{\chi}_0(n)}$ replacing $\sqrt{N}$ in the proof. 
\begin{theorem} \label{Th:norm_mod}
Given a weighted graph $\G$ with $N$ vertices, 
if for some $\gamma > 0$, a 
kernel $\hat{f}$ 
satisfies 
\begin{align} \label{Eq:gsum_cond}
\sum_{\l=1}^{N-1}{\tilde{\mu}_{\l}\left|\hat{f}\left(\tilde{\lambda}_{\l}\right)\right|} \leq \frac{\sqrt{d_{\min}}~|\hat{f}(0)|}{\norm{\sqrt{d}}_2~(1+\gamma)},
\end{align}
 then
 \begin{align}\label{Eq:mod_trans_result}
 \left|\widehat{\tilde{M}_k f}\left(\tilde{\lambda}_{k}\right)\right| \geq \gamma \left|\widehat{\tilde{M}_k f}\left(\tilde{\lambda}_{\l}\right)\right|~\hbox{ for all }\l\neq k.
 \end{align}
\end{theorem}

%
\subsubsection{Example: Resolution Tradeoff in the Normalized Laplacian Graph Fourier Basis}
We once again consider a heat kernel of the form $\hat{g}(\tilde{\lambda})=e^{-\tau \tilde{\lambda}}$. First, for the spread of a translated kernel in the vertex domain, we can replace the upper bound \eqref{Eq:heat_spread_vertex3} by
\begin{align}\label{Eq:heat_spread_vertex3_norm}
\Delta_i^2(\tilde{T}_i g) \leq \norm{\sqrt{d}}_2^2 \tau^2 \exp\left(\frac{\tau^2}{4(d_{\max}-1)}\right).
\end{align}


Next, our localization result on the generalized modulation, Theorem \ref{Th:norm_mod}, tells us that 
\begin{align}\label{Eq:mod_wts}
\sum_{\l=1}^{N-1} e^{-\tau\tilde{\lambda}_{\l}} \leq \frac{\sqrt{d_{\min}}}{\tilde{\mu}\norm{\sqrt{d}}_2~(1+\gamma)}
\end{align}
for some $\gamma > 0$ implies $|\widehat{\tilde{M}_k g}(\tilde{\lambda}_{k})| \geq \gamma |\widehat{\tilde{M}_k g}(\tilde{\lambda}_{\l})|$ for all $\l$ not equal to $k$. For a graph $\G$ with known isoperimetric dimension (see, e.g., \cite{chung_sobolev}, \cite[Chapter 11]{chung}), the following result upper bounds the left-hand side of \eqref{Eq:mod_wts}.
\begin{theorem}[Chung and Yau, {\cite[Theorem 7]{chung_sobolev}}]
The normalized graph Laplacian eigenvalues of a graph $\G$ satisfy 
\begin{align}\label{Eq:chung_bound}
\sum_{\l=1}^{N-1} e^{-\tau \tilde{\lambda}_{\l}} \leq \frac{C_{\delta}\norm{d}_1}{\tau^{\frac{\delta}{2}}},
\end{align}
where for every subset $\V_1 \subset \V$, $\delta$, the isoperimetric dimension of $\G$,
and $c_{\delta}$, the isoperimetric constant, satisfy
\begin{align*}
|\E(\V_1,\V_1^c)| \geq c_{\delta} \left[\sum_{n \in \V_1}d_n\right]^{\frac{\delta-1}{\delta}},
\end{align*}
and  $C_{\delta}$ is a constant that only depends on $\delta$. 
\end{theorem}


Combining \eqref{Eq:mod_wts} and \eqref{Eq:chung_bound}, for a fixed $\tau$,
\begin{align}\label{Eq:mod_tau_kappa}
\left|\widehat{\tilde{M}_k g}\left(\tilde{\lambda}_{k}\right)\right| \geq \left(\frac{\sqrt{d_{\min}}~\tau^{\frac{\delta}{2}}}{\tilde{\mu} C_{\delta}\norm{\sqrt{d}}_2
\norm{d}_1}-1\right)\left|\widehat{\tilde{M}_k g}\left(\tilde{\lambda}_{\l}\right)\right|,~\forall \l \neq k.
\end{align}
Similarly, to ensure $\left|\widehat{\tilde{M}_k g}\left(\tilde{\lambda}_{k}\right)\right| \geq \gamma \left|\widehat{\tilde{M}_k g}\left(\tilde{\lambda}_{\l}\right)\right|,~\forall \l \neq k$ for a desired $\gamma>0$, it suffices to choose the diffusion parameter of the heat kernel as
\begin{align}\label{Eq:mod_tau_kappa2}
\tau \geq \left(\frac{\tilde{\mu} C_{\delta}\norm{\sqrt{d}}_2\norm{d}_1(1+\gamma)}{\sqrt{d_{\min}}}\right)^{\frac{2}{\delta}}.
\end{align} 
Comparing \eqref{Eq:mod_tau_kappa} and \eqref{Eq:mod_tau_kappa2} to \eqref{Eq:heat_spread_vertex3_norm}, we see that as the diffusion parameter $\tau$ increases, our guarantee on the localization of $\widehat{M_k g}$ around frequency $\lambda_k$ in the graph spectral domain improves, whereas the guarantee on the localization of $T_i g$ around vertex $i$ in the vertex domain becomes weaker, and vice versa.


\subsection{Alternative Definition of Generalized Modulation} \label{Se:mod_alt}
The classical modulation operator corresponds to translation in the spectral domain. Therefore, a second approach to generalize modulation to the graph setting is to define a new weighted graph $\check{\G}$ on the graph Laplacian spectrum and then define the generalized modulation on ${\G}$ as a generalized translation on $\check{\G}$. More specifically, we first define a new graph $\check{\G}$, whose vertices are the eigenvalues of the original graph Laplacian ${\L}$. One simple choice of graphs is a weighted path graph, where each eigenvalue $\lambda_{\l}$ from the original graph ${\G}$ is connected only to $\lambda_{\l-1}$ and $\lambda_{\l+1}$, and the weights are inversely proportional to the distances between neighboring eigenvalues. Other possibilities include assigning exponential weights based on the distance between eigenvalues, or connecting each eigenvalue to multiple neighbors on each side and assigning weights according to a thresholded weighting function like \eqref{Eq:gkw}. Next, we form the Laplacian $\check{L}$ on this new graph $\check{\G}$, and denote its eigenpairs by $(\check{\lambda}_j,\check{\chi}_j)$. Then generalized modulation on $\G$ is defined as generalized translation on $\check{\G}$: 
\begin{align}\label{Eq:alt_mod}
\widehat{\left({M}_k f\right)}(\lambda_{\l}) := \left(\check{T}_{k} \hat{f}\right)(\l) = \sqrt{N} \hat{\hat{f}}(\check{\L})_{k,\l} =  \sqrt{N} \sum_{j=0}^{N-1} \hat{\hat{f}}\left(\check{\lambda}_j\right)\check{\chi}_j^*(k) \check{\chi}_j(\l), 
\end{align}
where we have indexed the vertices in $\check{\G}$ as $\l=0,1,\ldots,N-1$, rather than the usual 1 to $N$. In \eqref{Eq:alt_mod}, $f$ lives on the vertices of the original graph $\G$, $\hat{f}$ lives on both the spectrum $\sigma(\L)$ of the original graph and the vertices of $\check{\G}$, and $\hat{\hat{f}}$ (a graph Fourier transform of $f$ with respect to the Laplacian eigenvectors of $\L$ followed by a graph Fourier transform of $\hat{f}$ with respect to the Laplacian eigenvectors of $\check{L}$) lives on the spectrum $\sigma(\check{L})$ of $\check{\G}$. Taking an inverse graph Fourier transform of \eqref{Eq:alt_mod} (with respect to the eigenvectors of the original graph Laplacian $\L$) yields
\begin{align*}
\left({M}_k f\right)(n) = \sqrt{N} \sum_{\l=0}^{N-1} \sum_{j=0}^{N-1} \hat{\hat{f}}\left(\check{\lambda}_j\right)\check{\chi}_j^*(k) \check{\chi}_j(\l)\chi_{\l}^*(n). 
\end{align*}

We 
can define the 
kernel $f$ in \eqref{Eq:alt_mod} in one of two ways. First, to maintain direct comparability to the generalized modulation, we can define the kernel $\hat{f}(\lambda_{\l})$ directly on $\sigma(\L)$. For example,
in Figure \ref{Fig:mod_alt1}, we let $\hat{f}(\lambda_{\l})=Ce^{\tau \lambda_{\l}}$ (with $C$ chosen such that $\norm{\hat{f}}_2=1$), and then use the definition \eqref{Eq:alt_mod} to modulate a kernel on the Laplacian spectrum of the Minnesota graph.
\begin{figure}[h]
\centering
{\hfill
\begin{minipage}[b]{.23\linewidth}
   \centering
   \centerline{\includegraphics[width=\linewidth]{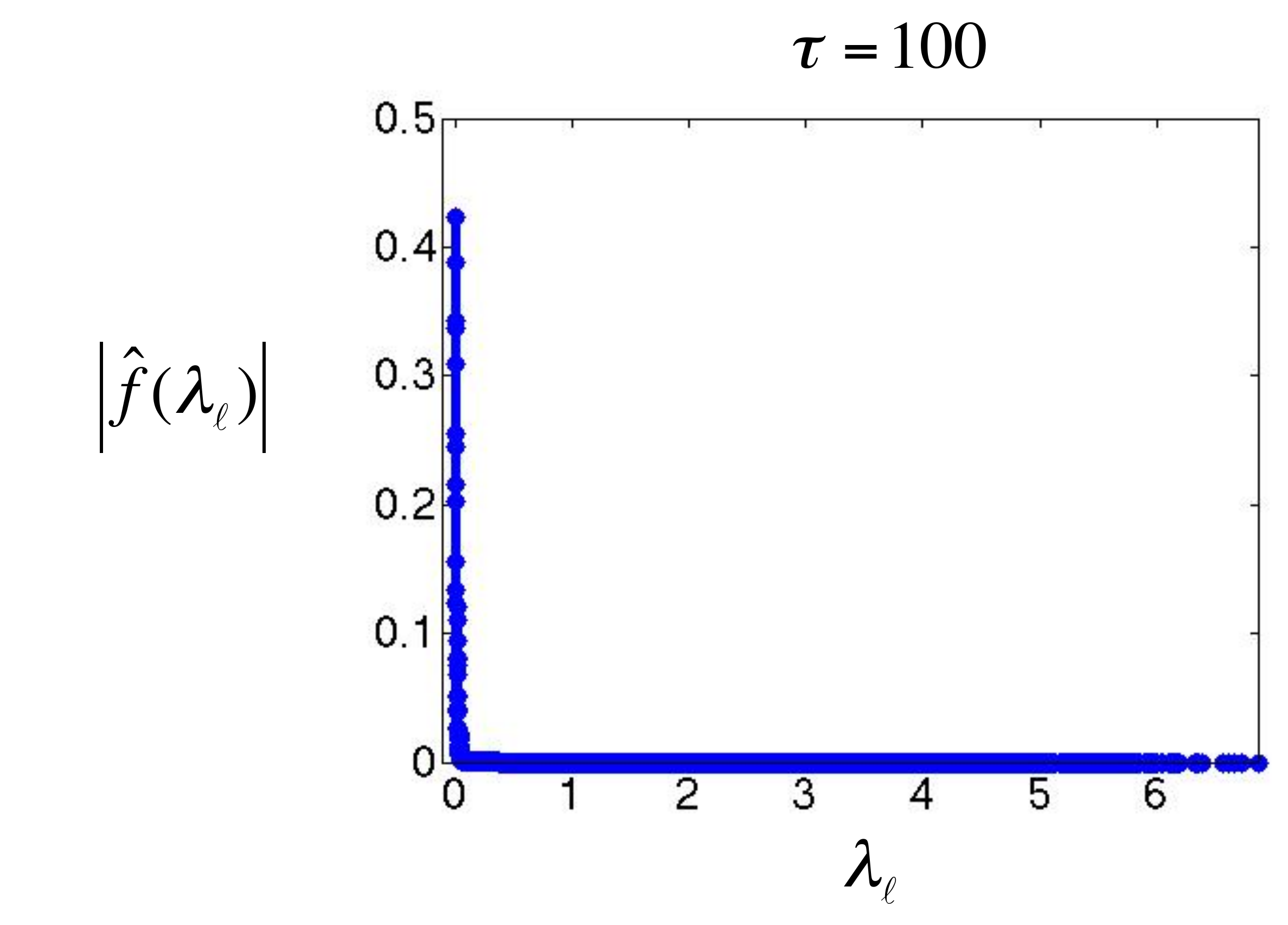}}
\centerline{\small{~~~~~~~~~~~~(a)}}
\end{minipage}
\hfill
\begin{minipage}[b]{.23\linewidth}
   \centering
   \centerline{\includegraphics[width=\linewidth]{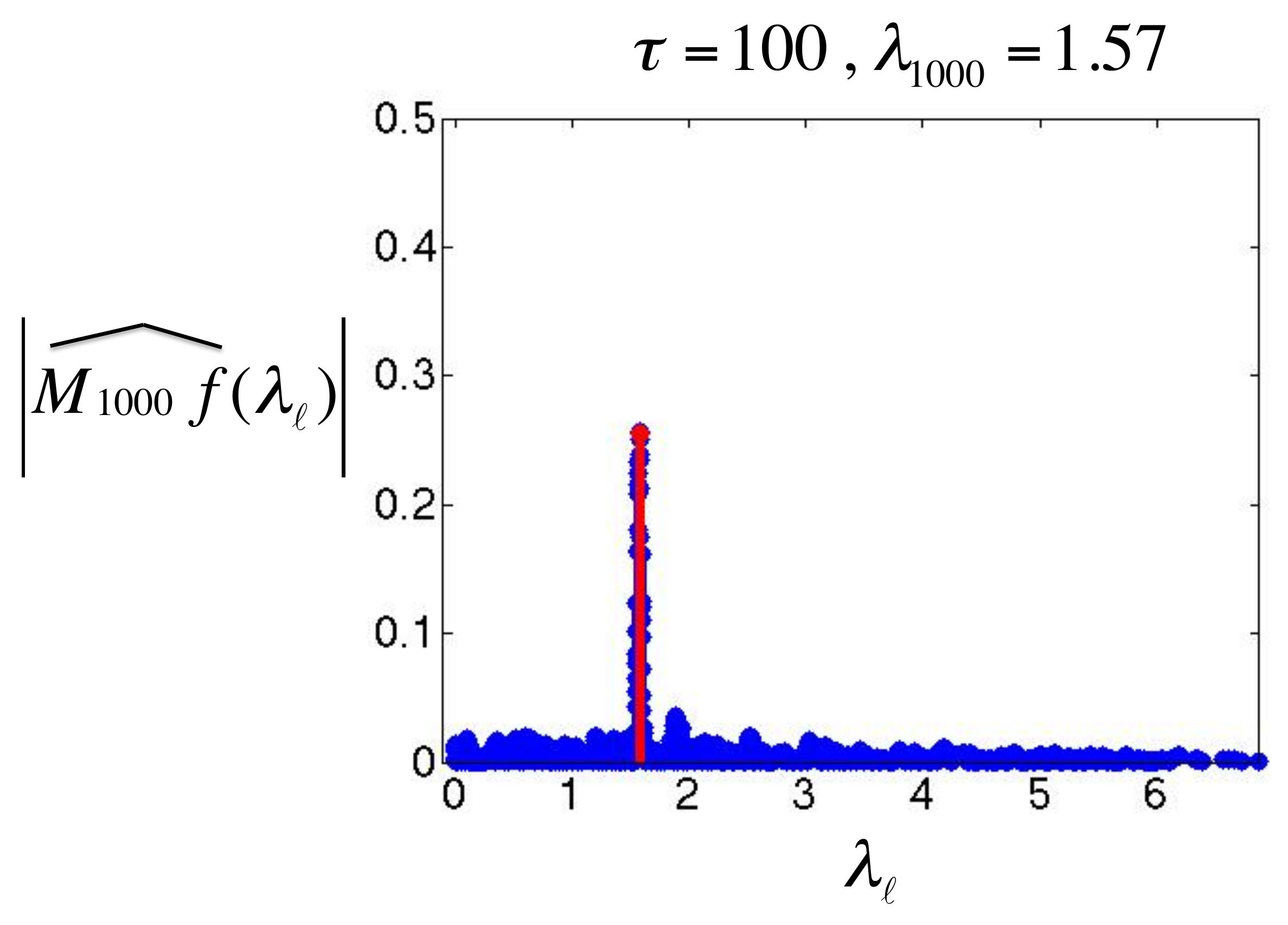}}
\centerline{\small{~~~~~~~~~~~~(b)}}
\end{minipage}
\hfill
\begin{minipage}[b]{.23\linewidth}
   \centering
   \centerline{\includegraphics[width=\linewidth]{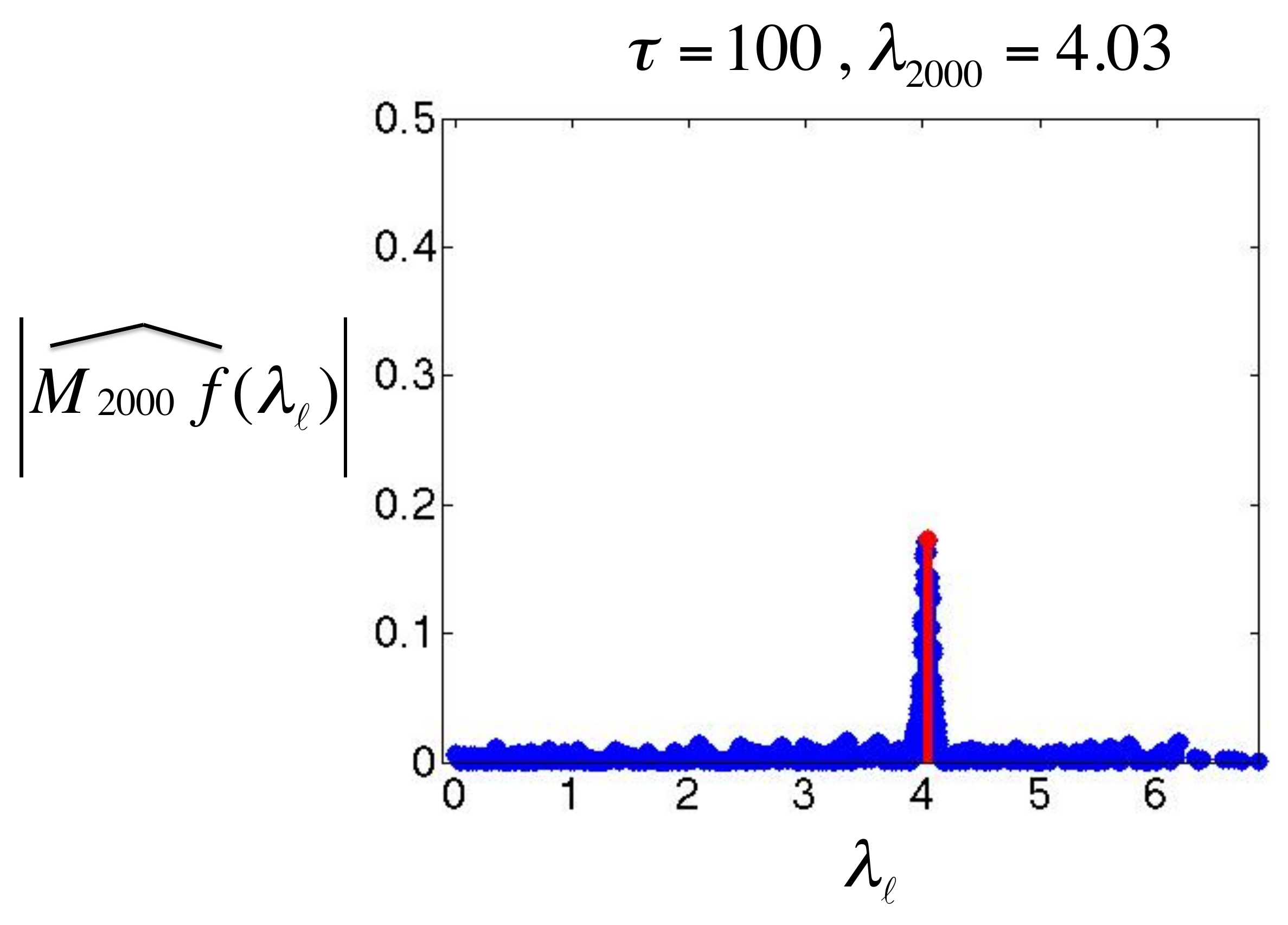}}
\centerline{\small{~~~~~~~~~~~~(c)}}
\end{minipage}
\hfill
\begin{minipage}[b]{.23\linewidth}
   \centering
   \centerline{\includegraphics[width=\linewidth]{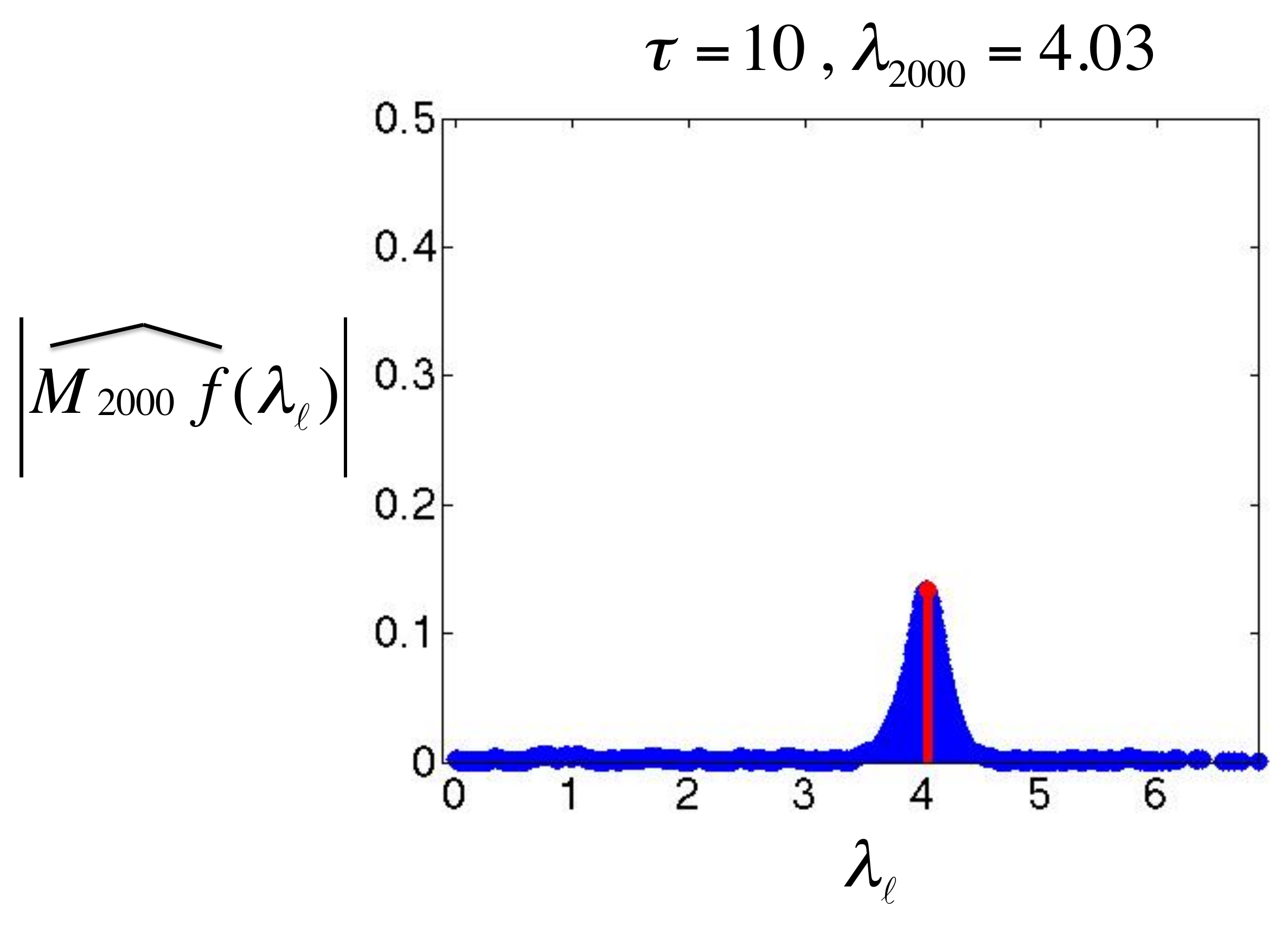}}
\centerline{\small{~~~~~~~~~~~~(d)}}
\end{minipage}
\hfill}
\caption {Generalized modulation on the Minnesota graph $\G$, via generalized translation on $\check{\G}$ with the same kernel $\hat{f}({\lambda}_{\l})=Ce^{-{\tau}{\lambda}_{\l}}$ defined on $\sigma({\L})$ as in Figure \ref{Fig:mod}. The graph $\check{G}$ is a weighted path graph with $\check{W}_{kl}=\frac{1}{|\lambda_k-\lambda_{\l}|+10^{-6}}$. (a) The normalized heat kernel $\hat{f}({\lambda}_{\l})$. (b) and (c) Modulated kernels with two different center frequencies, $\lambda_{1000}$ and $\lambda_{2000}$. In both cases, the modulated kernel is localized around the desired frequency of $\sigma(\L)$ (shown in red). (d) As ${\tau}$ \emph{decreases}, the spread of the modulated kernel around the center frequency in the graph spectral domain of $\G$ generally increases.} 
  \label{Fig:mod_alt1}
\end{figure}

A second option is to define $f$ on the spectrum $\sigma(\check{L})$. For example,
in Figure \ref{Fig:mod_alt2}, we let $\hat{\hat{f}}(\check{\lambda}_{j})=e^{\check{\tau} \check{\lambda}_{j}}$, and then modulate again via \eqref{Eq:alt_mod}. One advantage of defining the kernel in this manner is that we can more easily characterize the spread of the modulated kernel in the graph spectral domain of $\G$, using, e.g., our translation localization results from Section \ref{Se:trans_loc}. For example, with $\hat{\hat{f}}(\cdot)$ a heat kernel and $\check{G}$ a weighted path graph (which has a maximum Laplacian eigenvalue upper bounded by 4) as in Figure \ref{Fig:mod_alt2}, the bound \eqref{Eq:heat_spread_vertex3} becomes
\begin{align*}
\frac{\sum\limits_{\l=0}^{N-1} ({\l}-{k})^2 \left|\widehat{{M}_{k} f} (\lambda_{\l})\right|^2}{\norm{{M}_k f}_2^2} =
\Delta_k^2(\check{T}_{\l} \hat{f})
\leq 8N \check{\tau}^2 e^{\check{\tau}^2},
\end{align*} 
which is close to form of the desired bound on the spread of modulated kernels we mentioned in \eqref{Eq:correct_spectral_spread}, especially for a graph $\G$ whose spectrum is close to uniformly distributed on $[0,\lambda_{\max}]$.

\begin{figure}[h]
\centering
{\hfill
\begin{minipage}[b]{.23\linewidth}
   \centering
   \centerline{\includegraphics[width=\linewidth]{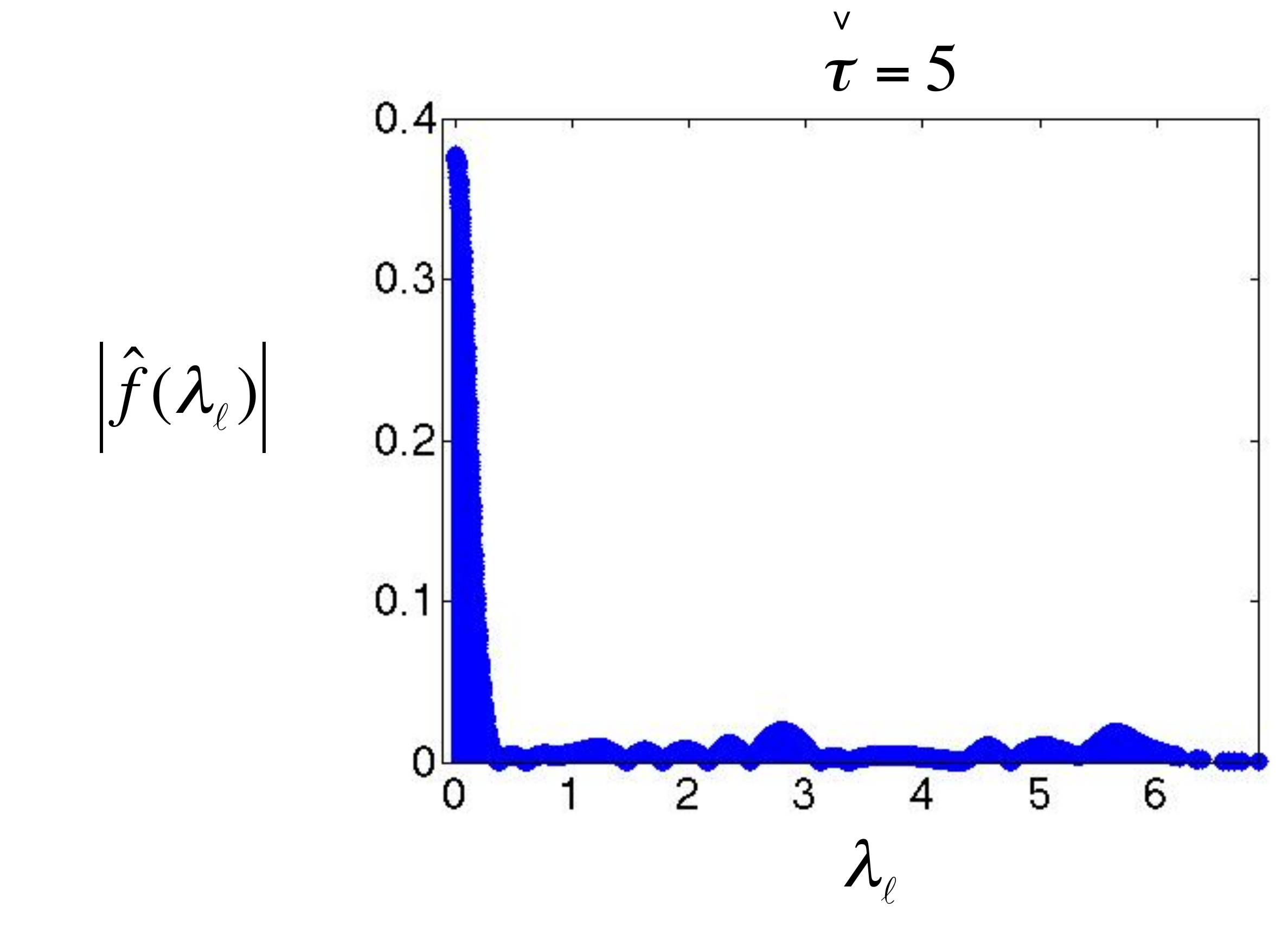}}
\centerline{\small{~~~~~~~~~~~~(a)}}
\end{minipage}
\hfill
\begin{minipage}[b]{.23\linewidth}
   \centering
   \centerline{\includegraphics[width=\linewidth]{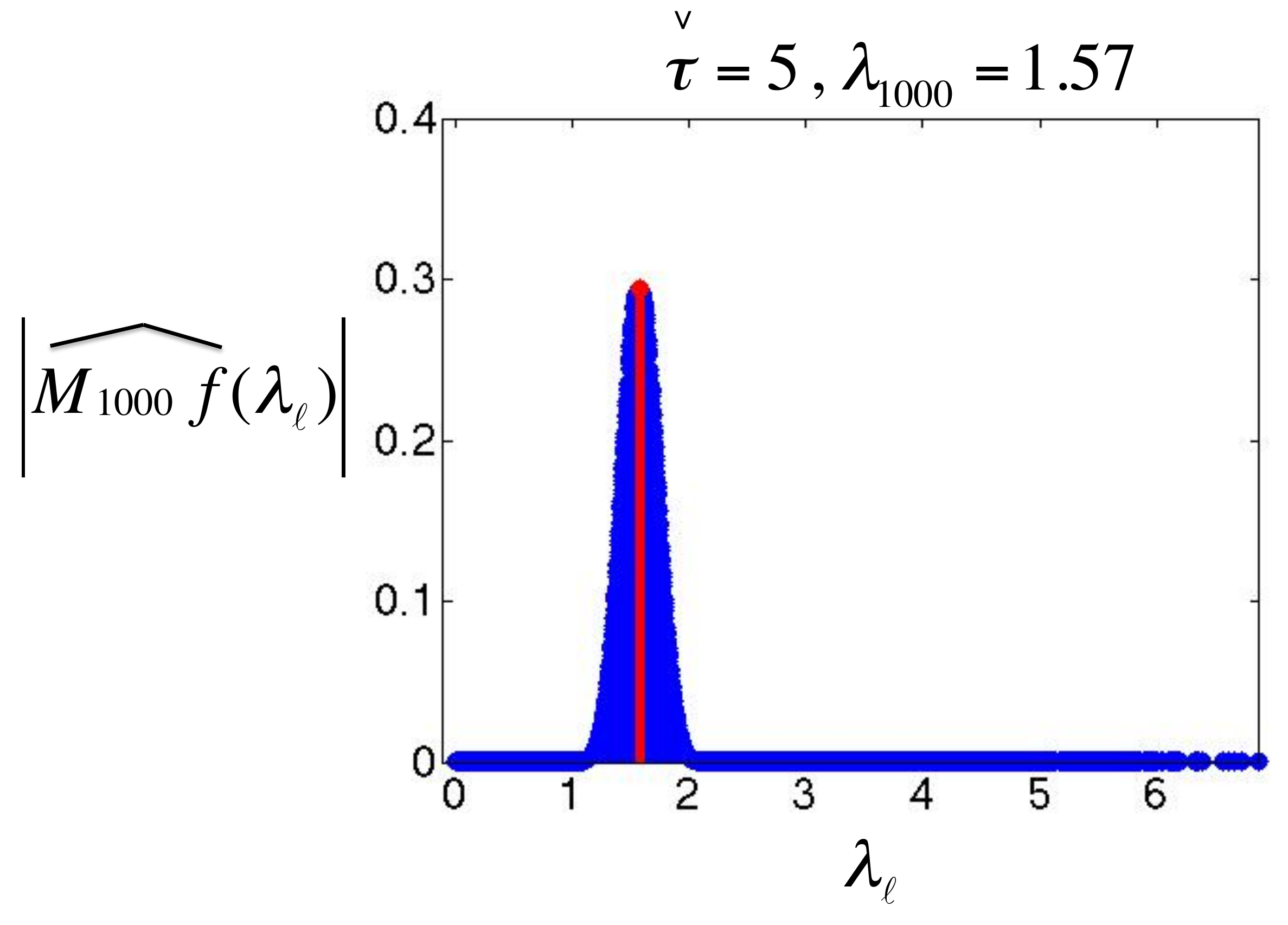}}
\centerline{\small{~~~~~~~~~~~~(b)}}
\end{minipage}
\hfill
\begin{minipage}[b]{.23\linewidth}
   \centering
   \centerline{\includegraphics[width=\linewidth]{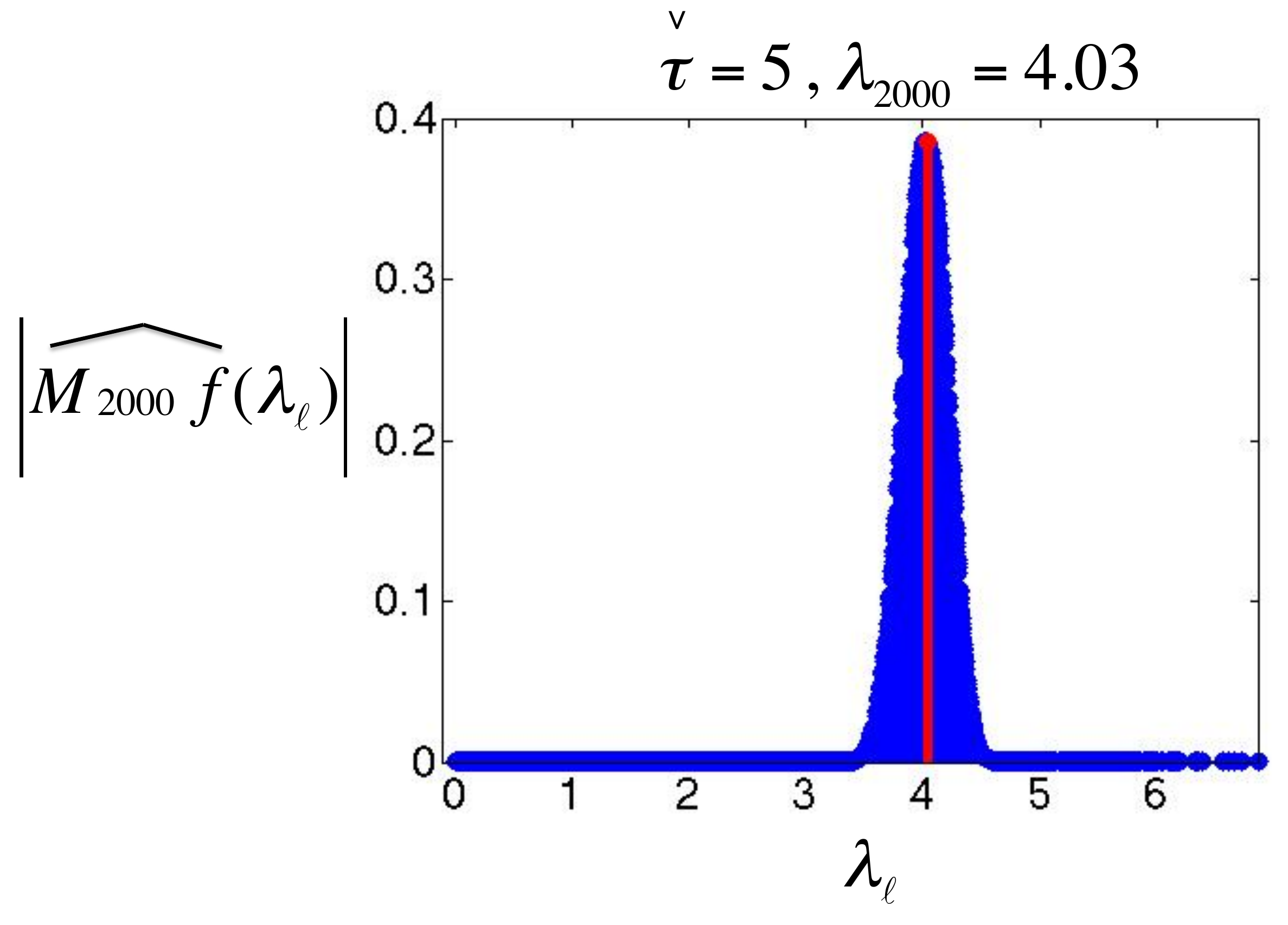}}
\centerline{\small{~~~~~~~~~~~~(c)}}
\end{minipage}
\hfill
\begin{minipage}[b]{.23\linewidth}
   \centering
   \centerline{\includegraphics[width=\linewidth]{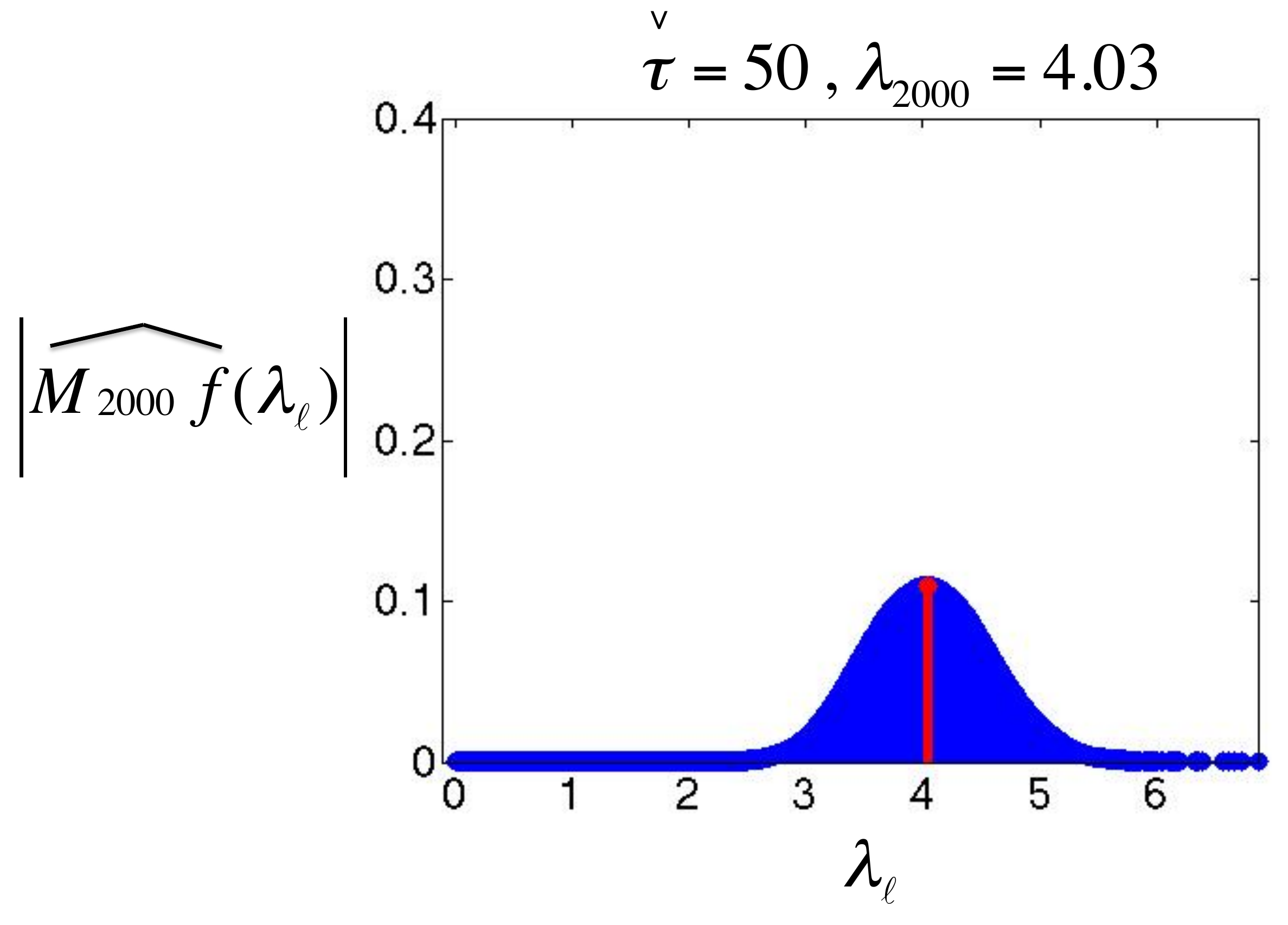}}
\centerline{\small{~~~~~~~~~~~~(d)}}
\end{minipage}
\hfill}
\caption {Generalized modulation on the Minnesota graph $\G$, via generalized translation on the same $\check{\G}$ from Figure \ref{Fig:mod_alt1}, with a kernel $\hat{\hat{f}}(\check{\lambda}_j)=e^{-\check{\tau}\check{\lambda}_j}$ defined on $\sigma(\check{\L})$. (a) We take an inverse graph Fourier transform $\hat{\hat{f}}(\check{\lambda}_j)$ to find $\hat{f}(\lambda_{\l})$. (b) and (c) We modulate the kernel to two different center frequencies, $\lambda_{1000}$ and $\lambda_{2000}$, by translating it on $\check{\G}$. In both cases, the modulated kernel is localized around the desired frequency of $\sigma(\L)$ (shown in red). (d) As $\check{\tau}$ \emph{increases}, the spread of the modulated kernel around the center frequency in the graph spectral domain of $\G$ generally increases.}
  \label{Fig:mod_alt2}
\end{figure}

\section{References}

\bibliography{bib_WGFT}
\bibliographystyle{elsarticle-num}

\end{document}